\documentclass[fleqn,11pt]{article}
\usepackage[active]{srcltx}
\usepackage{latexsym}
\usepackage{graphicx}
\usepackage{amsmath}
\usepackage{amssymb}
\usepackage{amsfonts}
\usepackage{epsfig}
\usepackage{bm}

\newtheorem{definition}{Definition}[section]

\newtheorem{theorem}{Theorem}[section]

\newtheorem{lemma}[theorem]{Lemma}
\newtheorem{observation}[theorem]{Observation}
\newtheorem{proposition}[theorem]{Proposition}

\newenvironment{proof}{{\em Proof.}}
       {\hspace*{\fill}$\Box$\par\vspace{4mm}}

\newcommand{\qed}{\hfill$\Box$\par\smallskip\noindent}

\newcommand{\comment}[1]{{}}

\begin{document}
\bibliographystyle{unsrt}

\title{On the Sequential Multiknapsack polytope}
\author{
Paolo Detti \thanks{Dipartimento di Ingegneria dell'Informazione,
Universit\`{a} di Siena, Via Roma 56, Italy, e-mail: detti@dii.unisi.it, tel. +39 0577-234850 (1022), fax +39 0577-233602.}}

\date{}
\maketitle

\begin{abstract}
The Sequential Multiple Knapsack Problem is a special case of Multiple knapsack problem in which the items sizes are divisible. A characterization of the optimal solutions of the problem and a description of the convex hull of all the integer solutions  are presented. More precisely, it is shown that a new formulation of the problem allows to generate a decomposition approach for enumerating all optimal solutions of the problem. Such a decomposition approach is used for finding  the inequalities (defined by an inductive scheme) describing the Sequential Multiple Knapsack polytope.

\end{abstract}

\section{Introduction}\label{sec:intro}
The Sequential Multiple Bounded Knapsack Problem 
(SMKP) can be stated as follows. There are a set $N$ of $n$ item types, $N=\{1, \ldots,n\}$, and a set $M$ of knapsacks, $M=\{1,
\ldots,m\}$. Each item of type $j$ has a size $s_j \in
\mathbb{Z}^+$, a value $v_j \in \mathbb{R}$ and an upper bound
$b_j \in \mathbb{Z}^+$. Item sizes are divisible, i.e., 
$s_{j+1}/s_j \in \mathbb{N}$, for all $j=1,\ldots, n-1$. 
Each knapsack $i$ has a
capacity $c_i \in \mathbb{Z}^+$. The problem is to find
the number $x_{i,j}$ of items of type $j$, for $j=1,\ldots, n$, to be assigned to each knapsack $i$, such that:
(1) The total value
$\sum_{i=1}^{m}\sum_{j=1}^{n}v_jx_{i,j}$ of the
assigned items is maximum;
(2) $\sum_{j=1}^{n}s_jx_{i,j}\leq c_i$ for $i=1,\ldots, m$
(i.e., the total size of items
assigned to a knapsack does not exceed the capacity of the knapsack);
(3) $\sum_{i=1}^{m}x_{i,j}\leq b_j$ for $j=1,\ldots, n$
(i.e., the total number of the assigned  items of type $j$ does not
exceed  the upper bound).
Without loss of generality we will assume $s_1=1$. The   sequential {\em single} knapsack problem (SKP) has been addressed in the literature
by several authors. For the unbounded case (i.e., $b_j = \infty$,
for all $j$), Marcotte \cite{Marcotte} presents  a linear time algorithm for SKP and
Pochet and Wolsey \cite{WolPoc} give an explicit description of the polytope.
For the bounded case (where $b_j < \infty$, for all $j$), in \cite{Pochet98}, a description of the bounded SKP polytope is provided, and in   \cite{ver_aarts}   an $O(n^2\log n)$ algorithm
  is proposed.    
In \cite{Detti}, the sequential {\em multiple} knapsack problem is addressed, and a polynomial $O(n^2 + nm)$ algorithm is presented.

In this paper, a new formulation for SMKP and a characterization of the optimal solutions of the  new formulation is provided, leading to a description of the convex hull of all the integer solutions of SMKP. 
More precisely, first new formulations are proposed for SMKP.
Some new formulations  restrict   the set of feasible solutions of the problem, but guarantee the  existence of at least an optimal solution in the restricted feasible set.
Then, a problem transformation from a new formulation to another problem is presented, such that SMKP and the transformed problem are equivalent in terms of optimization. This last result has been obtained  generalizing the approach proposed in \cite{Pochet98}, and is used for finding an decomposition approach that allows the enumeration of all the optimal solutions of the transformed problem. The decomposition approach is then used, as in  in \cite{Pochet98}, for finding a description of the SMKP polytope related to the new formulation.   

Summarizing, the main contributions of this paper are: the presentation of new ILP formulations for SMKP; the proposal of an approach for decomposing and enumerating the  optimal solutions of a transformed problem equivalent to SMKP in terms of optimization; the definition of inductive inequalities for describing the convex hull of all integer solutions of the new formulations proposed for SMKP. \\
In Section \ref{sec:newform}, basic properties of feasible solutions and  new formulations for SMKP are presented. In Section \ref{sec:trasnform}, a problem transformation is presented and in Section \ref{sec:enum} a decomposition scheme is given for the optimal solutions of the transformed problem. In Section \ref{sec:poly}, a description of the convex hull of the feasible solutions of SMKP is presented.

\section{A new formulation for SMKP}\label{sec:newform}
The two following proposition hold since the sizes of the items are divisible.

\begin{proposition}\label{prop:div}
Given a set $A \subseteq N$, let $s$ be the biggest size of the items in $A$. Let $c$ be an integer such that: $(1)$ $s$ divides $c$; $(2)$ $\sum_{j \in A}s_j \ge c$. Then a set $B \subseteq A$ exists such that $\sum_{j \in B}s_j = c$.
\end{proposition}

The following proposition directly follows from Proposition \ref{prop:div}. 

\begin{proposition}\label{prop:chunks}
Given a set $A \subseteq N$, let $s$ be the biggest size of the items in $A$, and let $b\ge s$ be an integer such that $s$ divides $b$. Then the minimum number of subsets, each of total size at most $b$, in which  $A$ can be partitioned is $\lceil \sum_{j \in A}s_j/b\rceil $. Moreover, a partition exists in which  $\lceil  \sum_{j \in A}s_j /b\rceil -1$ subsets have total size  $b$.
\end{proposition}

The following considerations  allow  to define a  new formulation for SMKP.
 Let $l$ be
the number of different item sizes in $N$, i.e.,
$l = |\{s_j | j = 1,\ldots ,n\}|$.
 Let these sizes be denoted by $d_1<d_2<\ldots <d_l$, and
let $n_1, n_2, \ldots, n_l$ be the number of {\em item types} with
sizes $d_1, d_2, \ldots, d_l$, respectively. Hence,
$n_k=|\{j | s_j=d_k\; \; j = 1,\ldots ,n\}|.$
Given a size $d_k$, we re-index the item types of size $d_k$, say
$\{j_1, j_2,\ldots, j_{n_k}\}$, in non-increasing order of values
$v$, i.e., in such a way that
$v_{j_1}\geq v_{j_2}\geq \ldots \geq v_{j_{n_k}},$
for all $k = 1,\ldots ,l$.\\
%
%
Let us consider the items with the
smallest size $d_1=1$,
and let $r_i^1 = (c_i \bmod
d_{2})$. Note that the effective capacity of knapsack $i$ that can
be used for all items of size larger than $d_1$ is $c_i-r_i^1$, for
$i=1,\ldots,m$. Hence, the capacities $r_i^1$, for $i=1,\ldots, m$,
of each knapsack can be {\em only} used  to assign items of size $d_1$. (Observe that
the capacity of knapsack $i$ that can be used  to assign only  items of size $d_1=1$ is $\lfloor r_i^1/ d_1\rfloor d_1=r_i^1$.
Let us consider the items with the size $d_2$. If $d_2<d_l$, let $r_i^2 = ((c_i-r_i^1) \bmod d_{3})$, otherwise let $r_i^2 = (c_i-r_i^1)$. Since, by definition, $d_2$ divides $c_i-r_i^1$ and $d_3$, it follows that $d_2$ divides $r_i^2$.
Note that, the capacities $r_i^2$, for $i=1,\ldots, m$, can be used only to assign items of sizes $d_1$ and $d_2$.  
The above argument can be repeated for items of sizes $d_3,\ldots,d_{l-1}$, by defining the capacities  $r_i^h$, for $h=3,\ldots,l-1$ and $i=1,\ldots,m$, and then setting $r_i^l=c_i -\sum\limits_{h=1}^{l-1}r_i^h$. By the above discussion, each knapsack $i$ can be partitioned into $l$ {\em parts} of capacities $r_i^1,\ldots,r_i^l$, such that only items not bigger than $d_h$ can be assigned to a part $h$, with $1\le h \le l$.
Given $j \in N$, we denote by $g(j)$, $1 \le g(j) \le l$, the index such that $d_{g(j)}=s_j$. Let $x_{i,j}^h$ be the number of items of type $j$ assigned to the part of knapsack $i$ (of capacity $r_i^h$), with $h\ge g(j)$. By the above discussion, it follows that SMKP can be also formulated as:
{\small
\begin{equation}\label{newform}
\begin{array}{rclr}
\max\sum\limits_{i=1}^{m}\sum\limits_{j=1}^{n}\sum\limits_{h=g(j)}^{l} v_j x_{i,j}^h&&\\
 \sum\limits_{j\in N: g(j) \le h} s_j x_{i,j}^h &\leq& r_i^h \text{ for } h=1,\ldots, l \text{ and } i=1,\ldots, m\\
 \sum\limits_{i=1}^{m} \sum\limits_{h: h\ge g(j)}x_{i,j}^h &\leq& b_j \text{ for } j=1,\ldots, n  \\
x_{i,j}^h& \in& \mathbb{Z}^+\cup\{0\} \;\text{ for } i =1,\ldots,m,\; j =1,\ldots, n \text{ and }  h
=g(j),\ldots, l.
\end{array}
\end{equation}
}
The objective function accounts for the maximization of the value of the assigned items. 
The first set of constraints states that the total size of the items $j$ such that $s_j \le d_h$ (or, equivalently $g(j) \le h$) assigned to the part $h$ of knapsack $i$ does not exceed the capacity $r_i^h$. The second set of constraints states that the total number of assigned items of type $j$ cannot exceed the upper bound  $b_j$.
From now on, we consider the SMKP formulation \eqref{newform}.

Given an instance of SMKP and an integer $h$, $1 \le h \le l$, we denote by SMKP($h$) the ``restricted" SMKP instance, in which the item set is $\{j\in N: s_j \le d_h\}$ and each knapsack $i$ has the reduced capacity $c_i=\sum\limits_{q=1}^{h} r_i^q$, for $i=1,\ldots,m$. A formulation for SMKP($h$) reads as follows
{\small
\begin{equation}\label{newformh}
\begin{array}{rclr}
\max\sum\limits_{i=1}^{m}\sum\limits_{j=1}^{n}\sum\limits_{q=g(j)}^{h} v_j x_{i,j}^q&&\\
 \sum\limits_{j\in N: g(j) \le q} s_j x_{i,j}^q &\leq& r_i^q \text{ for } q=1,\ldots, h \text{ and } i=1,\ldots, m\\
 \sum\limits_{i=1}^{m} \sum\limits_{q=g(j)}^h x_{i,j}^q &\leq& b_j \text{ for } j\in \{ N: s_j \le d_h\}   \\
x_{i,j}^q& \in& \mathbb{Z}^+\cup\{0\} \; \text{ for } i =1,\ldots, m, \;  j\in \{ N: s_j \le d_h\} \text{ and }  q
=g(j),\ldots, h.
\end{array}
\end{equation}
}

In \cite{Detti}, a polynomial algorithm, called A-OPT, for finding an optimal solution of SMKP is proposed. A-OPT is a recursive algorithm based on Formulation \eqref{newform}. The basic idea of the algorithm is reported in the following.
Let $\alpha_1$ be the maximum number of items of size $d_1$ that can be
assigned to the $m$ knapsacks, when each knapsack has a restricted
capacity $r_i^1$, $i=1,\ldots,m$. Since the total number of
items of size $d_1$ is $t_1=\sum\limits_{j=1: s_j=d_1}^{n}b_j$, we have
$\alpha_1=\min\{t_1,\sum\limits_{i=1}^{m} \lfloor {r_i \over
d_1}\rfloor\}$. Lemma \ref{lem:strip} holds  \cite{Detti}.
\begin{lemma}\label{lem:strip}
An optimal solution for SMKP exists in which the first $\alpha_1$ (i.e., with the biggest values)
items of size $d_1$ are assigned to the knapsacks using at most a
capacity $r_i^1$, for  $i=1,\ldots, m$.
\end{lemma}
According to Lemma \ref{lem:strip}, A-OPT recursively works as follows. First of all, the first $\alpha_1$ items of size $d_1$ are assigned to the knapsacks using at most a
capacity $r_i^1$ for each knapsack $i=1,\ldots, m$. The not assigned items of size $d_1$ are grouped  by a procedure, called {\em grouping procedure}, described in the following. The not assigned items of size $d_1$ are lined up  in non-increasing value order. Then,  groups of $d_2/d_1$ items are replaced by items with
size $d_2$. To each new item of size $d_2$, a value $v$ is
assigned, given by the sum of the values of the grouped items. (The
last new item built so far may not contain $d_2/d_1$ items of size
$d_1$, but we assign to it a size of $d_2$ and a value equal to
the sum of the values of the grouped items.) 
After the  grouping procedure, we get an instance with $l - 1$
different item sizes, where $d_1=d_2, d_2=d_3,\ldots,d_{l-1}=d_l$,
and $m$ knapsacks of capacity $c_i-r_i^1$, for $i=1,\ldots,m$. 
A-OPT can be recursively called on this instance. 
  
Notation that will be used in the following is now introduced. For $a \in  \mathbb{R}$, we denote $a^+=\max\{0, a\}$.
Given a feasible solution $x$ of SMKP and two integers $1 \le h_1\le h_2\le l$,
let $x^{h_1,h_2}$ be the "partial" solution of $x$ related to the  parts of the knapsacks $\{h_1,\ldots,h_2\}$, i.e.,  $x^{h_1,h_2}$ is the vector containing  the components $x^{q}_{i,j}$, for $q=h_1,\ldots,h_2$, $i=1,\ldots,m$ and $j\in \{N: d_j \le d_{h_2}\}$. If $h_1>h_2$ then $y^{h_1,h_2}$ does not assign any item. To simplify the notation, if
$h_1=h_2$, $x^{h_1}$ is used instead of $x^{h_1,h_1}$. Let $S(x)$ and $S(x^{h_1,h_2})$ be the set of items assigned in $x$ and $x^{h_1,h_2}$ respectively. Given an item type $j$,  let $x_j^{h_1,h_2}$ be the number of
items of type $j$ assigned in $x$ to the parts $\{h_1,\ldots,h_2\}$ of knapsacks, and let  $S(x_j^{h_1,h_2})$ be the corresponding set of items. Moreover, given a subset of items $A$, let $f(A)$  and $v(A)$ be the total size and value, respectively,  of the items in $A$, i.e.,
$f(A)=\sum\limits_{j \in A} s_j |A_j|$ and $v(A)=\sum\limits_{j \in A} v_j |A_j|$,  where $|A_j|$ denotes the number of items of type $j$ in $A$. For the sake of simplicity, $f(x^{h_1,h_2})$ ($v(x^{h_1,h_2})$) can be used instead of $f(S(x^{h_1,h_2}))$ ($v(S(x^{h_1,h_2}))$) to denote the total size (the total value) of the items assigned by $x$ to the knapsacks parts $\{h_1,\ldots,h_2\}$. 

In the remaining part of this section, new formulations of SMKP are proposed based on restrictions of  the set of feasible solutions of Problem \eqref{newform}. Such restrictions guarantee the  existence of at least an optimal solution of the SMKP in the restricted sets.

\subsection{Definition of OPT solution}\label{sec:OPT}
In what follows, the definition of {\em OPT solution} is given.
\begin{definition}\label{def:OPTprop}
Let $x$ be  a feasible solution for SMKP. The solution $x$ has the {\em OPT property} (or, equivalently, $x$ is an {\em OPT solution}), 
if for each feasible solution $\bar x$ that assigns the same items of $x$, i.e.,   $S(x)=S(\bar x)$, we have $v(x^{1,h}) \ge v(\bar x^{1,h})$, 
 for $h=1,\ldots,l$. 
 \end{definition}
 

The following observation holds since algorithm A-OPT assigns to the  $h$-th part of the knapsack, at the $h$-th iteration, as many (grouped) items of size $d_h$ as possible (according to their order).
\begin{observation}\label{A-OPT}
A-OPT produces optimal OPT solutions. 
\end{observation}

In view of Definition \ref{def:OPTprop} and Observation \ref{A-OPT}, it follows that for each feasible solution $x$, the solution $\bar x$, obtained by applying  A-OPT on the restricted item set $S(x)$, is an OPT solution. Hence, given an SMKP instance, we can restrict to consider only  OPT solutions and reformulate SMKP as follows.  Let $P^{OPT}$ be the convex hull of the  OPT solutions
for SMKP, then SMKP can be also formulated as 
$\max \{ v(x): x \in P^{OPT}\}$.\\
Given an SMKP instance, let $O^{OPT}$ be the set of the  OPT solutions that are optimal. 
Given an optimal OPT solution $x$ for SMKP, in the following (Lemma \ref{lem:opt}) we show that  $x^{1,h}$ is optimal for SMKP($h$), for $h=1,\ldots,l$. At this aim Lemma \ref{lem:opt_aux} is useful.
\begin{lemma}\label{lem:opt_aux}
Given an SMKP instance, let $x$ be  an optimal solution for SMKP.
Let $\bar x$ be the solution produced by A-OPT when applied to the set $S(x)$. Let $x_{O}$ be the optimal solution found by A-OPT when applied to the whole item set $N$. Then, if during A-OPT  the (grouped) items with the same size and value are suitably ordered,  $S(x_{O}^{1,h})=S(\bar x^{1,h})$, for $h=1,\ldots,l$.
\end{lemma}
\begin{proof}
The proof is reported in the Appendix.
\end{proof}

\begin{lemma}\label{lem:opt}
Let $x$ be  an optimal  solution for SMKP satisfying the OPT property.
Then $ x^{1,h}$ is optimal for SMKP($h$), for $h=1,\ldots,l$.
\end{lemma}
\begin{proof}
The thesis directly follows from Obervation \ref{A-OPT} and Lemma \ref{lem:opt_aux}.
\end{proof}


\subsection{Definition of ordered solution}\label{sec:ordered}
In this section, the definition of {\em ordered solution} is given. 

\begin{definition}\label{def:ordered}
Given a feasible solution $x$ for SMKP,  $x$ is called ordered if, for each knapsack part $h=1,\ldots,l$, a set of items $\Gamma$  assigned to the part $h$ in $x$ with size  $s(\Gamma)\le d_h$ exists only if an item of size $s(\Gamma)$ and value $v(\Gamma)$ either assigned to a part bigger than $h$ or not assigned in $x$ does not exist.
\end{definition}

Given a feasible solution $x$  for SMKP, it is always possible to build an ordered solution $\bar x$ such that $v(x^{1,h})=v(\bar x^{1,h})$ and $s(x^{1,h})=s(\bar x^{1,h})$,  for $h=1,\ldots,l$. In fact, let us suppose that a set $\Gamma$, with  $|\Gamma|>1$ and $f(\Gamma)\le d_h$, is assigned to a part $h$ in $x$ and that an item $j$ exists, with  $s_j=s(\Gamma)$ and value $v_j=v(\Gamma)$, either assigned to a part bigger than $h$ or not assigned in $x$. In the first case, let $\bar x$ be  the solution obtained by swapping $\Gamma$ and $j$ in $x$,  while in the second case let $\bar x$ be  the solution obtained by assigning $j$ in place of the items in $\Gamma$. If $\bar x$ contains a new set $\Gamma$ as defined above, then the above argument can repeated, otherwise, $\bar x$ is an ordered solution.

\subsection{A new formulation for SMKP}\label{sec:newnewfor}
By the discussions presented in Sections \ref{sec:OPT} and \ref{sec:ordered}, it follows that we can restrict to consider only OPT and ordered solutions. 
Let  $P^{OO}$ be the convex hull  of  ordered and OPT solutions  for SMKP (formulated as in \eqref{newform}), and let $O^{OO}$  be the set of
 optimal ordered and OPT solutions.  Then, SMKP can be formulated as the problem of finding a solution in  $O^{OO}$, or equivalently, as
{\small
\begin{equation}\label{formOO}
\max \{ v(x): x \in P^{OO}\}
\end{equation}
In the rest of the paper we will show how to find a description of polytope $P^{OO}$. It is important to observe that, by definition of OPT and ordered solution, $P^{OO}$ directly depends from the values of the objective coefficients $v_1, \ldots, v_n$.

\section{A problem transformation}\label{sec:trasnform}

In this section, a transformation is presented from SMKP, formulated as in \eqref{newform}, to a new
problem, called Modified-$SMKP$ (M-$SP$). This transformation
extends that presented in \cite{Pochet98} (for the sequential
knapsack problem) and uses the notion of {\em block} introduced in
\cite{Pochet98}.
\begin{definition}\label{def:block}
Given an MSKP  instance, let $B_w=\{w_1,\ldots, w_k\}$, $w_1<\ldots < w_k$, be a subset of
item types. $B_w$ is called a block if, for every $j\in \{2,\ldots, k\}$, $s_{w_j}\leq s_{w_1} + \sum\limits_{v=1}^{j-1}b_{w_v}s_{w_v}$. The number
$\tilde b_w=\sum\limits_{v=1}^{k}b_{w_v}s_{w_v}/s_{w_1}$ is called the
{\em multiplicity} of block $B_w$.
\end{definition}
The following property holds.
\begin{proposition}\label{prop:block}
Given a block $B_w=\{w_1,\ldots, w_k\}$,
let\\ $\delta^h_q \in\{0,s_{w_1}, 2s_{w_1},
\ldots,b_{w_1}s_{w_1},(b_{w_1}+1)s_{w_1}, \ldots,b_{w_1}s_{w_1}+b_{w_2}s_{w_2}, \ldots,\sum\limits_{v=1}^{k}b_{w_v}s_{w_v}\},$ 
 for $q=1,\ldots, l$ and $h=q,\ldots,l$, such that: 
\begin{equation}\label{eq:block1bis}
\sum\limits_{h=q}^{l}
 \lceil \delta^h_q/d_q\rceil  \leq \sum\limits_{w \in B_w: g(w)=q}  b_{w}\;\;\;\; \;\; \text{for } \;\;  q=1,\ldots, l\text{ and } h=q,\ldots,l.\end{equation}

Then  a subset $R
\subseteq B_w$ and not negative integers $\lambda^h_{j}$, for 
$j \in R$ and $h=g(j),\ldots, l$, exist such that: 
\begin{equation}\label{eq:block1}
0\le
\sum\limits_{h=g(j)}^{l}\lambda^h_{j} \leq  b_{j}, \text{ for all }j \in
R;
\end{equation}
\begin{equation}\label{eq:block2}
\sum\limits_{j
\in R:\\ g(j)=q}\lambda^h_{j} = \lceil \delta^h_q/d_q\rceil, \text{ for } q=1,\ldots, l \text{ and } h=q,\ldots,l.
\end{equation}
\end{proposition}
\begin{proof}
The proof is by induction on the number of different item sizes $l$. Suppose that $l=1$. Then, by Definition \ref{def:block}, all the items have size $d_1=s_{w_1}=1$ and belong to the same block $B_w$. 
Let $j'$ be the smallest index such that $\sum\limits_{j\in B_w : j \le j'} b_{j} \ge
 \lceil \delta^1_1/d_1\rceil=\delta^1_1$. Then, since $\delta^1_1 \in\{0,s_{w_1}, 2s_{w_1},\ldots,b_{w_1}s_{w_1}\}$ and $d_1=s_{w_1}=1$,
the thesis follows by setting $R=\{w_1,\ldots,j'\}$, $\lambda^1_j=b_j$ if $j<j'$ and $\lambda^1_{j'}=   \delta^1_1 - \sum\limits_{j\in R: j< j'}\lambda^1_j$.\\ 
Let us suppose that the thesis holds up to $l=a-1\ge 1$ and show it for $l=a$.\\
Given a block $B_w=\{w_1,\ldots, w_k\}$, let $\delta^h_q \in\{0,s_{w_1}, 2s_{w_1},
\ldots,b_{w_1}s_{w_1}, \ldots,\sum\limits_{v=1}^{k}b_{w_v}s_{w_v}\},$
 for $q=1,\ldots, a$ and $h=q,\ldots,a$, such that: 
\begin{equation}\label{eqn:propblock}
\sum\limits_{h=q}^{a-1}
 \lceil \delta^h_q/d_q\rceil +  \lceil \delta^a_q/d_q\rceil \leq \sum\limits_{w \in B_w: g(w)=q}  b_{w}\;\;\;\; \;\; \text{for } \;\;\;\;\;\;  q=1,\ldots, a.
\end{equation}
Observe that, by \eqref{eqn:propblock}, we also have
\begin{equation}\label{eq:block1bbis}
\sum\limits_{h=q}^{a-1}
 \lceil \delta^h_q/d_q\rceil  \leq \sum\limits_{w \in B_w: g(w)=q}  b_{w} \;\;\;\; \;\; \text{for } \;\;\;\;\;\;  q=1,\ldots, a-1.\end{equation}
 Let $A=\sum\limits_{w \in B_w: g(w)<a}b_{w}s_w$. By \eqref{eq:block1bbis}, we have $\sum\limits_{h=q}^{a-1}
 \lceil \delta^h_q/d_q\rceil d_q \leq \sum\limits_{w \in B_w: g(w)=q}  b_{w}d_q \le A \text{    for }  q=1,\ldots, a-1.$
Hence,  $\delta^h_q \le A$,  i.e.,
$\delta^h_q \in\{0,s_{w_1}, 2s_{w_1},
\ldots,b_{w_1}s_{w_1}, \ldots,\sum\limits_{v: g(v)<a }b_{w_v}s_{w_v}\}$, for  $q=1,\ldots, a-1$ and $h=q,\ldots, a-1$.
Hence, the values $\delta^h_q$,  for   $q=1,\ldots, a-1$, satisfy the hypothesis of the proposition, and by induction, a subset $\bar R\subseteq B_w$, such that $s_j <d_a$ for all $j \in \bar R$, and integers $ \lambda^h_{j}$, for $j \in \bar R$ and $h=g(j),\ldots, a-1$, exist such that   $0\le \sum\limits_{h=g(j)}^{a-1} \lambda^h_{j} \leq  b_{j}$, for all $j \in
\bar R$, and $\sum\limits_{j \in \bar R: g(j)=q}\lambda^h_{j} = \lceil \delta^h_q/d_q\rceil $, for $q=1,\ldots, a-1$ and $h=q,\ldots, a-1$.
Hence, Inequality \eqref{eqn:propblock} reads as
$$\sum\limits_{j\in \bar R: g(j)=q} \sum\limits_{h=q}^{a-1}
\lambda^h_j +  \lceil \delta^a_q/d_q\rceil \leq \sum\limits_{j \in \bar R: g(j)=q}  b_{j}+ \sum\limits_{w \in B_w\setminus \bar R: g(w)=q}  b_{w}\;\;\;\; \;\; \text{for } \;\;\;\;\;\;  q=1,\ldots, a,$$
that can be rewritten as
\begin{equation}\label{eq:indBlock}
  \lceil \delta^a_q/d_q\rceil \leq \sum\limits_{j \in \bar R: g(j)=q}  \left( b_{j}- \sum\limits_{h=q}^{a-1}
\lambda^h_j\right)+ \sum\limits_{w \in B_w\setminus \bar R: g(w)=q}  b_{w}\;\;\;\; \;\; \text{for } \;\;\;\;\;\;  q=1,\ldots, a.\end{equation}
For each $q=1,\ldots,a$, let $R^q$ be a subset of items in $B_w$ of size $d_q$ with minimum total size such that:\\ $(i)$  
$ |R^q_{j}| \le  b_{j}- \sum\limits_{h=q}^{a-1} \lambda^h_j$, for all $j \in \bar R\cap R^q$, and $ |R^q_{j}| \le  b_{j}$, for all $j \in R^q \setminus \bar R$ (recall that  $R^q_j$ is the set of items of type $j$ in $R^q$);\\ 
$(ii)$ $\sum\limits_{j\in R^q}   |R^q_{j}| =  \lceil \delta^a_q/d_q\rceil.$\\
Observe that, by \eqref{eq:indBlock}, the sets $R^1,\ldots,R^a$ exist and they are disjoints by definition. Then, the thesis follows by setting $R=\bar R\cup \{R^1\cup \ldots \cup R^a\}$, and
$\lambda^a_j = |R^{g(j)}_{j}|$ 
for all $j \in  \{R^1\cup \ldots \cup R^a\}$.
\end{proof}


Given an SMKP instance, let $B_1,\ldots, B_t$ be a partition of $N$
into blocks. Under the block partition $B_1,\ldots, B_t$, an
instance of the new problem M-$SP$ is obtained from an instance of SMKP by modifying
the item and the knapsack sets as explained in the following. All the item types belonging to a block
$B_w=\{w_1,\ldots, w_k\}$ are replaced by items of size $s_{w_1}$
and profit $v_{w_1}$, for $w=1,\ldots, t$. 
More precisely, an item
 $w_v \in B_w$, of size $s_{w_v}$, upper bound $b_{w_v}$,
value $v_{w_v}$ is replaced by
$b_{w_v}s_{w_v}/s_{w_1}$ items of type $w$ of size $f_w=s_{w_1}$ and value
$p_w=v_{w_1}$. Such new items are denoted as items of type $w$.
Let $N'$ be the set of all the new items produced by applying the above
replacing procedure to all the blocks. Note that, $N'$ can be partitioned into sets $T_1,\ldots, T_l$, where $T_q$, $1\le q \le l$, contains all the new items obtained by items in $N$ of size $d_q$. 
Items in $N'$ of different sizes may belong to the same
set $T_q$, since they may belong to different blocks.
In what follows, let  $\tilde b_{w,q}$ be the total number of items of type $w$ belonging to $T_q$, in M-SP, and let $\tilde b_{w}=\sum\limits_{q=1}^{l}\tilde b_{w,q}$.
By definition, the following observation holds.
\begin{observation}\label{obs:T_q} 
The 
number of items of type $w$ belonging to a set $T_q$,  $\tilde b_{w,q}$, is multiple of   $\frac{d_q}{f_w}$ and is equal to
\begin{equation}\label{eq:tilde bb}
\tilde b_{w,q}= \sum\limits_{w \in B_w: f(w)=d_q}  \frac{b_{w} d_q}{f_w}.
\end{equation}
\end{observation}



Furthermore, in M-$SP$, the $m$ knapsacks are replaced by a {\em single} knapsack composed of $l$ parts. To each part $h$  a capacity $\bar c_h=\sum\limits_{i=1}^m r_i^h$ is associated, for $h=1,\ldots,l$. Items of the set $T_q$ can be only assigned to the parts $q,q+1,\ldots,l$.\\
M-$SP$ is the problem of finding an assignment of the items in $N'$ that maximizes the total profit of the assigned items and that satisfies given conditions, described in the following. 
Let $y_{w,q}^h$ be an integer
variable denoting the number of items of type $w$, belonging to the set $T_q$ and assigned to the part $h$ of the knapsack, for $w=1,\ldots,t$, $q=1,\ldots,l$ and $h=1,\ldots,l$.
A solution $y$ is feasible for M-$SP$ if the following constraints are satisfied:
\begin{eqnarray}
\sum\limits_{w=1}^{t}\sum\limits_{q=1}^{h} \left\lceil f_w y_{w,q}^h/d_q\right\rceil d_q &\leq& \bar c_h \text{ for } h=1,\ldots, l \label{prob:M-SP-cond1}\\
\sum\limits_{h=q}^{l}\left\lceil f_w y_{w,q}^h/d_q \right\rceil&\leq&  \sum\limits_{w \in B_w: f(w)=d_q}  b_{w} = f_w \tilde b_{w,q}/d_q \;\;\;\;\;\;\text{ for } w
=1,\ldots, t \text{ and } q =1,\ldots, l \label{prob:M-SP-cond2}
\end{eqnarray}
where the last equality  in \eqref{prob:M-SP-cond2} follows from \eqref{eq:tilde bb}.

In the above constraints, $\left\lceil f_w y_{w,q}^h/d_q \right\rceil d_q$ represents  the ``occupancy" of the  items of type $w$ in $ T_q$ assigned to the part $h$ of the knapsack, in terms of chunks of size $d_q$. 
While $\left\lceil f_w y_{w,q}^h/d_q\right\rceil$ is the minimum number of chunks of size at most $d_q$ that can be obtained using 
the  items of type $w$ in $ T_q$, that are assigned to the part $h$ of the knapsack.
Hence, the first set of constraints state that the total ``occupancy" in the part $h$ cannot exceed the capacity $\bar c_h$. Observe that, only  items in $T_q$ with $q \le h$ can be assigned to the part $h$.
constraints \eqref{prob:M-SP-cond2} limit the total number of chunks of size at most $d_q$  that can be obtained using 
the items of type $w$ in $T_q$.

A formulation for M-$SP$ is reported in the following.
{\small
\begin{eqnarray*}
\max\sum\limits_{h=1}^{l}\sum\limits_{q=1}^{h}\sum\limits_{w=1}^{t}p_w y_{w,q}^h&& \label{prob:M-SP}\\
\text{subject to}&&\\
\text{constraints } \eqref{prob:M-SP-cond1}  \text{ and } \eqref{prob:M-SP-cond2}  \; 
y_{w,q}^h& \in&\mathbb{Z}^+\cup\{0\} \text{ for } w =1,\ldots, t, \; q =1,\ldots, l \text{ and } h =q,\ldots, l. \label{prob:M-SP4}
\end{eqnarray*}
}
Note that, in M-SP, the item sizes are divisible, too.
Pochet and Weismantel \cite{Pochet98} introduced a special partition of $N$ into blocks, called {\em maximal block partition} and showed its uniqueness. In this partition, a block  contains only items $j\in N$ with the same {\em gain per unit} $\frac{v_j}{s_j}$ and has the following property:  given two blocks $B_1$ and $B_2$ containing items with the same gain per unit, the set $B_1\cup B_2$ is not a block. 
In the rest of the paper, we defines  M-$SP$ on the maximal block partition and assume the items  in M-$SP$ ordered
in such a way that: $f_1 \leq \ldots \leq f_t$ and that $p_w \ge
p_{w+1}$, if $f_w = f_{w+1}$.
By the procedure for the construction of the maximal blocks, the following relation holds \cite{Pochet98}:
\begin{equation}\label{proper:maxblock}
f_w > \sum\limits_{u \in \{1,\ldots,w-1\}: \frac{p_w}{f_w}=\frac{p_u}{f_u}} f_u \sum\limits_{q=1}^{l}\tilde b_{u,q}=
\sum\limits_{u \in \{1,\ldots,w-1\}: \frac{p_w}{f_w}=\frac{p_u}{f_u}} f_u \tilde b_{u} \;\;\; \;{\text  for  }  \;w=1,\ldots, t.
\end{equation}
A notation similar to that introduced for SMKP is now introduced for M-SP.
Given a feasible solution $y$ of M-$SP$ and two integers $1 \le h_1\le h_2\le l$,
let $y^{h_1,h_2}$ be the "partial" solution of $y$ related to the  parts of the knapsack $\{h_1,\ldots,h_2\}$, i.e.,  $y^{h_1,h_2}$ is the vector containing  the components $y^{h}_{w,q}$, for $w=1, \ldots, t$, $h=h_1,\ldots,h_2$ and$q=1,\ldots,h_2$. If $h_1>h_2$, then $y^{h_1,h_2}$ does not assign any item. To simplify the notation, if
$h_1=h_2$, $y^{h_1}$ is used instead of $y^{h_1,h_1}$. Let $S(y)$ and $S(y^{h_1,h_2})$ be the set of items assigned in $y$ and $y^{h_1,h_2}$ respectively. Given a block $w$ and a set $T_q$,  let $y_{w,q}^{h_1,h_2}$ be the number of
items of type $w$ belonging to the set $T_q$ assigned, in $y$, to the parts $\{h_1,\ldots,h_2\}$ of the knapsack, and let  $S(y_{w,q}^{h_1,h_2})$ be the corresponding set of items. Moreover, given a set of items $A\subseteq N'$, let $f(A)$  and $p(A)$ be the total size (called in the following ``total size", too) and value, respectively,  of the items in $A$, i.e.,
$f(A)=\sum\limits_{w \in A} f_w |A_w|$ and $p(A)=\sum\limits_{w \in A} p_w |A_w|$,  where $|A_w|$ denotes the number of items of type $w$ in $A$. For the sake of simplicity, $f(y^{h_1,h_2})$ ($p(y^{h_1,h_2})$) is used instead of $f(S(x^{h_1,h_2}))$ ($p(S(x^{h_1,h_2}))$) to denote the total size (the total profit) of the items assigned by $y$ to the knapsacks parts $\{h_1,\ldots,h_2\}$. 

\subsection{Correspondance between feasible solutions of SMKP and M-SP}\label{sec:corr}
Given an instance of M-$SP$ and an integer $h$, $1 \le h \le l$, we denote by M-SP($h$) the ``restricted" M-$SP$ instance, obtained from the instance SMKP($h$), 
as defined in \eqref{newformh}.
Observe that, by Definition \ref{def:block}, if a set $B \subseteq N$ is block in SMKP, then the set $ \{j \in B: s_j \le d_h\}$ is a block in SMKP($h$), too.
We show now that there is a one to many correspondence between feasible solutions of M-$SP$ (M-SP($h$)) and 
feasible solutions of SMKP (SMKP($h$)). 

Given a M-$SP$ instance, corresponding to the maximal block partition of $N$, let $y \in \mathbb{Z}^{t
\times l\times l}$ be a feasible solution. 
For each block $B_w$, let $\delta_{wq}^h=f_w y_{w,q}^h$, for $q=1 , \ldots, l$ and $h=q , \ldots, l$.
By Proposition
\ref{prop:block}, a subset $R_w \subseteq B_w$ and not negative
integers $\lambda^h_{j}$, for $j \in R_w$, exist such that 
conditions \eqref{eq:block1} and \eqref{eq:block2} are satisfied.\\
A feasible solution of SMKP is built as follows. For each block $B_w$, let $\sum\limits_{i=1}^m x^h_{i,j}=\lambda^h_{j}$, if $j \in R_w$, for $h=1,\ldots,l$. By condition \eqref{eq:block1}, it follows that $  \sum\limits_{h= g(j)}^l \sum\limits_{i=1}^m x^h_{i,j}=\sum\limits_{h=g(j)}^{l}\lambda^h_{j} \le b_j$, for $j \in B_w$ and $w=1,\ldots,t$ (i.e., $x$ satisfies the second set of constraints in \eqref{newform}). 
Let $R=\bigcup\limits_{w=1}^t R_w$. Since $\delta_{wq}^h=f_w y_{w,q}^h$, and by conditions \eqref{eq:block2} and \eqref{prob:M-SP-cond1}, we have:\\
$\sum\limits_{w=1}^{t}\sum\limits_{q=1}^{h} \left\lceil \frac{f_w y_{w,q}^h}{d_q}\right\rceil d_q =
\sum\limits_{w=1}^{t}\sum\limits_{q=1}^{h} \left\lceil \frac{\delta_{wq}^h}{d_q}\right\rceil d_q
=\sum\limits_{q=1}^{h} \sum\limits_{j
\in R:\\ g(j)=q}\lambda^h_{j}  d_q =
\sum\limits_{j
\in R:\\ g(j)\le h} \lambda^h_{j}  s_j
\leq \bar c_h= \sum\limits_{i=1}^m r_i^h\;\;  \text{ for } h=1,\ldots, l.$
Hence, since $\sum\limits_{i=1}^m x^h_{i,j}=\lambda^h_{j}$, the following inequality holds
\begin{equation}\label{eq:ytox}
\sum\limits_{j
\in R:\\ g(j)\le h}\sum\limits_{i=1}^m x^h_{i,j}  s_j
\leq \bar c_h= \sum\limits_{i=1}^m r_i^h\;\;  \text{ for } h=1,\ldots, l.
\end{equation}
Recall that the  items in $R$ that can be assigned to a part $h$ must have sizes not greater than $d_h$. By \eqref{eq:ytox}, since the item sizes are divisible and since $d_h$ divides $r_{i}^h$, for $i=1,\dots,m$, it is always possible to partition the set $\{j \in R: g(j) \le h\}$ 
into $m$ subsets, say $\Omega_1^h, \ldots, \Omega_m^h$,
such that $ \sum\limits_{j \in \Omega_i^h}    s_j x^h_{i,j} \le  r_{i}^h$, for $i=1,\dots,m$ (i.e., $x$ satisfies the first set of constraints in \eqref{newform}). Note that, several of such partitions may exist. Hence, $x$ is feasible for SMKP. We show now that $v(x) \ge p(y)$. In fact, since  $\sum\limits_{j
\in R_w:\\ g(j)=q}\lambda^h_{j}= \lceil \delta_{w,q}^h /d_q \rceil  = \lceil f_w y_{w,q}^h /d_q \rceil $ (condition \eqref{eq:block2}), then $\sum\limits_{j
\in R_w:\\ g(j)=q}\lambda^h_{j} \ge  f_w y_{w,q}^h/d_q$, and 
$\sum\limits_{j \in R_w:\\ g(j)\le h}\lambda^h_{j} =\sum\limits_{q=1}^h\sum\limits_{j \in R_w:\\ g(j)= q}\lambda^h_{j} \ge   \sum\limits_{q=1}^h f_w y_{w,q}^h/d_q$.
 By the partition into maximal blocks, for each $j \in B_w$ we have $v_j/s_j= v_j/d_{g(j)} = p_w/ f_w$, i.e., $v_j= p_w d_{g(j)} / f_w$. Hence, recalling that $R=\bigcup\limits_{w=1}^t R_w$, we have
$  
v(x)=\sum\limits_{h=1}^l \sum\limits_{j \in R:\\ g(j)\le h} v_j \sum\limits_{i=1}^m x^h_{i,j} =  \sum\limits_{h=1}^l \sum\limits_{j
\in R:\\ g(j)\le h} v_j \lambda^h_{j} \ge \sum\limits_{w=1}^t \sum\limits_{h=1}^l \sum\limits_{q=1}^h  v_j f_w y_{w,q}^h/d_q =p (y)$.

\medskip
On the other hand, given a feasible solution of SMKP, $x \in
\mathbb{Z}^{n \times m\times l}$ and a partition into maximal blocks, we build a corresponding feasible solution of
M-$SP$ by setting 
$y_{w,q}^h = \sum\limits_{j \in B_w: g(j)=q}
(d_q/f_w) \sum\limits_{i=1}^m x_{i,j}^h$. First, observe that, since $f_w y_{w,q}^h/d_q = \sum\limits_{j \in B_w: g(j)=q}
 \sum\limits_{i=1}^m x_{i,j}^h$, then $f_w y_{w,q}^h/d_q$ is integer.\\
The sum of the first set of constraints of \eqref{newform} with respect to index $i$ gives\\
$\sum\limits_{j \in N: s_j \le d_h} \sum\limits_{i=1}^{m}s_jx_{i,j}^h=\sum\limits_{w=1}^{t}\sum\limits_{q=1}^{h}f_w
y_{w,q}^h= \sum\limits_{w=1}^{t}\sum\limits_{q=1}^{h} \left\lceil \frac{f_w y_{w,q}^h}{d_q}\right \rceil d_q=\le \sum\limits_{i=1}^{m} r_i^h=\bar c_h$ (i.e., $y$ satisfies constraints   \eqref{prob:M-SP-cond1}).

The sum of the second set of constraints of \eqref{newform} with respect to the items $j$ in a given block $B_w$  gives\\
$\sum\limits_{j \in B_w:\\ g(j)=q} \sum\limits_{h: h\ge q} \sum\limits_{i=1}^m x_{i,j}^h= \sum\limits_{h: h\ge q} f_w y_{w,q}^h/d_q= \sum\limits_{h: h\ge q} \left\lceil \frac{f_w y_{w,q}^h}{d_q}\right \rceil \le \sum\limits_{j \in B_w: g(j)=q} b_j$, for $q=1,\ldots,l$ and $w=1,\ldots,t$ (i.e., $y$ satisfies constraints   \eqref{prob:M-SP-cond2}).\\ 
Finally, recalling that by the partition into maximal blocks it follows that $v_j= p_w d_{g(j)} / f_w$,
then $p_w y_{w,q}^h = \sum\limits_{j \in B_w: g(j)=q}\sum\limits_{i=1}^m  v_j x_{i,j}^h$, i.e., $p(y)=v(x)$  
 
 It is straightforward to observe that, all the above arguments can be also used to state the correspondence between  feasible solutions of SMKP($h$) and M-SP($h$), for $h=1,\ldots,l$. 
  The two following propositions hold.
 \begin{proposition}\label{prop:corr}
For a given knapsack's part $h$, a  many to one correspondence exists  between feasible solutions $x^{1,h}$ of SMKP($h$) and 
feasible solutions $y^{1,h}$ of M-SP($h$), and it holds that $v(x^{1,h}) \ge p(y^{1,h})$.
\end{proposition}

\begin{proposition}\label{prop:corr2}
Given a feasible solution $x^{1,h}$ of SMKP($h$), let $y^{1,h}$ be the corresponding  feasible solution of M-SP($h$). Then $v(x^{1,h}) = p(y^{1,h})$.
\end{proposition}

\subsection{Valid inequalities for MSKP and M-SP}
Suppose that the following inequality 
\begin{equation}\label{eq:valid1}
\sum\limits_{q=1}^{l} \sum\limits_{w=1}^{t} \nu_{w,q}   \sum\limits_{h=q}^{l}y^h_{w,q}\le \nu_0.
\end{equation}
is valid for all feasible solutions of  M-SP. 

By Section \ref{sec:corr}, we have that if $x$ is feasible for MSKP (formulated as in \eqref{newform}) then 
the solution $y$ in which $y_{w,q}^h = \sum\limits_{j \in B_w: g(j)=q}(d_q/f_w) \sum\limits_{i=1}^m x_{i,j}^h$, for $w=1,\ldots,t$, $q=1,\ldots,l$ and $h=q,\ldots,l$ is feasible for M-SP. Hence, by setting $\mu_{j,q}=\nu_{w,q}(d_q/f_w)$ if item $j \in B_w$ and $s_j=d_q$, it follows that the inequality is valid for MSKP (formulated as in \eqref{newform}). In fact
{\scriptsize
\begin{equation}\label{eq:valid2}
\sum\limits_{q=1}^{l} \sum\limits_{w=1}^{t} \nu_{w,q}  \sum\limits_{h=q}^{l}
\sum\limits_{j \in B_w: g(j)=q}
(d_q/f_w)  \sum\limits_{i=1}^m x_{i,j}^h=
\sum\limits_{q=1}^{l} \sum\limits_{w=1}^{t} \nu_{w,q} d_q/f_w
\sum\limits_{j \in B_w: g(j)=q}
\sum\limits_{h=q}^{l} \sum\limits_{i=1}^m x_{i,j}^h=\sum\limits_{q=1}^{l} \sum\limits_{j\in N : s_j=d_q} \mu_{j,q}\sum\limits_{h=q}^{l} \sum\limits_{i=1}^m x_{i,j}^h
\le \nu_0.
\end{equation}
}

\subsection{OPT and ordered solutions for M-SP}
As for SMKP, the definition of OPT solution is now introduced for M-SP.

\begin{definition}\label{def:OPTprop2}
Let $y$ be  a feasible solution for M-SP. The solution $y$ has the {\em OPT property}, and $y$ is called OPT solution,
if for each feasible solution $\bar y$ such that $S(y)=S(\bar y)$, it holds that $p(y^{1,h}) \ge p(\bar y^{1,h})$, 
 for $h=1,\ldots,l$. 
\end{definition}
Also in this case, observe that there are feasible solutions of  M-$SP$ that do not satisfy the OPT property, and that the items in a not OPT solutions can be always reallocated to the knapsack parts to get an OPT solutions.. 
 Relations between optimal  solutions of SMKP and M-$SP$ are now established. 
By Propositions \ref{prop:corr} and  \ref{prop:corr2}, Proposition \ref{prop:corr3} follows.
\begin{proposition}\label{prop:corr3}
Let $x$ and $y$ be two optimal solutions for SMKP and  M-SP, respectively. Then $v(x) = p(y)$.
\end{proposition}

\begin{proposition}\label{prop:M-SP-1}
Let $y$ be an optimal solution for M-$SP$ that satisfies the  OPT property.
Then, for any objective function coefficients $f_w$, with  $w =1,\ldots,t$, $d_q$ divides  $f_w y^h_{w,q}$, for  $q=1,\ldots,l$, $w =1,\ldots,t$ and $h=q,\ldots,l$.
\end{proposition}
\begin{proof}
By contradiction, let $h$ be the first part, in $y$, such that $(f_w y^h_{w,q} \mod d_q ) \ne 0$ for given $w$ and $q$, with $q\le h$. 
By definition of $h$, there exists a subset $A$ of items of type $w$ in $T_q$, not assigned to parts $q,\ldots,h-1$ in $y$, such that 
 $f(A)= \left \lceil  \frac{f_w y^h_{w,q} }{  d_q} \right \rceil -f_w y^h_{w,q}>0$. 
 Since $y$ is feasible, by \eqref{prob:M-SP-cond1}, 
it follows that it is feasible to assign to the part $h$ in $y$ the items in $A$.
 Then, the new solution $\bar y^{1,h}$ in which  $\bar y^{1,h-1}=y^{1,h-1}$,   
$\bar y^h_{e,u}= y^h_{e,u}$ for $e \ne w$ or $u \ne q$ and $y^h_{w,q}=y^h_{w,q}+f(A)/f_w$ is feasible, and we have  $p(\bar y^{1,h})>p( y^{1,h})$, contradicting the hypothesis.
\end{proof}

\begin{lemma}\label{lem:M-SP-OPT}
An optimal solution $x$ for SMKP satisfies the OPT property,  
if and only if the corresponding solution  $y$ of M-$SP$ is optimal and satisfies the OPT property.
\end{lemma}
\begin{proof}
Let $x$ be an optimal solution  for SMKP satisfying  the OPT property and let $y$ be the correspondent M-$SP$ solution. By construction, we have that $v(x^{1,h})=p(y^{1,h})$, for $h=1,\ldots,l$. 
Since $x$ satisfies the OPT property 
and by Proposition \ref{prop:corr}, it follows that $y$ satisfies the OPT property, too.
\\ On the other hand, let   $y$ be an optimal solution for M-$SP$ that satisfies the OPT property.
By Proposition \ref{prop:block}, 
 let $x$ be a solution of SMKP corresponding to $y$, in which, for each block $B_w$, $\sum\limits_{i=1}^m x^h_{i,j}=\lambda^h_{j}=f_w y_{w,q}^h$,
for $h=1,\ldots,l$ and $j \in R_w$. Recall that, by Proposition \ref{prop:M-SP-1},  $d_q$ divides $f_w y_{w,q}^h$, and, by the partition into maximal blocks  $p_w/f_w=v_j/s_j$ for all $j \in B_w$. Then by definition of $x$, $\sum\limits_{j
\in R_w:\\ g(j)=q}\sum\limits_{i=1}^m x^h_{i,j}= \sum\limits_{j
\in R_w:\\ g(j)=q}\lambda^h_{j} =  f_w y_{w,q}^h /d_q$, i.e.,
$y_{w,q}^h=\sum\limits_{j \in R_w:\\ g(j)=q}\sum\limits_{i=1}^m  x^h_{i,j} s_j/f_w$. Since, $p_w=f_w v_j/s_j$, then
$p_w y_{w,q}^h= \sum\limits_{j \in R_w:\\ g(j)=q}\sum\limits_{i=1}^m v_jx^h_{i,j}$, i.e., $p(y^{1,h})=v(x^{1,h})$ for all $h$. Observe that, $x^{1,h}$ is optimal for SMKP($h$) for all $h$. Otherwise, if a solution $\bar x^{1,h}$ such that  $v(x^{1,h})<v(\bar x^{1,h})$ exists, letting $\bar y^{1,h}$ be the solution corresponding to $\bar x^{1,h}$, by Proposition \ref{prop:corr2}, we have $p(y^{1,h})=v(x^{1,h})<v(\bar x^{1,h})= p(\bar y^{1,h})$. Contradicting the hypothesis on $y$.
\end{proof}
By Lemmas \ref{lem:opt} and \ref{lem:M-SP-OPT},  Lemma \ref{lem:opt2} follows. 
\begin{lemma}\label{lem:opt2}
Let $y$ be  an optimal  solution for M-$SP$ satisfying the OPT property.
Then $y^{1,h}$ is optimal for M-SP($h$), for $h=1,\ldots,l$.
\end{lemma}


In the following, as for SMKP, the definition of ordered solution for M-$SP$ is introduced.

\begin{definition}\label{def:orderedy}
Given a feasible solution $y$ for M-SP,  $y$ is called ordered solution if, for a given knapsack part $h=1,\ldots,l$ and a positive index $q\le h$, a set of items $\Gamma \subseteq 
\{T_1\cup \ldots \cup T_{q-1}\}$  assigned to the part $h$ in $y$ with  size  $f(\Gamma)=d_q$ exists, only if a set of items $\Gamma'$ belonging to the set 
$T_q$,  with total size  $f(\Gamma')=f(\Gamma)=d_q$ and total value $p(\Gamma')=p(\Gamma)$,
either assigned to a part bigger than $h$ or not assigned in $y$ does not exist.
\end{definition}


Given a  feasible solution $y$ for M-SP, it is always possible to build an ordered solution $\bar y$ such that $p(y^{1,h})=p(\bar y^{1,h})$ and $f(y^{1,h})=f(\bar y^{1,h})$,  for $h=1,\ldots,l$, as showed in the following.
 In fact, let us suppose that a set $\Gamma \subseteq 
\{T_1\cup \ldots \cup T_{q-1}\}$, with total size  $f(\Gamma)=d_q$ and total value $p(\Gamma)$, is assigned to a part $h$ in $y$, and that a set of items $\Gamma' \subseteq T_q$,  with total size  $f(\Gamma')=f(\Gamma)$ and total value $p(\Gamma')=p(\Gamma)$, exists, either $(i)$ assigned to a part bigger than $h$ or $(ii)$ not assigned in $y$. In case $(i)$, let $\bar y$ be  the feasible solution obtained by swapping $\Gamma$ and $\Gamma'$ in $y$,  while in case $(ii)$ let $\bar y$ be  the  feasible solution obtained by assigning $\Gamma'$ in place of $\Gamma$. 
 If $\bar x$ contains a new set $\Gamma$ as defined above, then the above argument can repeated, otherwise, $\bar x$ is an ordered solution.

Proposition \ref{prop:optimal-ord-opt}  directly follows by Proposition \ref{prop:M-SP-1} and by the correspondance between feasible solutions of SMKP and M-SP (stated in Section \ref{sec:corr}).
\begin{proposition}\label{prop:optimal-ord-opt}
If $x$ is an optimal ordered OPT solution for  SMKP then the corresponding solution $y$ of M-$SP$   is an optimal ordered OPT solution, too, and vice versa. 
\end{proposition}

Let $MP^{OO}$ be the  set containing the OPT and ordered solutions of M-SP, and let $MO^{OO}$ be the set of the optimal solutions of M-$SP$ in $MP^{OO}$. 

M-$SP$ can be formulated as: $\max \{\sum\limits_{h=1}^{l}\sum\limits_{q=1}^{h}\sum\limits_{w=1}^{t}p_w y_{w,q}^h: 
y \in MP^{OO}\}
$.


\section{Computing the optimal solutions of M-$SP$ and MSKP}\label{sec:enum}


In this section, we describe an inductive procedure for decomposing and
enumerating  the optimal solutions in $MP^{OO}$, i.e., the solutions in $MO^{OO}$. 
 At this aim, the following notation is introduced.

Given a positive integer $b\le l$ and an item type $k$, with $1\le k \le t$, let $N(k)$ be the set of items of types $1,\ldots, k$ belonging to the sets $T_1,\ldots,T_b$.
Let $\delta_h =\min\{d_h,d_b\}$, for $h =1,\ldots, l$, and let $F_h \le \bar c_h$, for $h=1,\ldots,l$, be positive integers such that 
$\delta_h$  divides $F_h$.
 We denote by $MP(k,b,{\bf F})$ the  set 
 of  points $y$
satisfying the  following conditions:
\begin{eqnarray}
 \sum\limits_{w=1}^{k}\sum\limits_{q=1}^{\min\{ h,b\}} \left\lceil \frac{f_w y_{wq}^h}{d_q}\right\rceil d_q &\leq& F_h \text{ for } h=1,\ldots, l \label{prob:M-SP-cond1R}\\
\sum\limits_{h=q}^{l}\left\lceil \frac{f_w y_{wq}^h}{d_q}\right\rceil&\leq&  \frac{f_w \tilde b_{wq}}{d_q} \;\;\;\;\text{ for } w
=1,\ldots, k \text{ and } q =1,\ldots, b \label{prob:M-SP-cond2R}\\
y_{wq}^h& \in&\mathbb{Z}^+\cup\{0\}\;\;\; \text{ for } w =1,\ldots, k, \; q =1,\ldots, b \text{ and } h =q,\ldots, l.
\end{eqnarray}
Furthermore, we denote by $MP^{OO}(k,b,{\bf F})$ the
 set of the points in $MP(k,b,{\bf F})$ that are OPT and ordered solutions. Let $MO^{OO}(k,b,{\bf F})$ be the set of the optimal solutions in $MP^{OO}(k,b,{\bf F})$.

We denote by 
MS-P$(k,b,{\bf F})$ the integer program
\begin{equation}
\max \{ \sum\limits_{h=1}^{l}\sum\limits_{q=1}^{\min\{ h,b\}}\sum\limits_{w=1}^{t}p_w y_{wq}^h: \;\;\;\; y \in MP^{OO}(k,b,{\bf F})\}
\end{equation}


Let  $T'_1,\ldots,T'_b$ be the sets containing the items in $N(k)$ belonging to the sets $T_1,\ldots,T_b$, respectively, and let $T'_{b+1}=T'_{b+2}=\ldots=T'_l=\emptyset$. W.l.o.g., in the rest of the paper we assume that $N(k)$ contains at least an item of type $k$. A consequence of this fact is that $f_k$ divides $d_b$.

In the following lemma we show that, if $b=1$,  $MO^{OO}(k,b,{\bf F})$ contains a single point.
\begin{lemma}\label{lem:solk=1}
 $MO^{OO}(k,1,{\bf F})$ contains a single  solution, $\hat y$, in which
 \begin{equation}\label{eq:j,1,Flemma}
\sum\limits_{h=1}^l \hat y^h_{j,1}= \min\{ \sum\limits_{h=1}^l F_h - \sum\limits_{w=1}^{j-1}\sum\limits_{h=1}^l \hat y^h_{w,1} ; \tilde b_{j,1} \} \;\; \text{for } j=1,\ldots,k,
\end{equation}
and 
 \begin{equation}\label{eq:j,1,single-sol}
 \hat y^h_{j,1}= \min\{  F_h - \sum\limits_{w=1}^{j-1} \hat y^h_{w,1}; \tilde b_{j,1} -\sum\limits_{p=1}^{h-1}\hat y^p_{j,1}\}
  \;\; \text{for } j=1,\ldots,k, \;\; \text{and } h=1,\ldots,l.
\end{equation}
\end{lemma}
\begin{proof}
Observe that, in M-SP($k,1,{\bf F}$), all items have the same size $\delta_1=d_1=f_1=1$, since they belong to the set $T_1'$, and can be assigned to all the parts of the knapsack. Hence, the $l$ knapsack parts can  be replaced by a single part of capacity $\sum\limits_{h=1}^l F_h$. 
By the ordering of the item types proposed in Section \ref{sec:trasnform}, since all items have size 1, we have $p_k< p_{k-1}<\ldots<p_1$.
Hence, an optimal solution $\hat y$ must assign the following number items of type 1:
\begin{equation}\label{eq:1,1,F}
\sum\limits_{h=1}^l \hat y^h_{1,1}= \min\{ \sum\limits_{h=1}^l F_h ; \tilde b_{1,1} \}.
\end{equation}
Moreover since $\hat y$ satisfies the OPT property, it must be $ \hat y^1_{1,1}= \min\{  F_1 ; \tilde b_{1,1} \}$, 
$ \hat y^2_{1,1}= \min\{  F_2 ; \tilde b_{1,1}-\hat y^1_{1,1} \}$, $\ldots$, and finally $ \hat y^l_{1,1}= \min\{  F_l ; \tilde b_{1,1} - \sum\limits_{h=1}^{l-1}\hat y^h_{1,1}\}$.\\
Let us consider now the items of type 2. Applying similar arguments to those used for items of type 1, we have
\begin{equation}\label{eq:2,1,F}
\sum\limits_{h=1}^l \hat y^h_{2,1}= \min\{ \sum\limits_{h=1}^l F_h - \sum\limits_{h=1}^l \hat y^h_{1,1} ; \tilde b_{2,1} \}.
\end{equation}
Also in this case, since $\hat y$ satisfies the OPT property, the unique values attained by $ \hat y^1_{2,1},  \hat y^2_{2,1}, \ldots,  \hat y^l_{2,1}$ can be easily derived, i.e., 
 $ \hat y^1_{2,1}= \min\{  F_1 - \hat y^1_{1,1}; \tilde b_{2,1} \}$, 
$ \hat y^2_{2,1}= \min\{  F_2 - \hat y^2_{1,1}; \tilde b_{2,1}-\hat y^1_{2,1} \}$, $\ldots$,$ \hat y^l_{2,1}= \min\{  F_l \hat y^1_{1,1}; \tilde b_{2,1} - \sum\limits_{h=1}^{l-1}\hat y^h_{2,1}\}$.\\
In general we have
\begin{equation}\label{eq:j,1,F}
\sum\limits_{h=1}^l \hat y^h_{j,1}= \min\{ \sum\limits_{h=1}^l F_h - \sum\limits_{w=1}^{j-1}\sum\limits_{h=1}^l \hat y^h_{w,1} ; \tilde b_{j,1} \} \;\; \text{for } j=1,\ldots,k,
\end{equation}
and the unique values attained by $ \hat y^1_{j,1},  \hat y^2_{j,1}, \ldots,  \hat y^l_{j,1}$ are $ \min\{  F_1 - \sum\limits_{w=1}^{j-1} \hat y^1_{w,1}; \tilde b_{j,1} \}, \min\{  F_2 - \sum\limits_{w=1}^{j-1}\hat  y^2_{w,1}; \tilde b_{j,1} -\hat y^1_{j,1}\}, \ldots,  \min\{  F_l - \sum\limits_{w=1}^{j-1} \hat y^l_{w,1}; \tilde b_{j,1} -\sum\limits_{h=1}^{l-1}\hat y^h_{j,1}\}$, respectively.
\end{proof}


\medskip
In what follows a characterization of the optimal solutions  in  $MP^{OO}(k,b,{\bf F})$ when $b>1$ is given.
Given an item type $j$, let $H_j$ be the set of items in MS-P$(k,b,{\bf F})$ having unit gain value
strictly bigger than $\frac{p_j}{f_j}$, for $j=1,\ldots,k$, i.e., 
\begin{equation}\label{eq:Hk}
H_j=\{t \in N(k): \frac{p_t}{f_t}>\frac{p_j}{f_j}\}.
\end{equation}
Given a  part $g=1, \ldots, l$ and an item type $j=1,\ldots,k$, 
let $\bar H_j^g$ be the subset of items of $H_j$ defined as follows
\begin{equation}\label{def: bar H}
\bar H_j^g=\{ H_j \setminus \{ u \in T'_g: u>j\}\} \cap (T'_1\cup T'_2\ldots  \cup T'_{g}).
\end{equation}
Note that, since $k$ is the biggest item type in MS-P$(k,b,{\bf F})$, we have
\begin{equation}\label{def: bar Hk}
\bar H_k^g= H_k \cap (T'_1\cup T'_2\ldots  \cup T'_{g}).
\end{equation}




In what follows, we introduce and (recursively) define the sets ${\cal H}_j^{g}$ for $j=k$ and  $g=1, \ldots, l$. We define ${\cal H}_k^{1}=\bar H_k^1$.
Let 
$v^{1}=\min\{ \lfloor {f({\bar H}_k^{1}) }/{\delta_1}\rfloor \delta_1; F_{1}\}=\min\{ f({\bar H}_k^{1}) ; F_{1}\}$ 
and let  ${\cal H}_k^{1}(min)$ be a subset of  ${\bar H}^{1}_k$ of total size $v^{1}$ (such a set exists, since all items in ${\bar H}^{1}_k $ have size $d_1=\delta_1=1$).
Moreover, let
${\cal H}_k^{2}=\bar H_k^{2} \setminus {\cal H}_k^{1}(min)$
 and let ${\cal H}_k^{2}(min)$ be  a subset of ${\cal H}_k^{2}$ of total size
  $v^{2}=\min\{ \lfloor {f({\cal H}_k^{2}) }/{\delta_2}\rfloor \delta_2; F_{2}\}$. 
By Proposition \ref{prop:div}, since $\delta_2$ divides $v^{2}$  and $F_2$, and since ${\cal H}_k^{2}$ contains items with size that are divisible and not bigger than $\delta_2$, the set ${\cal H}_k^{2}(min)$   always exists.\\
In general, let ${\cal H}_k^{g}$ and ${\cal H}_k^{g}$, for $g=1, \ldots, l$, be recursively defined as
\begin{equation}\label{eq:cal H}
{\cal H}_k^{g}=\bar H_k^g \setminus \bigcup\limits_{h=1}^{g-1}{\cal H}_k^{h}(min)
\end{equation}
in which
 ${\cal H}_k^{h}(min)$, for $h=1, \ldots, g-1$, is a subset of ${\cal H}_k^{h}$ of total size $v^{h}$, where
\begin{equation}\label{eq:w^{i_{k}}(min)}
v^{h}=\min\{ \lfloor {f({\cal H}_k^{h}) }/{\delta_h}\rfloor \delta_h; F_{h}\}.
\end{equation}
By Proposition \ref{prop:div}, it easy to see that the sets ${\cal H}_k^{g}(min)$, for $g=3,\ldots,l-1$, always exist, too.
\begin{definition}
Given a solution $\hat y$ in  $MP^{OO}(k,b,{\bf F})$ and a part $g\in \{1,\ldots, b\}$ we call the set of items
in $\bar H_k^g \setminus S(\hat y^{1,g-1})$ as {\em the set of items of $H_k$ available to be assigned to the part $g$ in $\hat y$}.
\end{definition}
Two key Lemmas are now introduced useful for characterizing the optimal solutions in  $MP^{OO}(k,b,{\bf F})$.
\begin{lemma}\label{lem:at:least_available}
Given an optimal solution $\hat y$ in  $MP^{OO}(k,b,{\bf F})$, $b>1$, for each $g=1,\ldots, l$, the total size of the items of $H_k$ \emph{available} to be assigned to the part $g$ in $\hat y$ (i.e., $f(\bar H_k^g \setminus S(\hat y^{1,g-1}))$) is at least $\lfloor f({\cal H}_k^{g})/\delta_g\rfloor \delta_g$.
\end{lemma}
\begin{proof}

 The quantity $f(\bar H_k^g \cap S(\hat y_k^{1,g-1}))$ is minimum when the total size of the  items in $H_k$ assigned to the parts $1,\ldots,g-1$ in $\hat y$  is maximum.
The minimum value of $f(\bar H_k^g \cap S(\hat y_k^{1,g-1}))$ can be determined by an iterative procedure described in the following and proved later by induction. Starting from the first part of the knapsack, the procedure assigns  items of $H_k$ with a total size as bigger as possible to each part.
Hence, recalling that all items in ${\bar H}^{1}_k$ have size $\delta_1=1$, the maximum total size of items in  $\bar H_k^1$ that can be assigned to the part 1 is
\begin{equation}\label{eq:v^{1}(max)}
\bar v^1 = \min\{f({\bar H}^{1}_k ) ; F_{1}\}.
\end{equation}
Hence, we must have $f(\bar H^1_k \cap S(\hat y_k^{1}))\le \bar v^1$. Let ${\cal H}_k^{1}(max)$ be a set of items of ${\bar H}^{1}_k$ of total size $\bar v^1$. Note that, since $d_1=1$ such a set exists, and, recalling the definition of ${\cal H}_k^{1}(min)$, we have $f({\cal H}_k^{1}(max))=f({\cal H}_k^{1}(min))$.
We define ${\cal H'}_k^{2}= \bar H_k^2 \setminus {\cal H}_k^{1}(max)$. 
In general, for a given knapsack part $h$, we define
\begin{equation}\label{eq:H'_definizione}
{\cal H'}_k^{h}= \bar H_k^h \setminus \bigcup\limits_{t=1}^{h-1}{\cal H}_k^{t}(max)
\end{equation}
where ${\cal H}_k^{h}(max)$ is a subset of ${\cal H'}_k^{h}$ of total size $\bar v^h=\min\{F_h, f({\cal H'}_k^{h})\}$. Observe that, 
 a set of total size $\bar v^h$ always exists by Proposition \ref{prop:div}: such a fact holds, since the biggest size of the items in ${\cal H'}_k^{h}$ is  $\delta_h$, the item sizes are divisible, and $\delta_h$ divides $F_{h}$. Then, the procedure assigns to the part $h$ the set ${\cal H}_k^{h}(max)$, for $h=1,\ldots,g-1$.\\
We now show, by induction on the knapsack parts, that the above  procedure is correct. Namely, that a set of the items of $H_k$ \emph{available} to be assigned to the part $g$ with minimum total size in  $\hat y$ is
\begin{equation}\label{eq:min_vol_set_g}
 {\cal H'}_k^{g}=\bar H_k^g \setminus  \bigcup\limits_{h=1}^{g-1} {\cal H}_k^{h}(max)
\end{equation}
i.e.,
\begin{equation}\label{eq:minimovolset}
f( {\cal H'}_k^{g}) \le f(\bar H_k^g \setminus S(\hat y_k^{1,g-1})).
\end{equation}
When $g=1$, relation \eqref{eq:minimovolset} follows by \eqref{eq:v^{1}(max)} and by definition of ${\cal H}_k^{1}(max)$. By induction, suppose that 
\begin{equation}\label{eq:min_vol_set_g1c}
f( {\cal H'}_k^{g-1}) \le f(\bar H_k^{g-1} \setminus S(\hat y_k^{1,g-2})),
\end{equation}
and let $\tilde y$ be a feasible solution such that  $ f(\bar H_k^{g-1} \setminus S(\tilde y_k^{1,g-2}))=f( {\cal H'}_k^{g-1})$. In $\tilde y$ and $\hat y$, the maximum total size of the items that can be assigned to the part $g-1$ respectively are
$ \min\{F_{g-1};f( {\cal H'}_k^{g-1}) \}=f( {\cal H}_k^{g-1}(max))$
and
$  \min\{F_{g-1}; f(\bar H_k^{g-1} \setminus S(\hat y_k^{1,g-2}))\}.$
Recall that, by definition of $\bar H_k^g$ and since $k$ is the biggest item type, given a feasible solution $y$, it holds that  $\bar H_k^g =\bar H_k^{g-1} \cup (H_k \cap T'^g)$.

Two cases can be distinguished: $a)$ $F_{g-1}\le f( {\cal H'}_k^{g-1})$;  $b)$ $F_{g-1}> f( {\cal H'}_k^{g-1})$. In Case $a)$, 
the maximum total size of the items that can be assigned to the part $g$  in $\tilde y$ and $\hat y$, respectively, are\\
 $ f(\bar H_k^{g-1} \setminus S(\tilde y_k^{1,g-2})) +f(H_k \cap T'^g)-F_{g-1}=f( {\cal H'}_k^{g-1})+f(H_k \cap T'^g)-f( {\cal H}_k^{g-1}(max))=f( {\cal H'}_k^{g})$\\
and
$ f(\bar H_k^{g-1} \setminus S(\hat y_k^{1,g-2})) +f(H_k \cap T'^g)-F_{g-1}$. And the thesis easily follows  by \eqref{eq:min_vol_set_g1c}.
In Case $b)$, if $F_{g-1}\ge f(\bar H_k^{g-1} \setminus S(\hat y_k^{1,g-2}))$,  then the minimum total size of the items that remains to assign to the part $g$ in $\tilde y$ and $\hat y$ is $ f(H_k \cap T'^g)= f({\cal H'}_k^{g})$. (Recall, that, by definition, $H_k \cap T'^g=\emptyset$ if $g>b$.) Hence, the  \eqref{eq:minimovolset} follows. Otherwise, if  $ f(\bar H_k^{g-1} \setminus S(\hat y_k^{1,g-2})) - F_{g-1} =\alpha>0 $, we have $f({\cal H'}_k^{g})=f(H_k \cap T'^g)$ and 
$f(\bar H_k^{g} \setminus S(\hat y_k^{1,g-1}))\ge f(H_k \cap T'^g) + \alpha$, and relation \eqref{eq:minimovolset} holds.

By the definition of ${\cal H'}_k^{g}$ (see \eqref{eq:H'_definizione}), by  \eqref{eq:minimovolset} and by definition of  ${\cal H}_k^{g}$ (see  \eqref {eq:cal H} and \eqref{eq:w^{i_{k}}(min)}) we have
\begin{equation}\label{eq:min_vol_set_g1}
f( {\cal H'}_k^{g}) \le f({\cal H}_k^{g}).
\end{equation}
The thesis of the lemma is proved by showing that
\begin{equation}\label{eq:min_vol_set_g2}
f( {\cal H'}_k^{g}) \ge \lfloor f( {\cal H'}_k^{g})/\delta_g\rfloor \delta_g= \lfloor f({\cal H}_k^{g})/\delta_g\rfloor \delta_g,
\end{equation}
implying that, by \eqref{eq:minimovolset}, $ f(\bar H_k^g \setminus S(\hat y_k^{1,g-1}))\ge  \lfloor f({\cal H}_k^{g})/\delta_g\rfloor \delta_g.$ Relation \eqref{eq:min_vol_set_g2} is proved by induction. When $g=1$, recalling the definitions of ${\cal H}_k^{g}$ \eqref{eq:cal H} and ${\cal H'}_k^{g}$ \eqref{eq:H'_definizione},  the thesis trivially holds since
 ${\cal H}_k^{0}(max)= {\cal H}_k^{0}(min)=\emptyset$. Assume that Relation \eqref{eq:min_vol_set_g2} holds for $g-1$ and show it for $g$.
By induction, we have
\begin{equation}\label{eq:min_vol_set_g2bis}
\lfloor f( {\cal H'}_k^{g-1})/\delta_{g-1}\rfloor \delta_{g-1}= \lfloor f({\cal H}_k^{g-1})/\delta_{g-1}\rfloor \delta_{g-1}
\end{equation}
and hence
 $f( {\cal H'}_k^{g-1})+\delta_{g-1} > f({\cal H}_k^{g-1})$.
Recall that, by definition,
$\bar H_k^{g}=\bar H_k^{g-1}\cup (T'_g\cap H_k)$
where, by definition of M-SP, $\delta_g$ divides $f(T'_g\cap H_k)$.
Hence, by definition,
{\small 
\begin{equation}\label{eq:min_set_new}
{\cal H'}_k^{g}=\bar H_k^{g} \setminus  \bigcup\limits_{h=1}^{g-1} {\cal H}_k^{h}(max)=\bar H_k^{g-1}\cup (T'_g\cap H_k)\setminus  \bigcup\limits_{h=1}^{g-1} {\cal H}_k^{h}(max)={\cal H'}_k^{g-1} \cup (T'_g\cap H_k)\setminus {\cal H}_k^{g-1}(max)
\end{equation}}
and
{\small \begin{equation}\label{eq:max_set_new}
{\cal H}_k^{g}=\bar H_k^{g} \setminus  \bigcup\limits_{h=1}^{g-1} {\cal H}_k^{h}(min)=\bar H_k^{g-1}\cup (T'_g\cap H_k)\setminus  \bigcup\limits_{h=1}^{g-1} {\cal H}_k^{h}(min)={\cal H}_k^{g-1} \cup (T'_g\cap H_k)\setminus {\cal H}_k^{g-1}(min)
\end{equation}
}
where the last equality of \eqref{eq:min_set_new} and \eqref{eq:max_set_new} respectively follow since
the sets ${\cal H}_k^{h}(max)$ and ${\cal H}_k^{h}(min)$, for $h=1,\ldots,g-2$, are disjoint subsets of $\bar H_k^{g-1}$, and since,  by definition, ${\cal H'}_k^{g-1}=\bar H_k^{g-1} \setminus  \bigcup\limits_{h=1}^{g-2} {\cal H}_k^{h}(max)$ and  ${\cal H}_k^{g-1}=\bar H_k^{g-1} \setminus  \bigcup\limits_{h=1}^{g-2} {\cal H}_k^{h}(min)$.
From \eqref{eq:min_set_new} and \eqref{eq:max_set_new} we respectively have
\begin{eqnarray}\label{eq:min_set_new1}
f({\cal H'}_k^{g})=
f({\cal H'}_k^{g-1})+f(T'_g\cap H_k)-f({\cal H}_k^{g-1}(max))
\end{eqnarray}
and
\begin{eqnarray}\label{eq:max_set_new1}
f({\cal H}_k^{g})=f({\cal H}_k^{g-1})+f(T'_g\cap H_k)-f({\cal H}_k^{g-1}(min)).
\end{eqnarray}

Two cases are possible: $(1)$ $f({\cal H}_k^{g-1}(max))=\min\{f({\cal H'}_k^{g-1}), F_{g-1}\}=F_{g-1}$; and $(2)$ $f({\cal H}_k^{g-1}(max))=\min\{f({\cal H'}_k^{g-1}), F_{g-1}\}=f({\cal H'}_k^{g-1})$.

In Case $(1)$, since by induction
$\lfloor f( {\cal H'}_k^{g-1})/\delta_{g-1}\rfloor \delta_{g-1}= \lfloor f({\cal H}_k^{g-1})/\delta_{g-1}\rfloor \delta_{g-1}$ (see \eqref{eq:min_vol_set_g2bis}) and since $\delta_{g-1}$ divides $F_{g-1}$, we have $\min\{ \lfloor f({\cal H}_k^{g-1})/\delta_{g-1}\rfloor \delta_{g-1}, F_{g-1}\}=F_{g-1}$, too, i.e.,  $f({\cal H}_k^{g-1}(min))=f({\cal H}_k^{g-1}(max))=F_{g-1}$. Hence,  since $\delta_{g-1}$ divides 
$f(T'_g\cap H_k)-f({\cal H}_k^{g-1}(min))=f(T'_g\cap H_k)-f({\cal H}_k^{g-1}(max))$ and $\delta_{g}$, and from  \eqref{eq:min_set_new1}, we have:
\begin{eqnarray*}
f({\cal H'}_k^{g})=f({\cal H'}_k^{g-1})+f(T'_g\cap H_k)-f({\cal H}_k^{g-1}(max)) \ge \\  \lfloor f({\cal H'}_k^{g-1})/\delta_{g-1}\rfloor \delta_{g-1}+f(T'_g\cap H_k)-f({\cal H}_k^{g-1}(max)) 
 =\\ \lfloor f({\cal H}_k^{g-1})/\delta_{g-1}\rfloor \delta_{g-1}+f(T'_g\cap H_k)-f({\cal H}_k^{g-1}(min))=\\
  \left\lfloor \frac{f({\cal H}_k^{g-1})+f(T'_g\cap H_k)-f({\cal H}_k^{g-1}(min))}{\delta_{g-1}}\right\rfloor \delta_{g-1} \ge \\
   \left\lfloor \frac{f({\cal H}_k^{g-1})+f(T'_g\cap H_k)-f({\cal H}_k^{g-1}(min))}{\delta_{g}}\right\rfloor \delta_{g}=
   \left\lfloor \frac{f({\cal H}_k^{g})}{\delta_{g}}\right\rfloor \delta_{g}
\end{eqnarray*}
showing the thesis.\\
In Case $(2)$, relation \eqref{eq:min_set_new1}  becomes
$f({\cal H'}_k^{g})=f(T'_g\cap H_k).$
 By definition of ${\cal H}_k^{g-1}(min)$ and by \eqref{eq:min_vol_set_g2bis} we have

$f({\cal H}_k^{g-1}(min))= \min\{ \lfloor \frac{f({\cal H}_k^{g-1}) }{\delta_{g-1}}\rfloor \delta_{g-1}; F_{g-1}\}=\min\{ \lfloor \frac{f({\cal H'}_k^{g-1}) }{\delta_{g-1}}\rfloor \delta_{g-1}; F_{g-1}\}=\lfloor \frac{f({\cal H'}_k^{g-1}) }{\delta_{g-1}}\rfloor \delta_{g-1}.$ Hence,
Equation \eqref{eq:max_set_new1} becomes:
$f({\cal H}_k^{g})= f({\cal H}_k^{g-1})+f(T'_g\cap H_k)-\lfloor \frac{f({\cal H}_k^{g-1}) }{\delta_{g-1}}\rfloor \delta_{g-1}.$
Since $f({\cal H}_k^{g-1})-\lfloor \frac{f({\cal H}_k^{g-1}) }{\delta_{g-1}}\rfloor \delta_{g-1}<\delta_{g-1} $ and $\delta_{g} \ge \delta_{g-1}$, relation \eqref{eq:min_vol_set_g2} holds, i.e., the thesis.

\end{proof}

 In Lemma \ref{lem:minimum2}, an upper bound on the total size of the items of $H_k$ \emph{available} to be assigned to the part $h$ in an optimal solution $\hat y$ is given.

 


\begin{lemma}\label{lem:minimum2}
Given an optimal  solution $\hat y$ in  $MP^{OO}(k,b,{\bf F})$, with $b>1$,
the total size of the items of $H_k$,  \emph{available} to be assigned to the part $h$ in $\hat y$ (i.e., $f(\bar H_k^h \setminus S(\hat y^{1,h-1}))$) is smaller than  or equal to $f({\cal H}_k^{h}) + \delta_{h}$, for $h=1,\ldots, l$.
\end{lemma}
\begin{proof}
 The proof is by induction on the knapsack part $h$. If $h=1$ then $S(\hat y^{1,0})=\emptyset$, and 
$\bar H_k^1 \setminus S(\hat y^{1,0})=\bar H_k^1$. Since, by definition, ${\cal H}_k^{1}=\bar H_k^1$, the thesis trivially follows. Assume the thesis holds up to part $h <l$, and let us consider the part $h+1$.
By induction, the total size of the items in the set $\bar H_k^h \setminus S(\hat y^{1,h-1})$, in the following denoted as $A_h$ (i.e., $A_h=f(\bar H_k^h \setminus S(\hat y^{1,h-1}))$), is smaller than  or equal to $f({\cal H}_k^{h}) + \delta_{h}$. Observe that,  $A_{h+1}=f(\bar H_k^{h+1} \setminus S(\hat y^{1,h}))$  is maximum when  $A_h$ is maximum and  $S(\hat y^h) \cap \bar H_k^{h}$ has minimum total size. 
By definition,  $\bar H_k^{h+1} \setminus S(\hat y^{1,h})=\bar H_k^{h} \setminus (S(\hat y^{1,h-1})\cup  S(\hat y^{h}))\cup (H_k \cap T'_{h+1})$. (Recall that, by definition of $T'_h$, $H_k \cap T'_{h}=\emptyset$ if $h>b$.) Observe that, $S(\hat y^{1,h-1})$ and $S(\hat y^{h})$ are disjoint sets, and $\bar H_k^h\setminus (S(\hat y^{1,h-1})\cup  S(\hat y^{h}))$  and $(H_k \cap T'_{h+1})$ are disjoint sets, too. Hence, in general, we have
 \begin{equation}\label{eq:A_h+1}
 f(\bar H_k^{h+1} \setminus S(\hat y^{1,h}))=A_{h+1}= A_h - f(\bar H_k^h \cap S(\hat y^{h})) + f(H_k \cap T'_{h+1}).
 \end{equation}
Observe that, if $\bar H_k^{h}=\emptyset$, then by definition $A_h=0$ $\bar H_k^{h+1}=H_k \cap T'_{h+1}$ and $f({\cal H}_k^{h+1})=f(H_k \cap T'_{h+1})$. Hence, by \eqref{eq:A_h+1} and since $\delta_{h+1}\ne 0$, the Lemma trivially holds.
In the following, let us suppose that $\bar H_k^{h}\ne \emptyset$. 

Two cases can be considered: $1)$ $F_h > \lfloor A_h/\delta_h\rfloor \delta_h$; $2)$ $F_h \le \lfloor A_h/\delta_h\rfloor \delta_h$.\\
Case $1$. Since, by definition, $\delta_h$ divides $F_h$, then $F_h \ge  \lfloor A_h/\delta_h\rfloor \delta_h + \delta_h> A_h$. Let $\gamma$ be  the total size of the items of $H_k \setminus S(\hat y^{1,h-1})$ assigned by $\hat y$  to $h$, i.e., $\gamma=f(\hat y^h \cap H_k )=\sum\limits_{r \in H_k}f_r \hat y_r^{h}$.
In the following we show that 
$\gamma > \lfloor A_h /\delta_h\rfloor \delta_h -\delta_h.$
By contradiction, let us suppose that $\gamma\le \lfloor A_h /\delta_h\rfloor \delta_h-\delta_h$, and 
let $\Omega= \bar H_k^h \setminus S(\hat y^{1,h})$. Hence,  $f(\Omega)=A_h-\gamma \ge A_h -  \lfloor A_h /\delta_h\rfloor \delta_h+\delta_h \ge \delta_h$.

Let $\Psi=S(\hat y^{h} ) \setminus H_k$, i.e., the set of items in $ N(k)\setminus H_k $ assigned to part $h$ in $\hat y$, and let $\sigma=f(\Psi)=\sum\limits_{r \in N(k)\setminus H_k}f_r \hat y_r^{h}$, i.e., the total size of the items in $ N(k)\setminus H_k $ assigned to the part $h$ in $\hat y$.
If $\sigma \le F_h -A_h=F_h - \gamma - f(\Omega)$, then it is feasible to assign all items in $\Omega$ to the part $h$ in $\hat y$. Namely, the new solution $\tilde y^{1,h}$ that assigns  all items assigned in $S(\hat y^{1,h})$ to the parts $1,\ldots,h$, and all items in $\bar H_k^h \setminus S(\hat y^{1,h-1})$ to the part $h$ belongs to $MP^{OO}(k,b,{\bf F})$ and has a strictly better objective function than $\hat y^{1,h}$ (since  $S(\hat y^{1,h})$ is a subset strictly contained in  $S(\tilde y^{1,h})$), contradicting Lemma \ref{lem:opt2}.


Suppose now that $f(\Psi)=\sigma>  F_h - A_h$. Since $\delta_h$ is the biggest  size of the items that can be assigned to the part $h$, the item sizes are divisible and $\hat y$ is feasible (i.e., satisfies constraints \eqref{prob:M-SP-cond1}), $\Psi$ 
can be partitioned into two subsets $\Psi^1$ and $\Psi^2$, such that: 
$(i)$ $f(\Psi^1)= F_h - \lfloor A_h /\delta_h\rfloor \delta_h$, $(ii)$ $f(\Psi^2)=\sigma - f(\Psi^1)>0$, and $(iii)$  constraints \eqref{prob:M-SP-cond1} are satisfied, i.e., $\sum\limits_{q=1}^{\min\{h,b\}} (\sum_{w\in \Psi_1\cap T'_q} \left\lceil \frac{f_w \hat y_{wq}^h}{\delta_q}\right \rceil +\sum_{w\in \Psi_2\cap T'_q} \left\lceil \frac{f_w \hat y_{wq}^h}{\delta_q}\right \rceil)=\sigma$.\\
Observe that $\delta_h$ divides $f(\Psi^1)$ and that
$f(\Psi^2)\le \lfloor A_h /\delta_h\rfloor \delta_h-\gamma$. This last inequality holds since otherwise\\
 $f(\hat y^{h})=\gamma + f(\Psi^1)+f(\Psi^2)> \gamma + F_h - \lfloor A_h /\delta_h\rfloor \delta_h +\lfloor A_h /\delta_h\rfloor \delta_h - \gamma= F_h$, i.e., $\hat y \notin MP^{OO}(k,b,{\bf F})$. Hence, since $\gamma \le \lfloor A_h /\delta_h\rfloor \delta_h-\delta_h$ (by the contradiction hypothesis) then 
 $f(\Psi^2)\le \delta_h$.
 The total size of the items assigned to the part $h$ in $\hat y$ is 
 \begin{equation}\label{eq:psi}
 f(\hat y^{h})=\gamma + f(\Psi^1)+f(\Psi^2)\le  F_h - \lfloor A_h /\delta_h\rfloor \delta_h+ \lfloor A_h /\delta_h\rfloor \delta_h-\delta_h + f(\Psi^2)=F_h-\delta_h + f(\Psi^2).
 \end{equation}
 Recall that $f(\Omega)\ge \delta_h$. Hence, since $\delta_h$ is the biggest  size of the items that can be assigned to the part $h$, the item sizes are divisible, and $\hat y$ is feasible, a subset  $\Omega^1 \subseteq \Omega$ exists such that $f(\Omega^1)=\delta_h$ and $\sum\limits_{q=1}^{\min\{h,b\}}\sum_{w\in \Omega_1 \cap T'_q} \left\lceil \frac{f_w \hat y_{wq}^h}{\delta_q}\right \rceil=\delta_h$.
 By \eqref{eq:psi}, it is feasible to replace $\Psi^2$ by $ \Omega_1$ in $\hat y$.
Hence, the new solution $\tilde y$ with $\tilde y^{1,h-1}= \hat y^{1,h-1}$, $\tilde y^h_r=\hat y^h_r$ for $r \in \Psi^1$, 
$\tilde y^h_r=0$ for $r \in \Psi^2$  
and  $\tilde y^h_w=|\Omega^1_w|$, for $w \in \Omega^1$,  belongs to $MP^{OO}(k,b,{\bf F})$ and has a strictly better objective function than $\hat y^{1,h}$ because  $p_w/f_w >  p_r/f_r$ for all $w \in \Omega^1 \subseteq H_k $ and $r \in  \Psi^1 \subseteq N(k)\setminus H_k$, contradicting Lemma \ref{lem:opt2}. Hence, we have proved that $\gamma \ge \lfloor A_h /\delta_h\rfloor \delta_h$.\\
Recalling that $f(\bar H_k^h \cap S(\hat y^{h}) ) = \gamma $, relation \eqref{eq:A_h+1} reads as
  \begin{equation}\label{eq:A_h+1_case1}
 A_{h+1}=A_h -\gamma + f(H_k \cap T'_{h+1}).
 \end{equation}
By the above discussion we have  $ \gamma > \lfloor A_h /\delta_h\rfloor \delta_h -\delta_h$. Moreover, by Lemma \ref{lem:at:least_available}, $A_h \ge \lfloor f({\cal H}_k^{h})/\delta_h\rfloor \delta_h$ and, hence,  $ \gamma >  \lfloor f({\cal H}_k^{h})/\delta_h\rfloor \delta_h -\delta_h$, too.

As a consequence and since, by induction, $A_h \le f({\cal H}_k^{h}) + \delta_{h}$, relation \eqref{eq:A_h+1_case1} becomes
 $$A_{h+1}<  f({\cal H}_k^{h}) + \delta_{h} -\lfloor f({\cal H}_k^{h})/\delta_h\rfloor \delta_h+ \delta_h + f(H_k \cap T'_{h+1})$$
 that by the hypothesis of Case 1 is equal to $$   f({\cal H}_k^{h}) + 2\delta_{h} -\min\{F_{h}; \lfloor f({\cal H}_k^{h})/\delta_h\rfloor \delta_h\} + f(H_k \cap T'_{h+1}))= f({\cal H}_k^{h+1}) +2\delta_{h}\le f({\cal H}_k^{h+1}) + \delta_{h+1},$$
where the last inequality holds since  $\delta_{h+1} \ge 2\delta_h$. Hence, the thesis follows.

Case 2. 
Again, let  $\gamma=f(\hat y^h \cap H_k)=\sum\limits_{r \in H_k}f_r \hat y_r^{h}$.
If $F_h=0$, then $\gamma=0$ and $ \min\{F_{h};\lfloor f({\cal H}_k^{h})/\delta_h\rfloor \delta_h  \}=0$, too. 
Since by induction, $A_h \le f({\cal H}_k^{h}) + \delta_{h}$, relation \eqref{eq:A_h+1} becomes 
$$A_{h+1}=A_h - \gamma + f(H_k \cap T'_{h+1}) \le  $$$$f({\cal H}_k^{h}) + \delta_h - \min\{F_{h};\lfloor f({\cal H}_k^{h})/\delta_h\rfloor \delta_h  \} + f(H_k \cap T'_{h+1})\le  f({\cal H}_k^{h+1}) + \delta_{h} < f({\cal H}_k^{h+1}) + \delta_{h+1}.$$

Hence,  w.l.o.g., let us assume $F_{h}>0$.  In the following, we show that $\gamma >  F_{h} - \delta_h.$
By contradiction, let us suppose $\gamma \le  F_{h} - \delta_h $. Hence, since $ \lfloor A_h/\delta_h\rfloor \delta_h\ge F_{h}$, we have $A_h= f(\bar H_k^{h} \setminus S(\hat y^{1,h}))\ge \delta_h$. 
Then, by denoting $\rho=\sum\limits_{r \in N(k)\setminus H_k}f_r \hat y_r^{h}$, arguments  similar to those applied to Case 1 can be used to get a contradiction.
Hence, $\gamma > F_{h} -\delta_h$ and we have $\gamma>F_{h} -\delta_h \ge \min\{ \lfloor A_h/\delta_h\rfloor \delta_h;F_{h}\}  -\delta_h \ge \min\{ \lfloor  f({\cal H}_k^{h})/\delta_h \rfloor \delta_h;F_{h}\}  -\delta_h,$
 where the second inequality follows by Lemma \ref{lem:at:least_available}.
Hence, by relation \eqref{eq:A_h+1} and since by induction $A_h \le f({\cal H}_k^{h}) + \delta_h $ we have 
$$A_{h+1}=A_h - \gamma + f(H_k \cap T'_{h+1}) <
 f({\cal H}_k^{h}) + \delta_h - min\{F_{h};\lfloor f({\cal H}_k^{h})/\delta_h\rfloor \delta_h  \} + \delta_h + f(H_k \cap T'_{h+1})$$
 that, by the definition of ${\cal H}_k^{h+1}$, is equal to $f({\cal H}_k^{h+1}) + 2\delta_h \le f({\cal H}_k^{h+1}) + \delta_{h+1}, $
where  the last inequality holds since  $\delta_{h+1} \ge 2\delta_h$. Hence, $A_{h+1}<f({\cal H}_k^{h+1}) + \delta_{h+1}$, i.e., the thesis.
\end{proof}


In the following, 
given an item type $w$ and a knapsack part $h$, we denote by $T_h^w$ the set of items of type $w$ in $T'_h$.
The following theorem states the maximum and minimum value assumed by $\hat y_{k,b}^b$, for $\hat y$ in  $MO^{OO}(k,b,{\bf F})$.

\begin{theorem}\label{thm:itemb}
Given an optimal solution $\hat y$ in  $MP^{OO}(k,b,{\bf F})$, with $b>1$, the number of items of type $k$ of the set $T'_b$ assigned to the part $b$, $\hat y^{b}_{k,b}$, takes the following values:
\begin{equation}\label{eq:min_bb}
\hat y^{b}_{k,b}\geq \min\{\left(\frac{ F_b -\lceil (f({\cal H}_k^{b}) + \delta_b)/\delta_{b}\rceil \delta_{b}}{\delta_b}\right)^+ \frac{\delta_b}{f_k}, \tilde b_{k,b}  \}
\end{equation}
and
\begin{equation}\label{eq:max_bb}
\hat y^{b}_{k,b} \leq \min\{\left(\frac{ F_b -\lfloor f({\cal H}_k^{b})/\delta_{b}\rfloor \delta_{b}}{\delta_b}\right)^+ \frac{\delta_b}{f_k}, \tilde b_{k,b}  \}.
\end{equation}
\end{theorem}
\begin{proof}
Obviously, if $\tilde b_{k,b}=0$ then $y^{b}_{k,b}=0$ in any optimal solution of $MP^{OO}(k,b,{\bf F})$, and the theorem holds. Hence, suppose that $\tilde b_{k,b}>0$, i.e., items of type $k$ exist in $T'_b$.\\
Observe that, by definition of $b$, $\delta_b=d_b$.
Let $\hat H_k^{b}$ be the set of items of $H_k$ not assigned by $\hat y$ to $1,\ldots, b-1$ and that can be assigned to the part $b$ (i.e., the items contained in $\hat H_k^{b}\equiv H_k \cap \{T'_1\cup \dots T'_{b}\}\setminus S(\hat y^{1,b-1})\equiv \bar H_k^{b}\setminus S(\hat y^{1,b-1})$).
As in \cite{Pochet98}, several cases are considered.\\
Case 1. $F_{b} \le  f({\cal H}_k^{b})$.  Since  $\delta_b$ divides $F_{b}$, it follows that $F_{b} \le   \lfloor f({\cal H}_k^{b})/\delta_b\rfloor \delta_b$. In this case, the Theorem states that $\hat y^{b}_{k,b}$=0.
Observe that, by Lemma \ref{lem:at:least_available}, $f(\hat H_k^{b}) \ge  \lfloor f({\cal H}_k^{b})/\delta_b\rfloor \delta_b  $. Hence, $F_{b} \le  f(\hat H_k^{b})$.\\
By contradiction, suppose that
$\hat y^{b}_{k,b}>0$.
Since $F_{b} \le f(\hat H_k^{b})$, we have $\sum\limits_{q=1}^{b}\sum\limits_{r \in \hat H_k^{b}} f_{r}  \hat y^{b}_{r,q} + f_k  \hat y^{b}_{k,b} \leq F_{b} \leq
f(\hat H_k^{b})= \sum\limits_{r \in \hat H_k^{b}}f_r\sum\limits_{q=1}^{b}( \tilde b_{r,q}-  \hat y^{1,b-1}_{r,q})$,
and, since $\hat y^{b}_{k,b}>0$, $\sum\limits_{r \in \hat H_k^{b}} f_r\sum\limits_{q=1}^{b}(\tilde b_{r,q}-\hat y^{1,b-1}_{r,q} -
 \hat y^{b}_{r,q})\geq f_k  \hat y^{b}_{k,b}>0$. Observe that,
 by Lemma \ref{prop:M-SP-1}, $\delta_b$ divides $f_k\hat y^{b}_{k,b}$. 
 Moreover, observe that: ($i$) $\delta_b$ is the biggest item size; ($ii$)  $\delta_q$ divides  $f_r\tilde b_{r,q}$, for all $q$ and $r$ (by definition); ($iii$)  $\delta_q$ divides  
  $f_r\hat y^{h}_{r,q}$, for all $r=1,\ldots,k$, $q=1,\ldots,b$ and $h=q,\ldots,l$ (by Lemma \ref{prop:M-SP-1}).
Hence, since $\delta_1, \delta_2,\ldots, \delta_b$ are divisible, integers $\lambda_{r,q} \in \{0,\ldots,
\tilde b_{r,q}- \hat y^{q,b}_{r,q} \}$ for all items types  $r$ in $\hat H_k^{b}$ and $q=1,\ldots,b$ exist, such that $\delta_q$ divides $f_r \lambda_{r,q} $ and 
$\sum\limits_{q=1}^b\sum\limits_{r \in  \hat H_k^{b}} f_r \lambda_{r,q} =\sum\limits_{q=1}^b\sum\limits_{r \in  \hat H_k^{b}}\lceil f_r\lambda_{r,q} /\delta_q \rceil \delta_q  = f_k
\hat y^{b}_{k,b}$.\\
Consider now a new solution $\bar y^{1,b}$ such that: $\bar y^{1,b-1}={\hat y}^{1,b-1}$,
$\bar y^{b}_{r,q}=\lambda_{r,q}$ for $r \in \hat H_k^{b}$ and $q=1,\ldots,b$, and
$\bar y^{b}_{k,b}=0$. Note that, by construction and by the above observations, $\bar y^{1,b}$ is feasible and, by
definition of $\hat H_k^{b}$,  has objective function
strictly better than $\hat y^{1,b}$, i.e., $\hat y^{1,b}$ is not an OPT solution, a contradiction.

Case 2. 
$F_{b}- \lceil (f({\cal H}_k^{b}) + \delta_b)/\delta_b \rceil \delta_b \ge f_k \tilde b_{k,b}$. 
Then
 by \eqref{eq:min_bb} and \eqref{eq:max_bb}, 
$\hat y_{k,b}^{b}=\tilde b_{k,b}$. Note that, since by Lemma \ref{lem:minimum2} $f({\hat H}_k^{b})\le  f({\cal H}_k^{b}) + \delta_b$, it follows that $F_{b}- f({\hat H}_k^{b}) \ge f_k \tilde b_{k,b}$, too.\\
By contradiction, suppose that $\hat y_{k,b}^{b}<\tilde b_{k,b}$. By Lemma \ref{prop:M-SP-1},   $\delta_b$ divides $f_k\hat y_{k,b}^{b}$, and since $\delta_b$ divides $f_k\tilde b_{k,b}$ (by definition), it follows that $f_k\tilde b_{k,b} -f_k\hat y_{k,b}^{b} 
=\alpha \delta_b$,  with $\alpha$ positive integer. Hence,  there exists a set of items of type $k$ that are not assigned in $\hat y^b$, belonging to $T^k_b$ and having total size $\alpha \delta_b$. 
Since $\hat y$ is an ordered solution, 
a set $A$ of items belonging to $T_1^k\cup \ldots\cup T_{b-1}^k$, of total size $f(A)\ge \delta_b$ and assigned to the part $b$ in $\hat y$ does not exist. (Otherwise, in $\hat y$, an  unassigned set of items in  $T_b^k$, with total size $\delta_b$, can be feasibly allocated (by Lemma \ref{lem:at:least_available}) to the part $b$ in place of a subset of $A$ of total size $\delta_b$, i.e., $\hat y$ is not an ordered solution.) Let $\Omega$ be the set of items (of type $k$)  belonging to $T_1^k\cup \ldots\cup T_{b-1}^k$ and assigned to part $b$ in $\hat y$ (i.e., 
$\Omega=\{S(\hat y_{k1}^{b}) \cup  \ldots \cup S(\hat y_{kb-1}^{b})\} )$. By the above discussion, it follows that  $f(\Omega)< \delta_b$.
 Moreover, let $\Gamma= \{T'_1\cup\ldots\cup T'_b\}\setminus \{ T_{b}^k\cup H_k\}$, i.e., the set of items in  $\{T'_1\cup\ldots\cup T'_b\}$ that neither belong to $T_{b}^k$ nor to $H_k$. 
Let $\rho=\sum\limits_{q=1}^b\sum\limits_{r \in \Gamma}f_r \hat y_{r,q}^{b}$ (i.e., $\rho$ is the total size of the items in $\Gamma$ assigned to the part $b$ in $\hat y$). 
If $\rho < \alpha \delta_b$, since $F_{b}- f({\hat H}_k^{b}) \ge f_k \tilde b_{k,b}$, the solution $\bar y^{1,b}$ with $\bar y^{1,b-1}=\hat y^{1,b-1}$, $\bar y^{b}_{r,q}=0$ for $r \in
\Gamma$ and $q=1,\ldots,b$, $\bar y^{b}_{r,q}=\hat y^{b}_{r,q}$ for $r \in \hat H_k^{b}$ and $q=1,\ldots,b$,
and $\bar y^{b}_{k,b}=\tilde b_{k,b}$ is feasible
and yields a bigger objective
function value than $\hat y^{1,b}$, since  $\frac{p_k}{f_k}\geq \frac{p_r}{f_r}$ for
all $r \in \Gamma$ (recall that the items in $\Gamma$ do not belong to $H_k$). Hence, $\hat y^{1,b}$ is not an OPT solution. A contradiction. Hence, it must be $\rho \ge \alpha \delta_b$.

Recall that, by Lemma \ref{prop:M-SP-1}, $\delta_q$ divides $ f_r \hat y_{r,q}^b$, for all item types $r$ and $q=1,\ldots,b$.
Observe that, since $\hat y$ is an ordered solution, a subset   $\Theta$ of $\Gamma \cap \{ T'_1 \cup \ldots \cup T'_{b-1}\}$  such that 
$\sum\limits_{q=1}^b\sum\limits_{r \in
\Theta} \lceil f_r \hat y_{r,q}^b/\delta_q \rceil \delta_q= \sum\limits_{q=1}^b\sum\limits_{r \in
\Theta}  f_r \hat y_{r,q}^b=\delta_b$ and $p(\Theta)=p_k\delta_b/f_k$ does not exist. (Otherwise, 
a set of  items belonging to $T^k_b$ and having total size $ \delta_b$  that are not assigned or that are assigned to the parts $b+1,\ldots,l$ in $\hat y$ (such a set exists in this case) can be assigned in place of $\Theta$, i.e., $\hat y$ is not an ordered solution.) 

Let us consider the items in $ r \in \Gamma \cap  \{ T'_1 \cup \ldots \cup T'_{b-1}\}$ such that $\frac{p_r}{f_r}= \frac{p_k}{f_k}$. Hence, since $\Gamma$ does not contain items $j$ such that ${p_j \over f_j}>{p_k \over f_k}$, then 
\begin{equation}\label{eq3<d_b}
\sum\limits_{r\in \Gamma \cap  \{ T'_1 \cup \ldots \cup T'_{b-1}\}: \frac{p_r}{f_r}= \frac{p_k}{f_k}} \lceil f_r \hat y_{r,q}^b/\delta_q \rceil \delta_q <\delta_b
\end{equation}

Let $\Delta$  be the set containing all the items of types $1,\ldots,k-1$ in $ \{ T'_1 \cup \ldots \cup T'_{b}\}$.
By the partition into
maximal blocks (see \eqref{proper:maxblock}) and since $\Delta$ is strictly contained in the set of the items of types $1,\ldots,k-1$ in $N(k)$, we have $\delta_b \ge f_k>\sum\limits_{r\in \Delta: \frac{p_r}{f_r}= \frac{p_k}{f_k}}\sum\limits_{q=1}^{b}\tilde b_{r,q}$. Obviously, since $\Delta \setminus H_k \subseteq \Delta $,  we have $\delta_b \ge f_k>\sum\limits_{r\in  \Delta \setminus H_k: \frac{p_r}{f_r}= \frac{p_k}{f_k}}f_r
\sum\limits_{q=1}^{b}\tilde b_{r,q}$, too. Furthermore, since by definition $\Gamma \cap  \Delta\subseteq \Delta \setminus H_k$, it also holds that 
\begin{equation}\label{eq2<d_b}
\delta_b \ge f_k>\sum\limits_{r\in \Gamma \cap  \Delta: \frac{p_r}{f_r}= \frac{p_k}{f_k}}f_r
\sum\limits_{q=1}^{l}\tilde b_{r,q}.
\end{equation}
Recall that (by Proposition \ref{prop:M-SP-1}) if an item $r \in T'_b$ is assigned to the part $b$ in $\hat y$, then  $f_r \hat y_{rb}^b=\beta \delta_b$, with $\beta$ positive integer. As a consequence and by \eqref{eq2<d_b}, items in $T'_b \cap  \Delta\cap \Gamma$ such that $\frac{p_r}{f_r}= \frac{p_k}{f_k}$ do not exist.
By \eqref{eq3<d_b} and \eqref{eq2<d_b}, it follows that 
 \begin{equation}\label{eq<d_b}
 \sum\limits_{q=1}^b \sum\limits_{r \in
\Gamma : \frac{p_r}{f_r}= \frac{p_k}{f_k}} \lceil f_r \hat y_{r,q}^b/\delta_q \rceil \delta_q= \sum\limits_{q=1}^b \sum\limits_{r \in
\Gamma : \frac{p_r}{f_r}= \frac{p_k}{f_k}} f_r \hat y_{r,q}^b < \delta_b.
\end{equation}

By the divisibility of the item sizes, since $\rho \ge \alpha \delta_b$ and all the items have size not bigger than $\delta_b$, and by Lemma \ref{prop:M-SP-1},  integers
$\lambda_{r,q} \in \{0,\ldots, \hat y_{r,q}^{b}\}$ exist for all $r \in \Gamma$ and $q=1,\ldots,b$, such that $\sum\limits_{q=1}^b\sum\limits_{r \in
\Gamma} f_r \lambda_{r,q}= \sum\limits_{q=1}^b\sum\limits_{r \in
\Gamma}  \lceil f_r \hat y_{r,q}^b/ \delta_q \rceil \delta_q = f_k
(\tilde b_{k,b}-\hat y^{b}_{k,b} )=\alpha \delta_b$ (where  $\alpha \ge1$ and integer). As a consequence and
by \eqref{eq<d_b}, it follows that $\sum\limits_{q=1}^b\sum\limits_{r \in \Gamma} p_r \lambda_{r,q}<p_k(\tilde b_{k,b}-\hat y^{b}_{k,b})$.
Hence, 
the solution $\bar y^{1,b}$ with $\bar y^{1,b-1}=\hat y^{1,b-1}$, $\bar y^{b}_{r,q}=\hat y^{b}_{r,q}- \lambda_{r,q}$ for $r \in \Gamma$ and $q=1,\ldots,b$, $\bar y^{b}_{r,q}=\hat y^{b}_{r,q}$ for $r \in \hat H_k^{b}$ and $q=1,\ldots,b$,
and $\bar y^{b}_{k,b}=\tilde b_{k,b}$ is feasible, and
yields a bigger objective
function value than $\hat y^{1,b}$, i.e., $\hat y^{1,b}$ is not an OPT solution, a contradiction.\\
Case 3.
If Cases 1 and 2 do not hold, then $F_{b}  > f({\cal H}_k^{b})$
and $F_{b}  -\lceil (f({\cal H}_k^{b}) + \delta_b)/\delta_b\rceil  \delta_b < f_k \tilde b_{k,b}.$
 Hence, the theorem states that $\hat y^{b}_{k,b} \ge \left(\frac{ F_b -\lceil (f({\cal H}_k^{b}) + \delta_b)/\delta_{b}\rceil \delta_{b}}{\delta_b}\right)^+ \frac{\delta_b}{f_k}$ and $\hat y^{b}_{k,b} \leq \min\{\left(\frac{ F_b -\lfloor f({\cal H}_k^{b})/\delta_{b}\rfloor \delta_{b}}{\delta_b}\right)^+ \frac{\delta_b}{f_k}, \tilde b_{k,b}  \}$. 
Also in
this case, by contradiction, suppose that the theorem does not hold.
Hence, either
$(a)$  $\hat y^{b}_{k,b} < \left(\frac{ F_b -\lceil (f({\cal H}_k^{b}) + \delta_b)/\delta_{b}\rceil \delta_{b}}{\delta_b}\right)^+ \frac{\delta_b}{f_k}$ or $(b)$ $\hat y^{b}_{k,b} >
\min\{\left(\frac{ F_b -\lfloor f({\cal H}_k^{b})/\delta_{b}\rfloor \delta_{b}}{\delta_b}\right)^+ \frac{\delta_b}{f_k}, \tilde b_{k,b}  \}$.\\ 
 In Case $(a)$, since $F_{b}  -\lceil (f({\cal H}_k^{b}) + \delta_b)/\delta_b\rceil  \delta_b < f_k \tilde b_{k,b}$, we have 
 $\lceil (f({\cal H}_k^{b}) + \delta_b)/\delta_{b}\rceil \delta_{b} + \left(\frac{ F_b -\lceil (f({\cal H}_k^{b}) + \delta_b)/\delta_{b}\rceil \delta_{b}}{\delta_b}\right)^+ \delta_b \le F_b$. Hence,   an argument  similar
to that used in Case 2 can be applied  to show that there exists a feasible solution $\bar y$ with $\hat y^{b}_{k,b} = \left(\frac{ F_b -\lceil (f({\cal H}_k^{b}) + \delta_b)/\delta_{b}\rceil \delta_{b}}{\delta_b}\right)^+ \frac{\delta_b}{f_k}$. \\
In Case $(b)$, we obviously have $\hat y^{b}_{k,b} \le \tilde b_{k,b}$ and the theorem is not satisfied only if 
$\left(\frac{ F_b -\lfloor f({\cal H}_k^{b})/\delta_{b}\rfloor \delta_{b}}{\delta_b}\right)^+ \frac{\delta_b}{f_k}< \tilde b_{k,b} $ and 
$\hat y^{b}_{k,b} >\left(\frac{ F_b -\lfloor f({\cal H}_k^{b})/\delta_{b}\rfloor \delta_{b}}{\delta_b}\right)^+ \frac{\delta_b}{f_k}$. 
Since, by Lemma \ref{lem:at:least_available}, $f(\hat H_k^{b}) \ge  \lfloor f({\cal H}_k^{b})/\delta_{b}\rfloor \delta_{b}$, it follows that 
$\hat y^{b}_{k,b} >\left(\frac{ F_b -\lfloor f(\hat H_k^{b})/\delta_{b}\rfloor \delta_{b}}{\delta_b}\right)^+ \frac{\delta_b}{f_k}$, too.
Hence, since $\hat y$ is feasible and by algebra, we have $\sum\limits_{q=1}^{b}\sum\limits_{r \in \hat H_k^{b}} f_{r}  \hat y^{b}_{r,q} + f_k  \hat y^{b}_{k,b}  \leq F_{b}\le f(\hat H_k^{b})+ F_{b}-\lfloor f({\hat H}_k^{b})/\delta_b\rfloor \delta_b$ (the last inequality follows since  $f(\hat H_k^{b})-\lfloor f({\hat H}_k^{b})/\delta_b\rfloor \delta_b\ge 0$). Then a
similar argument to that of Case 1 can be applied  to get a contradiction.
\end{proof}

Recalling   Lemma \ref{prop:M-SP-1} (stating  that $\delta_b$ divides $ f_k \hat y^{b}_{k,b}$), Theorem \ref{thm:itemb} provides the at most three possible values assumed by $ \hat y^{b}_{k,b}$,
for $\hat y$ in  $MO^{OO}(k,b,{\bf F})$.
 Recalling that  $T'_{b+1}\cup T'_{b+2} \cup \ldots \cup T'_{l} = \emptyset$, Theorem \ref{thm:itemb} can be also applied to determine the possible values assumed by $\hat y^{b+1}_{k,b},\ldots, \hat y^{l}_{k,b}$, as explained in the following. For each value of $ \hat y^{b}_{k,b}$ specified by Theorem \ref{thm:itemb} (such that $\delta_b$ divides $ f_k \hat y^{b}_{k,b}$)  we can generate a new instance in which we set 
 $\tilde b_{k,b}=\tilde b_{k,b}- \hat y^{b}_{k,b}$ and $F_b=F_b - f_k \hat y^{b}_{k,b}$, $T_b^k=T_b^k \setminus S(\hat y^{b}_{k,b})$, i.e., $N(k)=N(k)\setminus S(\hat y^{b}_{k,b})$.  Observe that, by  Lemma \ref{prop:M-SP-1},  $\delta_b$ still divides $F_b$. Theorem \ref{thm:itemb}  can be now  applied to the new instance to  determine all the possible values assumed by  $\hat y^{b+1}_{k,b}$. Applying recursively the above argument, we can compute all the possible values assumed by $\hat y^{b+2}_{k,b},\ldots,\hat y^{l}_{k,b}$.
 In the following theorem, we show that the values assumed by $\sum\limits_{h=b}^l \hat y^{h}_{k,b}$ are at most three.
 
 \begin{theorem}\label{thm:itembk(hl)}
 Given an optimal solution $\hat y$ in  $MP^{OO}(k,b,{\bf F})$, with $b>1$, and an item type $k$, we have
\begin{equation}\label{eq:min_bbhl}
\sum\limits_{h=b}^l \hat y^{h}_{k,b} \geq \min\{\left(\frac{ (\sum\limits_{h=b}^l F_h) -\lceil (f({\cal H}_k^{b}) + \delta_b)/\delta_{b}\rceil \delta_{b}}{\delta_b}\right)^+ \frac{\delta_b}{f_k}, \tilde b_{k,b}  \}
\end{equation}
and
\begin{equation}\label{eq:max_bbhl}
\sum\limits_{h=b}^l \hat y^{h}_{k,b} \leq \min\{\left(\frac{ (\sum\limits_{h=b}^l F_h)  -\lfloor f({\cal H}_k^{b})/\delta_{b}\rfloor \delta_{b}}{\delta_b}\right)^+ \frac{\delta_b}{f_k}, \tilde b_{k,b}  \}.
\end{equation}
 \end{theorem}
\begin{proof}
Since, by  Lemma \ref{prop:M-SP-1},  $\delta_b$ divides $\delta_b \hat y^{h}_{k,b}$, it follows that $\delta_b$ divides $\delta_b \sum\limits_{h=b}^l \hat y^{h}_{k,b}$, too. Hence, if  \eqref{eq:min_bbhl} and \eqref{eq:max_bbhl} holds, we have that $\sum\limits_{h=b}^l \hat y^{h}_{k,b}$ takes at most three values.\\
To prove conditions \eqref{eq:min_bbhl} and \eqref{eq:max_bbhl}, the following observations are in order:\\
1) All items in M-SP$(k,b,{\bf F})$ belong to $T'_1\cup\dots\cup T'_b$ (i.e., $T'_{b+1}=\dots =T'_l=\emptyset$).\\ 
2) Each item in  M-SP$(k,b,{\bf F})$ can be assigned to all the knapsack parts $h$ in $\{b,b+1,\ldots,l\}$.\\
3) $\delta_b$ divides the part capacities $F_b,F_{b+1},\ldots,F_l$.\\
Hence, w.l.o.g, we can replace the knapsack parts $\{b,b+1,\ldots,l\}$ by a single part, say $b$, with capacity $\sum\limits_{h=b}^l F_h$.\\ 
Note that, Lemmas \ref{lem:at:least_available} and \ref{lem:minimum2} still hold providing that the minimum and maximum total size of items in $H_k$ available to be assigned to the (new) part $b$ in $\hat y$ are $\lfloor f({\cal H}_h^b)/\delta_b\rfloor\delta_b$ and $f({\cal H}_h^b)+\delta_b$, respectively. Hence, the thesis follows by applying Theorem \ref{thm:itemb}.
\end{proof}

Now, we show how the possible values assumed by $ \hat y^{b}_{k-1b}$ (i.e., the number of items of type $k-1$ in $T'_b$ assigned to part $b$ in $\hat y$) can be determined, too. For each value of $\hat y^{h}_{k,b}$, for $h=b,\ldots,l$, specified by Theorem \ref{thm:itemb}, 
we  consider a new M-SP$(k,b,{\bf F})$ instance in which $F_{h}=F_{h}- f_k\hat y^{h}_{k,b}$, for $h=b,\ldots,l$, ($F_{h}=F_{h}$, for $h=1,\ldots,b-1$) and  $N(k)=N(k)\setminus T_{b}^k$.  
Note that, by Lemma  \ref{prop:M-SP-1}, $F_{h}$, for $h=b,\ldots,l$,  is still multiple of $\delta_b$, and since $\delta_b$ is the biggest item size of the instance, $F_{b},F_{b+1},\ldots,F_{l}$ are multiple of $\delta_{b-1}, \delta_{b-1},\ldots,\delta_1$, too. Furthermore, observe that  $\bar H_{k-1}^b$, defined as in \eqref{def: bar H}, with $k-1$ in place of $k$, may contain items of type $k$ (belonging to $T'_1\cup\ldots \cup T'_{b-1}$ only), but this fact does not affect the definition of ${\cal H}_{k-1}^b$  (defined as in \eqref{eq:cal H}, with $k-1$ in place of $k$), since all items in the instance have size not bigger than $\delta_b$. Hence, 
Lemmas \ref{lem:at:least_available} and \ref{lem:minimum2} still hold, and allow to determine the minimum and maximum total size of items in  
$\bar H_{k-1}^b \setminus S(\hat y^{1,b-1})$. 
Theorem \ref{thm:itembk-1} sets the possible values taken by $ \hat y^{b}_{k-1b}$.



 \begin{theorem}\label{thm:itembk-1}
Given an optimal solution $\hat y$ in  $MP^{OO}(k,b,{\bf F})$, $b>1$, let $F_{h}=F_{h}- f_k\hat y^{h}_{k,b}$, for $h=b,\ldots,l$, and $N(k)=N(k)\setminus T_{b}^k$. Then, $\hat y^{b}_{k-1b}$  takes the following values.
\begin{equation}\label{eq:min_bb2}
\hat y^{b}_{k-1b}\geq \min\{\left(\frac{ F_b -\lceil (f({\cal H}_{k-1}^{b})+\delta_b)/\delta_{b}\rceil \delta_{b}}{\delta_b}\right)^+ \frac{\delta_b}{f_{k-1}}, \tilde b_{k-1b}  \}
\end{equation}
and
\begin{equation}\label{eq:max_bb2}
\hat y^{b}_{k-1b} \leq \min\{\left(\frac{ F_b -\lfloor f({\cal H}_{k-1}^{b})/\delta_{b}\rfloor \delta_{b}}{\delta_b}\right)^+ \frac{\delta_b}{f_{k-1}}, \tilde b_{k-1b}  \}.
\end{equation}

\end{theorem}
\begin{proof}
The same cases considered in the proof of Theorem  \ref{thm:itemb} are used. 
Observe that, by Lemma  \ref{prop:M-SP-1},   $\delta_b$ still divides  $F_{b}= F_{b}-f_k \hat y^{b}_{k,b}$.
Let $\hat H_{k-1}^{b}$ be the set of items of $H_{k-1}$ (defined in \eqref{eq:Hk}) not assigned by $\hat y$ to the parts $1,\ldots, b-1$ and that can be assigned to the part $b$ (i.e.,  $\hat H_{k-1}^{b}\equiv H_{k-1} \cap \{T'_1\cup \dots T'_{b}\}\setminus S(\hat y^{1,b-1})\equiv \bar H_{k-1}^{b}\setminus S(\hat y^{1,b-1})$). Obviously, since by the hypothesis $N(k)$ does not contain $T_{b}^k$, by definition, $H_{k-1}$, $\bar H_{k-1}^{b}$ and $\hat H_{k-1}^{b}$ do not contain items in $T_{b}^k$, too.\\
Case 1. $F_{b} \le  f({\cal H}_{k-1}^{b})$. Theorem states that
$\hat y^{b}_{k-1b}$=0. The same argument employed in Case 1 of Theorem  \ref{thm:itemb} can be used to show the thesis.\\
Case 2. 
$F_{b}- \lceil f({\cal H}_{k-1}^{b})+\delta_b)/\delta_b \rceil \delta_b \ge f_k \tilde b_{k-1b}$. Then
 by \eqref{eq:min_bb} and \eqref{eq:max_bb}, 
$\hat y_{k-1b}^{b}=\tilde b_{k-1b}$. Also in this case, the same argument employed in Case 2 of Theorem  \ref{thm:itemb} can be used to show the thesis. In particular, note that, the set $\Gamma$, now defined as 
  $\Gamma=N(k)\setminus \{ T_{b}^{k-1}\cup H_{k-1}\}$,  may contain items of type $k$ belonging $\{T'_1\cup\ldots T'_{b-1}\}$. However, since $\hat y$ is an ordered solution, a subset of  $\Theta$ of $\Gamma \cap \{ T'_1 \cup \ldots \cup T'_{b-1}\}$  such that 
$\sum\limits_{q=1}^b\sum\limits_{r \in
\Theta} \lceil f_r \hat y_{r,q}^b/\delta_q \rceil \delta_q= \sum\limits_{q=1}^b\sum\limits_{r \in
\Theta}  f_r \hat y_{r,q}^b=\delta_b$ and $p(\Theta)=p_{k-1}\delta_b/f_{k-1}$ does not exist also in this case.
A consequence of this fact is that Inequality \eqref{eq<d_b} still holds. Hence, the thesis follows by applying the same arguments employed in Case 2 of Theorem  \ref{thm:itemb}.\\
 Case 3. See Case 3 of Theorem  \ref{thm:itemb}.
\end{proof}
 
 Hence, recursively applying Theorems  \ref{thm:itemb} and \ref{thm:itembk-1}, all the possible values assumed by $ \hat y^{h}_{w,b}$ for $h=b,\ldots,l$ and $w=1,\ldots,k$ can be detected. Moreover, by recursively applying Theorem \ref{thm:itembk(hl)}, it follows that the values assumed by $\sum\limits_{h=b}^{l}\hat y^{h}_{w,b}$ are at most three, for $w=1,\ldots,k$.
 
Given an optimal solution  $\hat y$ in $MP^{OO}(k,b-1,{\bf F})$, let ${\bf G}$ be the vector with components $G_{h}=F_{h}- \sum\limits_{w=1}^kf_k\hat y^{h}_{w,b}$, for $h=b,\ldots,l$, $G_{h}=F_{h}$, for $h=1,\ldots,b-1$, and let $N(k)=N(k)\setminus T'_{b}$. Now, Theorems \ref{thm:itemb} and \ref{thm:itembk-1} can be applied on the   instance  
  M-SP($k,b-1,{\bf G}$)  
  for finding the values attained by $\hat y^{h}_{w,b-1}$, for $h=b-1,\ldots,l$ and $w=1,\ldots,k$ (with  $\hat y$ optimal solution in $MP^{OO}(k,b-1,{\bf G})$). And so on.

 {\bf Example}
 
 Consider the following instance of MKSP with  $|N|=6$  item types and $|M|=3$  knapsacks.
 \begin{eqnarray*}\label{exp:1}
 \max  \sum\limits_{i=1}^3 (4x_{i,1}+28x_{i,2}+ 15x_{i,3}+ 14 x_{i,4}+ 28 x_{i,5}+32x_{i,6})\\
 x_{1,1}+ 2 x_{1,2}+ 2 x_{1,3 }+ 2 x_{1,4  }+4 x_{1,5 }+ 4 x_{1,6}\le 7\\
 x_{2,1}+ 2 x_{2,2}+ 2 x_{2,3 }+ 2 x_{2,4 }+4 x_{2,5 }+ 4 x_{2,6}\le 2\\
 x_{3,1}+ 2 x_{3,2}+ 2 x_{3,3 }+ 2 x_{3,4  }+4 x_{3,5 }+ 4 x_{3,6} \le 6\\
 \sum\limits_{i=1}^3 x_{i,1}\le 2; \; \sum\limits_{i=1}^3 x_{i,2}\le 4; \; \sum\limits_{i=1}^3 x_{i,3}\le 8; \; \sum\limits_{i=1}^3 x_{i,4}\le 7; \; \sum\limits_{i=1}^3 x_{i,5}\le 2; \;\sum\limits_{i=1}^3 x_{i,6}\le 1\\
 x \in \mathbb{Z}^+\cup\{0\} \;\text{ for } i =1,\ldots,3,\; j =1,\ldots, 6 
 \end{eqnarray*}
In the instance, 
3 six different sizes exist (i.e., 1, 2, 4), hence $l=3$. To formulate the problem as in  \eqref{newform}, we have to compute the quantities $r_i^h$, for $h=1,2$ and $i=1,2,3$. 
Hence, we have 
$\sum\limits_{i=1}^3 r_i^1=1+0+0$, $\sum\limits_{i=1}^3 r_i^2=2+2+2=6$,  and  $\sum\limits_{i=1}^3 (c_i-\sum\limits_{h=1}^3 r_i^h)=4+0+4=8$.

The corresponding M-$SP$ is defined as follows.
The set of items in $N$ can be partitioned into the 5 maximal blocks $B_1=\{1\}$, $B_2=\{2\}$, $B_3=\{3\}$, $B_4=\{4,5\}$ and $B_5=\{6\}$, with multiplicity $\tilde b_{1,1}=2$, $\tilde b_{2,2}=4$, $\tilde b_{3,2}=8$, $\tilde b_{4,2}=7$ and $\tilde b_{4,3}=4$, and, finally,  $\tilde b_{5,3}=2$, respectively (all other $\tilde b_{p,q}$s are 0). 
Accordingly, the sets $T'_q$ contain the  following item types: $T'_1=\{1\}$ ,
$T'_2=\{2,3,4\}$ and $T'_3=\{4,5\}$.

Hence, in M-$SP$, the  item set is $N'=\{1,2,3,4,5\}$,  the profits $p_1,\ldots,p_5$ are equal to 4, 28, 15, 14 and 32, respectively, and the sizes $f_1,\ldots,f_5$ are equal to 1, 2, 2, 2 and 4, respectively. The ratios $p_i/f_i$, for $i=1,\ldots,5$, are 4, 14, 15/2, 7 and 8 respectively.

The knapsack has 3 parts of capacity $\bar c_1=\sum\limits_{i=1}^3 r_i^1=1$,  $\bar c_2=\sum\limits_{i=1}^3 r_i^2=6$ and $\bar c_3=\sum\limits_{i=1}^3 (c_i-\sum\limits_{h=1}^2 r_i^h)=8$.
 In what follows, $MO(k,b,(F_1,F_2,F_3))$ denotes the set $MO(k,b,{\bf F})$, where ${\bf F}$ has components $F_1$, $F_2$, and $F_3$.

 


$MO(5,3,(1,6,8))$:
Since, $T'_3=\{4,5\}$, by Theorems \ref{thm:itemb}--\ref{thm:itembk-1}, we can compute $y_{5,3}^3$ and $y_{4,3}^3$. At this aim, we need to calculate  $f({\cal H}_5^3)$ and $f({\cal H}_4^3)$. We have $\bar H_5^1=\emptyset$ and $\bar H_5^2=\bar H_5^3=\{2\}$.
Hence, $f({\cal H}_5^1)=0$, $f({\cal H}_5^2)=f(\bar H_5^2)=8$, 
$f({\cal H}_5^2(min))=\min\{\left\lfloor f({\cal H}_5^2)/2 \right\rfloor 2; F_2\}=6$,  $f({\cal H}_5^3)= f(\bar H_5^3)-f({\cal H}_5^2(min))=2$,  

We have $H_4=\{2,3,5\}$ and $\bar H_4^1=\emptyset$, $\bar H_4^2=\bar H_4^3=\{2,3\}$. Hence, $f({\cal H}_4^1)=0$, $f({\cal H}_4^2)=f(\bar H_4^2)=24$, $f({\cal H}_4^2(min))=\min\{\left\lfloor f({\cal H}_4^2)/2 \right\rfloor 2; F_2\}=6$,  $f({\cal H}_4^3)= f(\bar H_4^3)-f({\cal H}_4^2(min))=18$.

By Theorem \ref{thm:itemb}, since $F_3=8$ and $f({\cal H}_5^3)=2$, we have $y_{5,3}^3\in \{0,1,2\}$.  By Theorem \ref{thm:itembk-1}, since $f({\cal H}_4^3)=18$, it follows that $y_{4,3}^3=0$ in every optimal solution. Since no other item type exists in the set $T'_3$, the new sets to consider are $MO(5,2,(1,6,8))$, $MO(5,2,(1,6,4))$ and $MO(5,2,(1,6,0))$.

$MO(5,2,(1,6,F_3))$: $T'_2=\{2,3,4\}$. 
We have $H_3=\{2,5\}$, $\bar H_3^1=\emptyset$, $\bar H_3^2=\{2\}$ and $f({\cal H}_3^2)=f(\bar H_3^2)=8$. 
Recalling that $f({\cal H}_4^2)=24$ and $F_2=6$, since $f({\cal H}_3^2)=8$, we have that $y_{4,2}^2=y_{3,2}^2=0$ in every optimal solution. Moreover, by Theorem 
\ref{thm:itembk(hl)} and since $f({\cal H}_4^2)=24$ and $F_2+F_3  \le 14$, it follows that $y_{4,2}^2+y_{4,2}^3=0$, i.e., $y_{4,2}^3=0$ too.  
By Theorem \ref{thm:itembk(hl)} since $6\le F_2+F_3\le 14$ (with $F_2=6$),  $f({\cal H}_3^2)=8$ and $\tilde b_{3,2}=8$, we also have $y_{3,2}^3\in \{(F_3-4)^+/2,(F_3-2)^+/2\}$ (recall that $y_{3,2}^2=0$).\\
Since $H_2=\emptyset$ and  $f({\cal H}_2^2)=0$,   $y_{2,2}^2\in \{2,3\}$.\\ 
When $y_{3,2}^3= (F_3-4)^+/2$, then, by Theorem 
\ref{thm:itembk(hl)}, $y_{2,2}^2+y_{2,2}^3=4$. 
This leads to the new set $MO(5,1,(1,F_2,F_3))$, with $F_2+F_3=2$. 
When $y_{3,2}^3= (F_3-2)^+/2$, then, by Theorem 
\ref{thm:itembk(hl)}, $y_{2,2}^2+y_{2,2}^3\in \{3,4\}$. When  $y_{2,2}^2+y_{2,2}^3=3$, we have the new set $MO(5,1,(1,F_2,F_3))$, with $F_2+F_3=2$. When $y_{2,2}^2+y_{2,2}^3=4$ we have to consider the set $MO(5,1,(1,0,0))$. 


$MO(5,1,(1,F_2,F_3))$, with $F_2+F_3=2$; $MO(5,1,(1,0,0))$: Since $T_1'=\{1\}$ and $\tilde b_{1,1}=2$, 
by Lemma \ref{lem:solk=1}, we have  $y_{1,1}^1=1$ and $y_{1,1}^2+y_{1,1}^3=1$ in $MO(5,1,(1,F_2,F_3))$, and $y_{1,1}^1=1$ and $y_{1,1}^2+y_{1,1}^3=0$ in $MO(5,1,(1,0,0))$.




\section{Description of the convex hull of the feasible solutions of SMKP}\label{sec:poly}
  
In this section, a description of the convex hull of the solutions of MSKP that are ordered and that satisfy the OPT property, i.e., $P^{OO}$, is provided.
At this aim, we use the same approach employed in \cite{Pochet98} and first consider the set $MP^{OO}(k,b,{\bf F})$.

Given 
M-SP($k,b,{\bf F}$), let $I(k,b,{\bf F})$ be  a set of inequalities  satisfying the three following conditions:\\
1)  If  $F'_h=F_h$, for $h=1,\ldots, b-1$, and  if $d_b$  divides $F_h$ and  $F_h'$ for $h=b,\ldots, l$, then for each inequality $I$  in $I(k,b,{\bf F})$ there exists one in $I(k,b,{\bf F'})$ with the same left-hand side of $I$.\\
2) The inequalities in $I(k,b,{\bf F})$ are valid for $MP^{OO}(k,b,{\bf F})$.\\
3) Each solution in $MO^{OO}(k,b,{\bf F})$ is contained in (at least) a face induced by an inequality in $I(k,b,{\bf F})$.

We  first show that, when $b=1$,   $I(k,1,{\bf F})$ contains a single inequality.  
By Lemma \ref{lem:solk=1}, the following lemma directly follows. 
\begin{lemma}\label{lem:ineqk=1}
The following inequality  satisfies conditions $1)$, $2)$ and $3)$ for $MP^{OO}(k,1,{\bf F})$.
\begin{equation}\label{eq:MP^{OO}(k,1,F)}
\sum\limits_{w=1}^{k} \sum\limits_{h=1}^{l}y^h_{w,1}\le \sum\limits_{w=1}^{k} min\{ \sum\limits_{h=1}^{l} (F_h-\sum\limits_{j=1}^{w-1}\hat y_{j,1}^h); \tilde b_{w,1} \}
\end{equation}
where $ \sum\limits_{h=1}^{l} \hat y^h_{w,1}$, for $w=1,\ldots,k$, are defined as in \eqref{eq:j,1,F}.
\end{lemma}

Now, we the inequalities $I(k,b,{\bf F})$ for $b>1$, by a double induction on $b$ and on the item types contained in the set $T'_b$.  
Namely, we assume by induction that there exists a set of inequalities $I(k,b-1,{\bf F})$ that satisfies conditions $1)$, $2)$ and $3)$, where  ${\bf F}$ is a vector such that $d_h$ divides $F_h$, for $h=1,\ldots, b-2$, and $d_{b-1}$ divides $F_{b-1},\ldots, F_l$. The basic step of the induction is for $b=1$ and is given by Lemma \ref{lem:ineqk=1}. Then, we define the set $I(k,b,{\bf F})$ and prove that it satisfies conditions $1)$, $2)$ and $3)$ by induction  on the item types contained in the set $T'_b$.
Let  the inequalities in $I(k,b-1,{\bf F})$ be defined as 
$$\sum\limits_{q=1}^{b-1} \sum\limits_{w=1}^{k} \frac{a_{w,q} f_w}{d_q}  \sum\limits_{h=q}^{l}y^h_{w,q}\le g_{k,b-1}( \sum\limits_{h=b-1}^{l}F_h).$$


The inequalities in $I(k,b,{\bf F})$ for $b>1$ are 
of the form
$$\sum\limits_{q=1}^{b} \sum\limits_{w=1}^{k} \frac{a_{w,q} f_w}{d_q}  \sum\limits_{h=q}^{l}y^h_{w,q}\le g_{k,b}( \sum\limits_{h=b}^{l}F_h)$$
where for each inequality in  $I(k,b,{\bf F})$ there exists an inequality in $I(k,b-1,{\bf F})$ having the same coefficients $a_{w,q}$, for $w=1,\ldots,k$ and $q=1,\ldots,b-1$.


In what follows, notation and definitions are introduced in order to define the coefficient $a_{k,b}$. In particular, the coefficient $a_{k,b}$ can assume two values. 
(The coefficients $a_{w,b}$, for $w=1,\ldots,k-1$, and $a_{w,q}$, for $w=1,\ldots,k$ $q=1,\ldots,b-1$ are then recursively defined.) Recall that $\delta_b=d_b$.\\ 
We have 
\begin{equation}\label{formula gamma_k(F)}
a_{k,b}=  \left \{ \begin{array}{ll}
                \alpha_{k,b} \\
               \beta_{k,b}
              \end{array}
       \right.
   \end{equation}
 where, if $k>1$,
\begin{equation}\label{alpha}
\alpha_{k,b}=   g_{k-1,b}(\lfloor f({\cal H}_k^b)/\delta_b\rfloor \delta_b +\delta_b)- g_{k-1,b}(\lfloor f({\cal H}_k^b)/\delta_b\rfloor \delta_b)
\end{equation}
\begin{equation}\label{beta}
\beta_{k,b}=  g_{k-1,b}(\lfloor f({\cal H}_k^b)/\delta_b\rfloor \delta_b+2\delta_b)- g_{k-1,b}^h(\lfloor f({\cal H}_k^b)/\delta_b\rfloor \delta_b+\delta_b).
\end{equation}
 and,  if $k=1$,
\begin{equation}\label{alpha1}
\alpha_{1,b}=  g_{k,b-1}( F_{b-1} + \lfloor f({\cal H}_k^b)/\delta_b\rfloor \delta_b +\delta_b)- g_{k,b-1}( F_{b-1} + \lfloor f({\cal H}_k^b)/\delta_b\rfloor \delta_b)
\end{equation}
\begin{equation}\label{beta1}
\beta_{1,b}=  g_{k,b-1}( F_{b-1} + \lfloor f({\cal H}_k^b)/\delta_b\rfloor \delta_b+2\delta_b)- g_{k,b-1}(F_{b-1} + \lfloor f({\cal H}_k^b)/\delta_b\rfloor \delta_b+\delta_b).
\end{equation}

Function $g_{k,b}(\sum\limits_{h=b}^{l}F_h)$ will be defined later.
In the following, we define the quantities $s_{k,b}$, $\sigma_{k,b}$ and $U$ as:
\begin{equation}\label{skh)}
s_{k,b}=  \left \{ \begin{array}{ll}
                \frac{( \sum\limits_{h=b}^{l}F_h) - \lfloor f({\cal H}_k^b)/\delta_b\rfloor \delta_b}{\delta_b} & \hbox{if } \  a_{k,b}=\alpha_{k,b} \\
               \frac{( \sum\limits_{h=b}^{l}F_h) - (\lfloor f({\cal H}_k^b)/\delta_b\rfloor \delta_b+\delta_b)}{\delta_b}& \hbox{if } \ a_{k,b}=\beta_{k,b}
              \end{array}
       \right.
   \end{equation}

\begin{equation}\label{sigmakb)}
\sigma_{k,b}=  \left \{ \begin{array}{ll}
                 0 & \hbox{if } \ s_{k,b}<0\\
                 s_{k,b} & \hbox{if } \ 0\le s_{k,b} \le \frac{f_k\tilde b_{k,b}}{d_b}\\
                \frac{f_k\tilde b_{k,b}}{d_b} & \hbox{if } \ s_{k,b}>  \frac{f_k\tilde b_{k,b}}{d_b}
             \end{array}
       \right.
   \end{equation}

  \begin{equation}\label{U}
U=  \left \{ \begin{array}{ll}
                \sum\limits_{h=b}^{l}F_h  & \hbox{if } \   \ s_{k,b}<0\\
                \lfloor f({\cal H}_k^b)/\delta_b\rfloor \delta_b& \hbox{if } \  a_{k,b}=\alpha_{k,b}  \hbox{ and } \ 0\le s_{k,b} \le  \frac{f_k\tilde b_{k,b}}{d_b} \\
             \lfloor f({\cal H}_k^b)/\delta_b\rfloor \delta_b+\delta_b & \hbox{if } \ a_{k,b}=\beta_{k,b}  \hbox{ and } \ 0\le s_{k,b} \le  \frac{f_k\tilde b_{k,b}}{d_b}\\
               \sum\limits_{h=b}^{l}F_h  - f_k\tilde b_{k,b} & \hbox{if } \ s_{k,b}> \frac{f_k\tilde b_{k,b}}{d_b}.
              \end{array}
       \right.
   \end{equation}

Then, 
if $k=1$,
$g_{1,b}(\sum\limits_{h=b}^{l}F_h)$ is defined as
\begin{equation}\label{formula g_k(F)}
g_{1,b}( \sum\limits_{h=b}^{l} F_h)=   g_{k,b-1}(  F_{b-1}+U) + a_{1,b}\sigma_{1,b}
   \end{equation}

and, if $k>1$,
\begin{equation}\label{formula g_k(F)k>1}
g_{k,b}( \sum\limits_{h=b}^{l} F_h)=   g_{k-1,b}(U)+ a_{k,b} \sigma_{k,b}.
   \end{equation}
   

The following two lemmas establish properties  of function $g_{k,b}(\cdot)$. 


\begin{lemma}\label{lem:concave}
For any $b=1,\ldots,l$, $k=1,\ldots,t$ and integers $F_b,\ldots, F_l$ multiple of $d_b$,  the function $g_{k,b}(\sum\limits_{h=b}^{l}F_h+s_{k,b}d_b)$ is  a concave nondecreasing function of $s_{k,b}$ (where the function is defined over all integers $s_{k,b}$ such that $\sum\limits_{h=b}^{l}F_h+s_{k,b}d_b\ge 0$.
\end{lemma}
\begin{proof}
The thesis is proved by induction on $b$. 
When $b=1$, by \eqref{eq:MP^{OO}(k,1,F)}, function $g_{k,1}(\sum\limits_{h=b}^{l}F_h)$ is 
$$g_{k,1}(\sum\limits_{h=b}^{l}F_h)=\sum\limits_{w=1}^{k} min\{ \sum\limits_{h=1}^{l} (F_h-\sum\limits_{j=1}^{w-1}\hat y_{j,1}^h); \tilde b_{w,1} \}=\sum\limits_{w=1}^{k} min\{ \sum\limits_{h=1}^{b-1} F_h+ \sum\limits_{h=b}^{l} F_h  -\sum\limits_{h=1}^{l}\sum\limits_{j=1}^{w-1}\hat y_{j,1}^h; \tilde b_{w,1} \},
$$
where $\hat y_{w,1}^h$, for $w=1,\ldots, k$ and $h=1,\ldots, l$,  are constants whose values are specified by Lemma \ref{lem:solk=1}.
 Observe that, each function $\min\{  \sum\limits_{h=1}^{l} (F_h -\sum\limits_{j=1}^{w-1}\hat y_{j,1}^p)+d_b s_{k,b}; \tilde b_{w,1} \}$ is a concave non decreasing function of $s_{k,b}$, for $w=1,\ldots,k$. Hence, $g_{k,1}(\sum\limits_{h=b}^{l}F_h+d_bs_{k,b})$ is a concave non decreasing function of $s_{k,b}$, too.

Let us assume that the statement is true for $b-1$ and show it for $b$. We prove the thesis by induction on the item types contained in the set $T'_b$. Namely, we first show the thesis when $T'_b$ only contains  items of type 1, then we assume that the thesis holds when $T'_b$  contains  items of types $1,\ldots,k-1$ and show it when $T'_b$  contains  items of types $1,\ldots,k$.
We use the following notation (recall that $\delta_b=d_b$)
\begin{equation}\label{A}
A=  \left \{ \begin{array}{ll}
                \lfloor f({\cal H}_k^b)/\delta_b\rfloor \delta_b& \hbox{if } \  a_{k,b}=\alpha_{k,b} \\
               \lceil f({\cal H}_k^b)/\delta_b\rceil \delta_b& \hbox{if } \ a_{k,b}=\beta_{k,b}
              \end{array}
       \right.
   \end{equation}
and 
$G_{k,b}(s_{k,b})=g_{k,b}(A+(s_{k,b}+1)d_b)-g_{k,b}(A+s_{k,b}d_b).$

Recalling \eqref{skh)}, we prove the thesis by  showing that the function $G_{k,b}(s_{k,b})$  
is a non-increasing function of $s_{k,b}$. 
When $T'_b$ only contains  items of type 1, we have
$G_{1,b}(s_{1,b})=g_{1,b}(A+(s_{1,b}+1)d_b)-g_{1,b}(A+s_{1,b}d_b)$
and by definition of function $g_{1,b}(\cdot)$ (see \eqref{formula g_k(F)}) we obtain 
$$G_{1,b}(s_{1,b})=\left \{ \begin{array}{lll}
                g_{1,b-1}(F_{b-1}+A+(s_{1,b}+1)d_b)-g_{1,b-1}(F_{b-1}+A+s_{1,b}d_b)& \hbox{if } \ s_{1,b}<0\\
                a_{1b}=g_{1,b-1}(F_{b-1}+A+d_b)-g_{1,b-1}(F_{b-1}+A)& \hbox{if } \ 0\le s_{1,b} < \tilde b_{1,b}\\
               g_{1,b-1}(F_{b-1}+A+(s_{1,b}+1-\tilde b_{1,b}/d_b)d_b)-g_{1,b-1}(F_{b-1}+A+(s_{1,b}-\tilde b_{1,b}/d_d)d_b)& \hbox{otherwise} .
              \end{array}
       \right.$$
The thesis follows by induction, since $g_{1,b-1}( \sum\limits_{h=b-1}^{l} P_h+p\, d_b)=g_{1,b-1}(\sum\limits_{h=b-1}^{l} P_h+p ( d_b/d_{b-1})d_{b-1})$  is a concave non decreasing function of $p$, for any integers $P_{b-1},\ldots,P_l$ multiple of $d_{b-1}$.

Assume that the thesis holds when $T'_b$  contains  items of types $1,\ldots,k-1$ and show it when $T'_b$  contains  items of types $1,\ldots,k$.
We have
$G_{k,b}(s_{k,b})=g_{k,b}(A+(s_{k,b}+1)d_b)-g_{k,b}(A+s_{k,b}d_b)$,
and by definition of function $g_{k,b}(\cdot)$ (see \eqref{formula g_k(F)k>1}) we obtain 
$$G_{k,b}(s_{k,b})=\left \{ \begin{array}{lll}
                g_{k-1,b}(A+(s_{k,b}+1)d_b)-g_{k-1,b}(A+s_{k,b}d_b)& \hbox{if } \ s_{k,b}<0\\
                a_{k,b}=g_{k-1,b}(A+d_b)-g_{k-1,b}(A)& \hbox{if } \ 0\le s_{k,b} < \tilde b_{k,b}\\
               g_{k-1,b}(A+(s_{k,b}+1-\tilde b_{k,b}/d_b)d_b)-g_{k-1,b}(A+(s_{k,b}-\tilde b_{k,b}/d_d)d_b)& \hbox{otherwise} .
              \end{array}
       \right.$$
Also in this case the thesis follows by induction.

\end{proof}

In Lemma \ref{lem:two_properties1} two properties of function $g_{k,b}(\cdot)$ are given, directly following from Lemma \ref{lem:concave}.
\begin{lemma}\label{lem:two_properties1}
Let $F$ and $G$ be two natural numbers such that $F \le G$ and $d_b$ divides $F$ and $G$. Then,

$(i)$ $g_{k,b}(G)+ \sigma (g_{k,b}(F+d_b)-g_{k,b}(F))\ge g_{k,b}(G+\sigma d_b)$;

$(ii)$ $g_{k,b}(F-\sigma d_b)+ \sigma (g_{k,b}(G) - g_{k,b}(G-d_b))\le g_{k,b}(F)$;\\  
 for every $\sigma \in  \mathbb{N}$ such that $F-\sigma d_b \ge 0$.
\end{lemma}

   

To show that the inequalities in  $I(k,b,{\bf F})$ satisfy conditions $1)$, $2)$ and $3)$ we use an inductive argument on the item types contained in the item set $T'_b$. Namely, 
we first prove the thesis (in Theorem \ref{thm:polyMOP}) when $T'_b$ only contains items of type 1 (with size $s_1=d_1=f_1=1$), then we assume the thesis holds  when $T'_b$  contains items of types $1, \ldots,k-1$ and prove it when $T'_b$  contains items of types $1, \ldots,k$ (in Theorem \ref{thm:polyMOPMain}).
\begin{theorem}\label{thm:polyMOP}
Suppose that $T'_b$ only contains items of type 1. If the inequalities in  $I(k,b-1,{\bf F})$ satisfy conditions $1)$, $2)$ and $3)$, for all ${\bf F}$ such that $d_h$ divides $F_h$, for $h=1,\ldots,b-2$, and $d_{b-1}$ divides $F_{b-1},\ldots, F_l$, then the inequalities in $I(k,b,{\bf F})$ satisfy conditions $1)$, $2)$ and $3)$ for all ${\bf F}$ such that $d_h$ divides $F_h$, for $h=1,\ldots,b-1$, and $d_{b}$ divides $F_{b},\ldots, F_l$.
\end{theorem}
\begin{proof}
Since $T'_b$ only contains items of type $1$, the set $I(k,b,{\bf F})$ contains inequalities of the type:
\begin{equation}\label{eq:feasPrinc1}
\sum\limits_{q=1}^{b-1} \sum\limits_{w=1}^{k} \frac{a_{w,q} f_w}{d_q}  \sum\limits_{h=q}^{l}y^h_{w,q}+  \frac{a_{1,b} f_1}{d_b}  \sum\limits_{h=b}^{l}y^h_{1,b}\le g_{1,b}( \sum\limits_{h=b}^{l}F_h),
\end{equation}
where $f_1=1$ and the coefficients $a_{w,q}$, for $w=1,\ldots,k$ and $q=1,\ldots,b-1$, and $a_{1,b}$ are defined as in \eqref{formula gamma_k(F)}.

Recall that, by definition, $\delta_b=d_b$. In what follows $d_b$ and $\delta_b$ will be used indifferently. 

$(a)$

Let ${\bf F}$ and ${\bf F'}$ be  two vectors with $l$ components such that  $F_h=F_h'$ for $h=1,\ldots, b-1$ and  $d_b$  divides $F_h$ and  $F_h'$ for $h=b,\ldots, l$. Then condition $1)$ follows by the inductive hypothesis on $I(k,b-1,{\bf F})$ and
by definitions \eqref{alpha}--\eqref{beta1}. In fact, since, by definition, $f({\cal H}_w^q)$ depends by the values of $F_1,\ldots, F_{q-1}$, $a_{wq}$, for $w=1,\ldots, k$ and for $q=1,\ldots, b-1$, attains the same value both for ${\bf F}$ and for ${\bf F'}$.

$(b)$

We show that the inequalities are valid for $MP^{OO}(k,b,{\bf F})$ (i.e., condition $2)$ holds).
Let $y \in MP^{OO}(k,b,{\bf F})$. We show that the following inequality holds in $y$:
\begin{eqnarray}\label{eq:feasM}
\sum\limits_{q=1}^{b-1} \sum\limits_{w=1}^{k} \frac{a_{w,q} f_w}{d_q}  \sum\limits_{h=q}^{l}y^h_{w,q}+  \frac{a_{1,b} f_1}{d_b}  \sum\limits_{h=b}^{l}y^h_{1,b}\le 
g_{k,b-1}( F_{b-1} + \sum\limits_{h=b}^{l}F_h -  \sum\limits_{h=b}^{l}\left\lceil \frac{f_1y^h_{1,b}}{d_b}\right\rceil d_b)+ a_{1,b} \sum\limits_{h=b}^{l}\left\lceil \frac{f_1y^h_{1,b}}{d_b}\right\rceil .\end{eqnarray}
Recall that, by the hypothesis, $d_b$ divides $F_h$ for $h=b,\ldots,l$. Obviously, $\frac{a_{1,b} f_1}{d_b}  \sum\limits_{h=b}^{l}y^h_{1,b}\le a_{1,b} \sum\limits_{h=b}^{l}\left\lceil \frac{f_1y^h_{1,b}}{d_b}\right\rceil$. Hence,  inequality \eqref{eq:feasM} directly holds since, by condition $1)$ and induction, the  inequality:
$$\sum\limits_{q=1}^{b-1} \sum\limits_{w=1}^{k} \frac{a_{w,q} f_w}{d_q}  \sum\limits_{h=q}^{l}y^h_{w,q} \le g_{k,b-1}( G)$$
is valid  for $I(k,b-1,{\bf F'})$, for each vector  ${\bf F'}$  such that: $(i)$ $F'_h=F_h$ for $h=1,\ldots,b-2$;  $(ii)$ $F'_{b-1},\ldots,F'_l$ are  multiple of $d_{b-1}$ such that $\sum\limits_{h=b-1}^{l} F'_h=G$. 


To show that all the inequalities in $I(k,b,{\bf F})$, of the form \eqref{eq:feasPrinc1}, are valid, different cases are considered (corresponding to the different values attained by $s_{1,b}$, as defined in \eqref{skh)}). 

If $\sum\limits_{h=b}^{l}F_h \le \lfloor f({\cal H}_1^b)/d_b \rfloor d_b $, then, by \eqref{skh)} and \eqref{sigmakb)}, 
$s_{1,b}\le0$ and $\sigma_{1,b}=0$, both when $a_{1,b}=\alpha_{1,b}$ and $a_{1,b}=\beta_{1,b}$. If $\sum\limits_{h=b}^{l}y^h_{1,b}=0$. Hence, from definition 
\eqref{formula g_k(F)}, i.e.,  $g_{1,b}(\sum\limits_{h=b}^{l}F_h)=g_{k,b-1}( F_{b-1} + \sum\limits_{h=b}^{l}F_h)$, and by  \eqref{eq:feasM} it follows that the inequalities  \eqref{eq:feasPrinc1} are valid. Suppose that $\sum\limits_{h=b}^{l}y^h_{1,b}>0$. Then if $a_{1,b}=\alpha_{1,b}$ we have
{\scriptsize
\begin{eqnarray}
\sum\limits_{q=1}^{b-1} \sum\limits_{w=1}^{k} \frac{a_{w,q} f_w}{d_q}  \sum\limits_{h=q}^{l}y^h_{w,q}+  \frac{\alpha_{1,b} f_1}{d_b}  \sum\limits_{h=b}^{l}y^h_{1,b}\le
g_{k,b-1}( F_{b-1}   +\sum\limits_{h=b}^{l}F_h-  \sum\limits_{h=b}^{l}\left\lceil \frac{f_1y^h_{1,b}}{d_b}\right\rceil d_b)+ \alpha_{1,b}\sum\limits_{h=b}^{l}\left\lceil \frac{f_1y^h_{1,b}}{d_b}\right\rceil=\label{eq:Case11c}\\
g_{k,b-1}( F_{b-1}   + \sum\limits_{h=b}^{l}F_h-  \sum\limits_{h=b}^{l}\left\lceil \frac{f_1y^h_{1,b}}{d_b}\right\rceil d_b)+
\sum\limits_{h=b}^{l}\left\lceil \frac{f_1y^h_{1,b}}{d_b}\right\rceil  ( g_{k,b-1}( F_{b-1}  +\lfloor f({\cal H}_1^{b})/\delta_b\rfloor \delta_b +\delta_b)- g_{k,b-1}( F_{b-1}  +\lfloor f({\cal H}_1^{b})/\delta_b\rfloor \delta_b) ) \le\label{eq:Case11}\\
g_{k,b-1}( F_{b-1}   +\sum\limits_{h=b}^{l}F_h)= g_{1,b}(\sum\limits_{h=b}^{l}F_h) \label{eq:Case1}\end{eqnarray}
}
where the equality \eqref{eq:Case11c} follows from the definition of $\alpha_{1,b}$ (see \eqref{alpha1}), the inequality \eqref{eq:Case11}  follows from condition $(ii)$ of Lemma \ref{lem:two_properties1}, and  equality \eqref{eq:Case1} from \eqref{formula g_k(F)}. Hence, the inequality \eqref{eq:feasPrinc1} in which 
 $a_{1,b}= \alpha_{1,b}$ is valid. 
The above argument can be applied by replacing  in \eqref{eq:Case11c}, $\alpha_{1,b}$ with $\beta_{1,b}$, and, in  \eqref{eq:Case11} the expression \eqref{beta1} in  place of \eqref{alpha1}.
Hence, the thesis follows. 

If $\sum\limits_{h=b}^{l}F_h - (\lceil f({\cal H}_1^b)/d_b \rceil d_b +d_b) \ge f_1 \tilde b_{1,b}$, then
$s_{1,b}>\frac{f_1\tilde b_{1,b}}{d_b}$ and $\sigma_{1,b}=\frac{f_1\tilde b_{1,b}}{d_b}$,  both when $a_{1,b}=\alpha_{1,b}$ and $a_{1,b}=\beta_{1,b}$. If $\sum\limits_{h=b}^{l}y^h_{1,b}=\tilde b_{1,b}$, then from \eqref{formula g_k(F)} and by induction we have that the inequalities \eqref{eq:feasPrinc1}  are valid. 
Suppose that $\sum\limits_{h=b}^{l}y^h_{1,b}<\tilde b_{1,b}$. If $a_{1,b}=\beta_{1,b}$, by \eqref{eq:feasM} and the definition of $\beta_{1,b}$ (see \eqref{beta1}) we have 
$$\sum\limits_{q=1}^{b-1} \sum\limits_{w=1}^{k} \frac{a_{w,q} f_w}{d_q}  \sum\limits_{h=q}^{l}y^h_{w,q}+  \frac{\beta_{1,b} f_1}{d_b}  \sum\limits_{h=b}^{l}y^h_{1,b}\le$$$$ g_{k,b-1}( F_{b-1}  + \sum\limits_{h=b}^{l}F_h-  \sum\limits_{h=b}^{l}\left\lceil \frac{f_1y^h_{1,b}}{d_b}\right\rceil d_b)+ \beta_{1,b} \sum\limits_{h=b}^{l}\left\lceil \frac{f_1y^h_{1,b}}{d_b}\right\rceil=$$
$$ g_{k,b-1}( F_{b-1}   +\sum\limits_{h=b}^{l}F_h-  \sum\limits_{h=b}^{l}\left\lceil \frac{f_1y^h_{1,b}}{d_b}\right\rceil d_b)+ $$$$\sum\limits_{h=b}^{l}\left\lceil \frac{f_1y^h_{1,b}}{d_b}\right\rceil ( g_{k,b-1}( F_{b-1}  +\lfloor f({\cal H}_1^{b})/\delta_b\rfloor \delta_b +2\delta_b)- g_{k,b-1}( F_{b-1}  +\lfloor f({\cal H}_1^{b})/\delta_b\rfloor \delta_b +\delta_b ))=Q$$
and by condition $(i)$ of Lemma \ref{lem:two_properties1} we have 
$$Q\le 
g_{k,b-1}( F_{b-1}   +\sum\limits_{h=b}^{l}F_h-f_1 \tilde b_{1,b})+ $$
$$(\frac{ f_1}{d_b}\tilde b_{1,b}-\sum\limits_{h=b}^{l}\left\lceil \frac{f_1y^h_{1,b}}{d_b}\right\rceil)( g_{k,b-1}( F_{b-1}  +\lfloor f({\cal H}_1^{b})/\delta_b\rfloor \delta_b +2\delta_b)- g_{k,b-1}( F_{b-1}  +\lfloor f({\cal H}_1^{b})/\delta_b\rfloor \delta_b+\delta_b))+   $$
$$\sum\limits_{h=b}^{l}\left\lceil \frac{f_1y^h_{1,b}}{d_b}\right\rceil ( g_{k,b-1}( F_{b-1}  +\lfloor f({\cal H}_1^{b})/\delta_b\rfloor \delta_b +2\delta_b)- g_{k,b-1}( F_{b-1}  +\lfloor f({\cal H}_1^{b})/\delta_b\rfloor \delta_b  +\delta_b))=$$$$
g_{k,b-1}( F_{b-1}   +\sum\limits_{h=b}^{l}F_h-f_1 \tilde b_{1,b})+\frac{ f_1}{d_b}\tilde b_{1,b}\beta_{1,b}=g_{1,b}(  \sum\limits_{h=b}^{l}F_h).$$
Hence, the thesis follows. If $a_{1,b}=\alpha_{1,b}$, since $\sum\limits_{h=b}^{l}F_h - (\lceil f({\cal H}_1^b)/d_b \rceil d_b +d_b) \ge f_1 \tilde b_{1,b}$ implies $\sum\limits_{h=b}^{l}F_h - (\lfloor f({\cal H}_1^b)/d_b \rfloor d_b +d_b) \ge f_1 \tilde b_{1,b}$, the above arguments can be repeated and easily adapted to show the thesis.

Finally, consider the case in which $\sum\limits_{h=b}^{l}F_h - (\lceil f({\cal H}_1^b)/d_b \rceil d_b +d_b) < f_1 \tilde b_{1,b}$ and $ \sum\limits_{h=b}^{l}F_h>\lfloor f({\cal H}_1^b)/d_b \rfloor d_b$ (i.e., $ \sum\limits_{h=b}^{l}F_h\ge \lfloor f({\cal H}_1^b)/d_b \rfloor d_b + d_b$, since $d_b$ divides $ \sum\limits_{h=b}^{l}F_h$). Two subcases can be considered: either $\sum\limits_{h=b}^{l}F_h - \lfloor f({\cal H}_1^b)/d_b \rfloor d_b \le f_1 \tilde b_{1,b}$ or $\sum\limits_{h=b}^{l}F_h - \lfloor f({\cal H}_1^b)/d_b \rfloor d_b > f_1 \tilde b_{1,b}$.
In the first subcase, obviously, $\sum\limits_{h=b}^{l}F_h - \lceil f({\cal H}_1^b)/d_b \rceil d_b \le f_1 \tilde b_{1,b}$ holds, too. Hence, we have
$0\le s_{1,b}\le \frac{f_1\tilde b_{1,b}}{d_b}$. Then, by \eqref{skh)}, \eqref{U} and \eqref{formula g_k(F)}, and since $\sigma_{1,b}=s_{1,b}$ (by \eqref{sigmakb)}),
inequality \eqref{eq:feasM} becomes
$$\sum\limits_{q=1}^{b-1} \sum\limits_{w=1}^{k} \frac{a_{w,q} f_w}{d_q}  \sum\limits_{h=q}^{l}y^h_{w,q}+  \frac{a_{1,b} f_1}{d_b}  \sum\limits_{h=b}^{l}y^h_{1,b}\le$$$$ 
g_{k,b-1}( F_{b-1}   +\sum\limits_{h=b}^{l}F_h-  \sum\limits_{h=b}^{l}\left\lceil \frac{f_1y^h_{1,b}}{d_b}\right\rceil d_b)+ a_{1,b}\sum\limits_{h=b}^{l}\left\lceil \frac{f_1y^h_{1,b}}{d_b}\right\rceil =$$
\begin{equation}\label{eq:feas4}
g_{k,b-1}( F_{b-1}  +U- d_b(  \sum\limits_{h=b}^{l}\left\lceil \frac{f_1y^h_{1,b}}{d_b}\right\rceil - s_{1,b}))+ a_{1,b}( \sum\limits_{h=b}^{l}\left\lceil \frac{f_1y^h_{1,b}}{d_b}\right\rceil  - s_{1,b})+a_{1,b}s_{1,b}.
\end{equation}
If $\sum\limits_{h=b}^{l}\left\lceil \frac{f_1y^h_{1,b}}{d_b}\right\rceil=s_{1,b}$, 
by \eqref{eq:feasM} and by recalling that  $g_{1,b}(\sum\limits_{h=b}^{l}F_h)=g_{k,b-1}( F_{b-1}  +U)+ a_{1,b}\sigma_{1,b}$, the inequalities \eqref{eq:feasPrinc1} are satisfied for $a_{1,b}\in \{\alpha_{1,b}; \beta_{1,b}\}$. \\
Observe that since $0\le s_{1,b}\le \frac{f_1\tilde b_{1,b}}{d_b}$ then
\begin{equation}\label{eq:newa1b}
a_{1,b}=g_{k,b-1}( F_{b-1}  +U+ d_b)- g_{k,b-1}( F_{b-1}  +U).
\end{equation}
If $ \sum\limits_{h=b}^{l}\left\lceil \frac{f_1y^h_{1,b}}{d_b}\right\rceil  > s_{1,b}$, then  by condition $(ii)$ of Lemma \ref{lem:two_properties1} and by \eqref{eq:newa1b} we have that
$$g_{k,b-1}( F_{b-1}  +U- d_b(  \sum\limits_{h=b}^{l}\left\lceil \frac{f_1y^h_{1,b}}{d_b}\right\rceil - s_{1,b}))+ a_{1,b}(  \sum\limits_{h=b}^{l}\left\lceil \frac{f_1y^h_{1,b}}{d_b}\right\rceil  - s_{1,b})\le g_{k,b-1}( F_{b-1}  +U).$$
Then,  by \eqref{eq:feas4} and since $g_{1,b}(\sum\limits_{h=b}^{l}F_h)=g_{k,b-1}( F_{b-1}  +U)+ a_{1,b}s_{1,b}$:
$$\sum\limits_{q=1}^{b-1} \sum\limits_{w=1}^{k} \frac{a_{w,q} f_w}{d_q}  \sum\limits_{h=q}^{l}y^h_{w,q}+  \frac{a_{1,b} f_1}{d_b}  \sum\limits_{h=b}^{l}y^h_{1,b}\le g_{1,b}(\sum\limits_{h=b}^{l}F_h)$$
i.e., the thesis. Finally, if $ \sum\limits_{h=b}^{l}\left\lceil \frac{f_1y^h_{1,b}}{d_b}\right\rceil  < s_{1,b}$,  by condition $(i)$ of Lemma \ref{lem:two_properties1} and by \eqref{eq:newa1b} it follows that
$$g_{k,b-1}( F_{b-1}  +U- d_b(  \sum\limits_{h=b}^{l}\left\lceil \frac{f_1y^h_{1,b}}{d_b}\right\rceil - s_{1,b}))+ a_{1,b}(  \sum\limits_{h=b}^{l}\left\lceil \frac{f_1y^h_{1,b}}{d_b}\right\rceil  - s_{1,b})=$$
$$g_{k,b-1}( F_{b-1}  +U+ d_b(  s_{1,b}-\sum\limits_{h=b}^{l}\left\lceil \frac{f_1y^h_{1,b}}{d_b}\right\rceil ))- a_{1,b}(  s_{1,b} - \sum\limits_{h=b}^{l}\left\lceil \frac{f_1y^h_{1,b}}{d_b}\right\rceil )
\le g_{k,b-1}( F_{b-1}  +U).$$
Hence, by \eqref{eq:feas4}  and since $g_{1,b}(\sum\limits_{h=b}^{l}F_h)=g_{k,b-1}( F_{b-1}  +U)+ a_{1,b}s_{1,b}$, it follows that the inequalities \eqref{eq:feasPrinc1} are valid.

Let us consider the second subcase, i.e., $\sum\limits_{h=b}^{l}F_h - \lfloor f({\cal H}_1^b)/d_b \rfloor d_b > f_1 \tilde b_{1,b}$. Since $d_b$ divides $f_1\tilde b_{1,b}$ and $\sum\limits_{h=b}^{l}F_h$, then $ \sum\limits_{h=b}^{l}F_h- \lfloor f({\cal H}_1^b)/d_b \rfloor d_b + d_b \ge f_1 \tilde b_{1,b}$, too. Hence,  we have
$0\le s_{1,b}\le \frac{f_1\tilde b_{1,b}}{d_b}+1$. Hence, if $0\le s_{1,b}\le \frac{f_1\tilde b_{1,b}}{d_b}$ the same arguments of the fisrt subcase can be used to show the thesis. If $ s_{1,b}= \frac{f_1\tilde b_{1,b}}{d_b}+1>\frac{f_1\tilde b_{1,b}}{d_b}$, the same arguments applied in the case $\sum\limits_{h=b}^{l}F_h - (\lceil f({\cal H}_1^b)/d_b \rceil d_b +d_b) \ge f_1 \tilde b_{1,b}$ can be used to prove that   \eqref{eq:feasPrinc1} holds.

\medskip
$(c)$

We show now that condition $3)$ holds, i.e.,  $MO^{OO}(k,b,{\bf F})$ is contained in the faces induced by the inequalities in $I(k,b,{\bf F})$.
According to the Definition of $U$ and $\sigma_{1,b}$, we can write $\sum\limits_{h=b}^{l}F_h=U+\sigma_{1,b}d_b$. 
Let   $y \in MO^{OO}(k,b,{\bf F})$ be an optimal solution and   let 
$$c= \frac{\sum\limits_{h=b}^{l}F_h - \lfloor f({\cal H}_1^b)/d_b\rfloor d_b}{d_b}.$$

Since $d_b=\delta_b$, Theorem \ref{thm:itembk(hl)} and Proposition \ref{prop:M-SP-1} imply that 
$\sum\limits_{h=b}^{l}\frac{f_1y^h_{1,b}}{d_b}$ can attain one of the following values: $\min \{c^+; f_1 \tilde b_{1,b}/d_b\}$ or 
$\min \{(c-1)^+; f_1 \tilde b_{1,b}/d_b\}$ or $\min \{(c-2)^+; f_1 \tilde b_{1,b}/d_b\}$.

In the case $c\le 0$ then  $\sum\limits_{h=b}^{l}y^h_{1,b}=0$ in every optimal solution. Hence, by induction,  for each optimal solution  $ y \in MO^{OO}(k,b,{\bf F})$ there exists an inequality in $I(k,b-1,{\bf F})$ such that
$$\sum\limits_{q=1}^{b-1} \sum\limits_{w=1}^{k} \frac{a_{w,q} f_w}{d_q}  \sum\limits_{h=q}^{l}y^h_{w,q} = g_{k,b-1}( F_{b-1}  + \sum\limits_{h=b}^{l}F_h).$$
Since in this case $s_{1,b}\le 0$, then $\sigma_{1,b}= 0$ and $g_{k,b-1}( F_{b-1}  + \sum\limits_{h=b}^{l}F_h)=g_{k,b}( \sum\limits_{h=b}^{l}F_h)$.
Hence, the thesis follows. 

In the case $c-2\ge f_1 \tilde b_{1,b}/d_b $, then Theorem \ref{thm:itembk(hl)} and Proposition \ref{prop:M-SP-1} imply that $\sum\limits_{h=b}^{l}y^h_{1,b}= \tilde b_{1,b}$ in every optimal solution.
Hence, by induction,  for each optimal solution in $ y \in MO^{OO}(k,b,{\bf F})$ there exists an inequality such that
$$\sum\limits_{q=1}^{b-1} \sum\limits_{w=1}^{k} \frac{a_{w,q} f_w}{d_q}  \sum\limits_{h=q}^{l}y^h_{w,q} = g_{k,b-1}( F_{b-1}  + \sum\limits_{h=b}^{l}F_h-  f_1 \tilde b_{1,b}).$$
In this case $s_{1,b}\ge f_1 \tilde b_{1,b}/d_b$, both when $a_{1,b}=\alpha_{1,b}$ and $a_{1,b}=\beta_{1,b}$. Hence, by definition, $\sigma_{1,b}= f_1 \tilde b_{1,b}/d_b$. As  $\sum\limits_{h=b}^{l}y^h_{1,b}= \tilde b_{1,b}$ in every optimal solution,
we obtain 
$$\sum\limits_{q=1}^{b-1} \sum\limits_{w=1}^{k} \frac{a_{w,q} f_w}{d_q}  \sum\limits_{h=q}^{l}y^h_{w,q} 
+ \frac{a_{1,b} f_1}{d_b}  \sum\limits_{h=b}^{l}y^h_{1,b}
= g_{k,b-1}( F_{b-1}  + \sum\limits_{h=b}^{l}F_h-  f_1 \tilde b_{1,b})+a_{1,b}f_1 \tilde b_{1,b}/d_b=g_{k,b}( \sum\limits_{h=b}^{l}F_h).$$ 
And, the thesis holds both if $a_{1,b}=\alpha_{1,b}$ and  $a_{1,b}=\beta_{1,b}$. 

In the case $c= f_1 \tilde b_{1,b}/d_b +1$,  
two subcases can be considered: $(I)$ $ \lfloor f({\cal H}_1^b)/d_b\rfloor = \lceil f({\cal H}_1^b)/d_b\rceil $; $(II)$ $ \lfloor f({\cal H}_1^b)/d_b\rfloor \ne \lceil f({\cal H}_1^b)/d_b\rceil $.
 Theorem \ref{thm:itembk(hl)} and Proposition \ref{prop:M-SP-1} imply that $\sum\limits_{h=b}^{l}f_1 y^h_{1,b}/d_b= c-1=f_1 \tilde b_{1,b}/d_b $ in Case $(I)$, while 
$\sum\limits_{h=b}^{l}f_1 y^h_{1,b}/d_b\in \{c-1; c-2\}$ in  Case $(II)$, in every optimal solution.
(Note that, since w.l.o.g.  $c=f_1 \tilde b_{1,b}/d_b +1\ge  2$, then $c-2\ge0$.)
 In Case $(I)$, by induction,  for each optimal solution $ y \in MO^{OO}(k,b,{\bf F})$ there exists an inequality such that: 
$\sum\limits_{q=1}^{b-1} \sum\limits_{w=1}^{k} \frac{a_{w,q} f_w}{d_q}  \sum\limits_{h=q}^{l}y^h_{w,q} = g_{k,b-1}( F_{b-1}  + \sum\limits_{h=b}^{l}F_h-  (c-1)d_b)= g_{k,b-1}( F_{b-1} + \lfloor f({\cal H}_1^b)/d_b\rfloor+d_b).$\\
Let $a_{1,b}=\alpha_{1,b} $, then we have $s_{1,b}=\sigma_{1,b}=c-1$ and we obtain 
$$\sum\limits_{q=1}^{b-1} \sum\limits_{w=1}^{k} \frac{a_{w,q} f_w}{d_q}  \sum\limits_{h=q}^{l}y^h_{w,q}+  \frac{\alpha_{1,b} f_1}{d_b}  \sum\limits_{h=b}^{l}y^h_{1,b}=
 g_{k,b-1}( F_{b-1}  + \lfloor f({\cal H}_1^b)/d_b\rfloor+d_b)+ (c-1)\alpha_{1,b}= $$
$$g_{k,b-1}( F_{b-1}  + \lfloor f({\cal H}_1^b)/d_b\rfloor)+  (g_{k,b-1}( F_{b-1}  + \lfloor f({\cal H}_1^b)/d_b\rfloor+d_b)-g_{k,b-1}( F_{b-1}  + \lfloor f({\cal H}_1^b)/d_b\rfloor))+ (c-1)\alpha_{1,b}=$$
$$g_{k,b-1}( F_{b-1}  + \lfloor f({\cal H}_1^b)/d_b\rfloor)+ \alpha_{1,b}\sigma_{1,b}=g_{1,b}( \sum\limits_{h=b}^{l}F_h)$$
Note that, since $ \lfloor f({\cal H}_1^b)/d_b\rfloor = \lceil f({\cal H}_1^b)/d_b\rceil $, the above relations hold even if $\alpha_{1,b}$ and $\lfloor f({\cal H}_1^b)/d_b\rfloor$ are replaced by $\beta_{1,b}$ and $\lceil f({\cal H}_1^b)/d_b\rceil$, respectively. Hence, in this case, two inequalities exist that hold with equality at $y$. In Case $(II)$, by induction,  for each optimal solution  $ y \in MO^{OO}(k,b,{\bf F})$, two sub cases hold: $(II.a)$ $\sum\limits_{h=b}^{l}f_1 y^h_{1,b}/d_b= c-1$;  $(II.b)$ $\sum\limits_{h=b}^{l}f_1 y^h_{1,b}/d_b= c-2$.
Then, by induction,  there exist inequalities such that \\ Case $(II.a)$
$$\sum\limits_{q=1}^{b-1} \sum\limits_{w=1}^{k} \frac{a_{w,q} f_w}{d_q}  \sum\limits_{h=q}^{l}y^h_{w,q} = g_{k,b-1}( F_{b-1}  + \sum\limits_{h=b}^{l}F_h-  (c-1)d_b)= g_{k,b-1}( F_{b-1} + \lfloor f({\cal H}_1^b)/d_b\rfloor+d_b)$$
or\\ Case $(II.b)$
$$\sum\limits_{q=1}^{b-1} \sum\limits_{w=1}^{k} \frac{a_{w,q} f_w}{d_q}  \sum\limits_{h=q}^{l}y^h_{w,q} = g_{k,b-1}( F_{b-1}  + \sum\limits_{h=b}^{l}F_h- (c-2)d_b)= g_{k,b-1}( F_{b-1}  + \lfloor f({\cal H}_1^b)/d_b\rfloor+2d_b).$$

Note that, since $ \lfloor f({\cal H}_1^b)/d_b\rfloor \ne \lceil f({\cal H}_1^b)/d_b\rceil $, when $a_{1,b}=\beta_{1,b}$ we have 
$s_{1,b}=\sigma_{1,b}=c-1$.

In Case  $(II.a)$ we obtain 
$$\sum\limits_{q=1}^{b-1} \sum\limits_{w=1}^{k} \frac{a_{w,q} f_w}{d_q}  \sum\limits_{h=q}^{l}y^h_{w,q}+  \frac{\beta_{1,b} f_1}{d_b}  \sum\limits_{h=b}^{l}y^h_{1,b}=g_{k,b-1}( F_{b-1} + \lfloor f({\cal H}_1^b)/d_b\rfloor+d_b)+\beta_{1,b}(c-1)=$$
$$g_{k,b-1}( F_{b-1} + \lfloor f({\cal H}_1^b)/d_b\rfloor+d_b)+\beta_{1,b}\sigma_{1,b}=g_{1,b}( \sum\limits_{h=b}^{l}F_h).$$
In Case  $(II.b)$ we have 
$$\sum\limits_{q=1}^{b-1} \sum\limits_{w=1}^{k} \frac{a_{w,q} f_w}{d_q}  \sum\limits_{h=q}^{l}y^h_{w,q}+  \frac{\beta_{1,b} f_1}{d_b}  \sum\limits_{h=b}^{l}y^h_{1,b}=g_{k,b-1}( F_{b-1} +\lfloor f({\cal H}_1^b)/d_b\rfloor +2d_b)+\beta_{1,b}(c-2)=$$
$$g_{k,b-1}( F_{b-1} +\lfloor f({\cal H}_1^b)/d_b\rfloor+d_b) + (g_{k,b-1}( F_{b-1} + \lfloor f({\cal H}_1^b)/d_b\rfloor +2d_b)-g_{k,b-1}( F_{b-1} + \lfloor f({\cal H}_1^b)/d_b\rfloor +d_b))+\beta_{1,b}(c-2)=$$
$$g_{k,b-1}( F_{b-1} + \lfloor f({\cal H}_1^b)/d_b\rfloor + d_b)+\beta_{1,b}(c-1)=g_{k,b-1}( F_{b-1} +\lfloor f({\cal H}_1^b)/d_b\rfloor+d_b)+\beta_{1,b}\sigma_{1,b}=g_{1,b}( \sum\limits_{h=b}^{l}F_h).$$

Finally, we have the case  $1 \le c\le f_1 \tilde b_{1,b}/d_b $. Two subcases are considered: $2 \le c\le f_1 \tilde b_{1,b}/d_b $ and $c=1\le f_1 \tilde b_{1,b}/d_b $.
In the subcase $2 \le c\le f_1 \tilde b_{1,b}/d_b $, 
Theorem \ref{thm:itembk(hl)} and Proposition \ref{prop:M-SP-1} imply that $\sum\limits_{h=b}^{l}f_1 y^h_{1,b}/d_b\in \{c-2; c-1; c\}$ in every optimal solution. If $\sum\limits_{h=b}^{l}f_1 y^h_{1,b}/d_b=c-2$ or $\sum\limits_{h=b}^{l}f_1 y^h_{1,b}/d_b=c-1$ then the thesis follows by using the same arguments of Cases  $(II.a)$ and $(II.b)$.
 If $\sum\limits_{h=b}^{l}f_1 y^h_{1,b}/d_b=c$,  by the induction hypothesis, for each optimal solution in $ y \in MO^{OO}(k,b,{\bf F})$ there exists an inequality such that: 
\begin{eqnarray}\label{eq:c=1c}
\sum\limits_{q=1}^{b-1} \sum\limits_{w=1}^{k} \frac{a_{w,q} f_w}{d_q}  \sum\limits_{h=q}^{l}y^h_{w,q} = g_{k,b-1}( F_{b-1}  + \sum\limits_{h=b}^{l}F_h-  c d_b)= g_{k,b-1}( F_{b-1} + \lfloor f({\cal H}_1^b)/d_b\rfloor).
\end{eqnarray}
Hence, we obtain 
\begin{eqnarray}\label{eq:c=1c1}
\sum\limits_{q=1}^{b-1} \sum\limits_{w=1}^{k} \frac{a_{w,q} f_w}{d_q}  \sum\limits_{h=q}^{l}y^h_{w,q}+  \frac{\alpha_{1,b} f_1}{d_b}  \sum\limits_{h=b}^{l}y^h_{1,b}= 
 g_{k,b-1}( F_{b-1}  + \lfloor f({\cal H}_1^b)/d_b\rfloor)+  \alpha_{1,b}c=\nonumber\\
 g_{k,b-1}( F_{b-1}  + \lfloor f({\cal H}_1^b)/d_b\rfloor) +\alpha_{1,b}\sigma_{1,b}=g_{1,b}( \sum\limits_{h=b}^{l}F_h),
\end{eqnarray}
and the thesis follows. In the subcase $c=1\le f_1 \tilde b_{1,b}/d_b $, 
Theorem \ref{thm:itembk(hl)} and Proposition \ref{prop:M-SP-1} imply that $\sum\limits_{h=b}^{l}f_1 y^h_{1,b}/d_b\in \{ c-1=0; c\}$ in every optimal solution. 
If $\sum\limits_{h=b}^{l}f_1 y^h_{1,b}/d_b=c-1$ the thesis follows by the above Case $(I)$. If $\sum\limits_{h=b}^{l}f_1 y^h_{1,b}/d_b=c$ then the thesis follows from \eqref{eq:c=1c} and \eqref{eq:c=1c1}. 
\end{proof}

We are now in the position of proving the main theorem.
\begin{theorem}\label{thm:polyMOPMain}
If the inequalities in  $I(k,b-1,{\bf F})$ satisfy conditions $1)$, $2)$ and $3)$, for all ${\bf F}$ such that $d_h$ divides $F_h$, for $h=1,\ldots,b-2$, and $d_{b-1}$ divides $F_{b-1},\ldots, F_l$, then the inequalities in $I(k,b,{\bf F})$ satisfy conditions $1)$, $2)$ and $3)$ for all ${\bf F}$ such that $d_h$ divides $F_h$, for $h=1,\ldots,b-1$, and $d_{b}$ divides $F_{b},\ldots, F_l$.
\end{theorem}
\begin{proof}
The thesis is proved by induction on the item types contained in the set $T'_b$ (i.e., the items that can be assigned to the parts $b,b+1,\ldots,l$). If  $T'_b$ only contains items of type 1, the thesis follows by Theorem \ref{thm:polyMOP}. Let us assume that the thesis holds when  $T'_b$ contains items of type $1,\dots,k-1$ and show it when  $T'_b$ contains items of type $1,\dots,k$. Again, observe that $\delta_b=d_b$. Hence,  in what follows we use $d_b$  and $\delta_b$ indifferently.

The set $I(k,b,F)$ contains inequalities of the type (where each coefficient $a_{wq} \in \{\alpha_{wq}, \beta_{wq}\}$):
\begin{equation}\label{eq:feasPrinc}
\sum\limits_{q=1}^{b-1} \sum\limits_{w=1}^{k} \frac{a_{w,q} f_w}{d_q}  \sum\limits_{h=q}^{l}y^h_{w,q}+ \sum\limits_{w=1}^{k-1}\frac{a_{w,b} f_{w}}{d_b}  \sum\limits_{h=b}^{l}y^h_{w,b}  +
\frac{a_{k,b} f_k}{d_b}  \sum\limits_{h=b}^{l}y^h_{k,b}\le g_{k,b}( \sum\limits_{h=b}^{l}F_h).
\end{equation}

$(a)$

Let ${\bf F}$ and ${\bf F'}$ be two vectors with $l$ components such that  $F_h=F_h'$ for $h=1,\ldots, b-1$ and  $d_b$  divides $F_h$ and  $F_h'$ for $h=b,\ldots, l$. The condition $1)$ directly follows by induction, by definition of the coefficients $a_{k,b}$ (see \eqref{alpha} and \eqref{beta}), and since $f({\cal H}_k^b)$ only depends by the values of $F_1,\ldots, F_{b-1}$.

$(b)$

We show that the inequalities in $I(k,b,{\bf F})$ are valid for $MP^{OO}(k,b,{\bf F})$ (i.e., condition $2)$ holds).
Let $y \in MP^{OO}(k,b,{\bf F})$. First we prove that the following inequality holds in $y$:
\begin{eqnarray}\label{eq:feasM1}
\sum\limits_{q=1}^{b-1} \sum\limits_{w=1}^{k} \frac{a_{w,q} f_w}{d_q}  \sum\limits_{h=q}^{l}y^h_{w,q}+ \sum\limits_{w=1}^{k-1}\frac{a_{w,b} f_{w}}{d_b}  \sum\limits_{h=b}^{l}y^h_{w,b}  +
\frac{a_{k,b} f_k}{d_b}  \sum\limits_{h=b}^{l}y^h_{k,b}\le \nonumber \\ 
g_{k-1,b}( \sum\limits_{h=b}^{l}F_h-  \sum\limits_{h=b}^{l}\left\lceil \frac{f_ky^h_{k,b}}{d_b}\right\rceil d_b) + a_{k,b}
\sum\limits_{h=b}^{l}\left\lceil \frac{f_ky^h_{k,b}}{d_b}\right\rceil
\end{eqnarray}

In fact, by induction and condition $1)$ we have that
$$\sum\limits_{q=1}^{b-1} \sum\limits_{w=1}^{k} \frac{a_{w,q} f_w}{d_q}  \sum\limits_{h=q}^{l}y^h_{w,q}+ \sum\limits_{w=1}^{k-1}\frac{a_{w,b} f_{w}}{d_b}  \sum\limits_{h=b}^{l}y^h_{w,b}   \le g_{k-1,b}( G)$$
is a valid inequality for $MP(k,b,{\bf F'})$, when $T'_b$ contains items of type $1,\ldots,k-1$, for each vector ${\bf F'}$  such that: $(i)$ $F'_h=F_h$, for $h=1,\ldots,b-1$;  $(ii)$ $F'_{b},\ldots,F'_l$ are integers multiple of $d_b$ such that $\sum\limits_{h=b}^{l} F'_h=G$. As a consequence and since, by the hypothesis, $d_b$ divides $F_h$ for $h=b,\ldots,l$, the inequality \eqref{eq:feasM1} follows.

To show that all the inequalities in $I(k,b,{\bf F})$  are valid, different cases are considered, as in Theorem \ref{thm:polyMOP}. 

If $\sum\limits_{h=b}^{l}F_h \le f({\cal H}_k^b)$, then  $s_{k,b}\le0$ (and  $\sigma_{k,b}=0$) both when $a_{k,b}=\alpha_{k,b}$ and $a_{k,b}=\beta_{k,b}$. If $\sum\limits_{h=b}^{l}y^h_{k,b}=0$, then, by definition \eqref{formula g_k(F)k>1}, $g_{k,b}(\sum\limits_{h=b}^{l}F_h)=g_{k-1,b}(\sum\limits_{h=b}^{l}F_h)$. Hence, by  \eqref{eq:feasM1}, the inequalities  \eqref{eq:feasPrinc} are valid. Suppose that $\sum\limits_{h=b}^{l}y^h_{k,b}>0$. Then, if $a_{k,b}=\alpha_{k,b}$ by \eqref{eq:feasM1} we have

$$\sum\limits_{q=1}^{b-1} \sum\limits_{w=1}^{k} \frac{a_{w,q} f_w}{d_q}  \sum\limits_{h=q}^{l}y^h_{w,q}+ \sum\limits_{w=1}^{k-1}\frac{a_{w,b} f_{w}}{d_b}  \sum\limits_{h=b}^{l}y^h_{w,b}  + \frac{\alpha_{k,b} f_k}{d_b}  \sum\limits_{h=b}^{l}y^h_{k,b}\le$$$$
g_{k-1,b}( \sum\limits_{h=b}^{l}F_h-  \sum\limits_{h=b}^{l}\left\lceil \frac{f_ky^h_{k,b}}{d_b}\right\rceil d_b)+ \alpha_{k,b}\sum\limits_{h=b}^{l}\left\lceil \frac{f_ky^h_{k,b}}{d_b}\right\rceil=$$
$$
g_{k-1,b}( \sum\limits_{h=b}^{l}F_h-  \sum\limits_{h=b}^{l}\left\lceil \frac{f_ky^h_{k,b}}{d_b}\right\rceil d_b)+$$
$$ \sum\limits_{h=b}^{l}\left\lceil \frac{f_ky^h_{k,b}}{d_b}\right\rceil  ( g_{k-1,b}(\lfloor f({\cal H}_k^{b})/\delta_b\rfloor \delta_b +\delta_b)- g_{k-1,b}(\lfloor f({\cal H}_k^{b})/\delta_b\rfloor \delta_b) ) \le $$$$
g_{k-1,b}( \sum\limits_{h=b}^{l}F_h)=g_{k-1,b}( U)= g_{k,b}(\sum\limits_{h=b}^{l}F_h) $$
where the last inequality follows from condition $(ii)$ of Lemma \ref{lem:two_properties1}, and the last two equalities from 
\eqref{U} and \eqref{formula g_k(F)k>1} (and since $\sigma_{k,b}=0$). Hence, the inequalities \eqref{eq:feasPrinc} in which 
 $a_{k,b}= \alpha_{k,b}$ are valid. 
The above argument holds even if   $\alpha_{k,b}$ is replaced with $\beta_{k,b}$, and if $\beta_{k,b}$ is replaced by the expression \eqref{beta}.
Hence, the thesis follows. 

If $\sum\limits_{h=b}^{l}F_h - (\lceil f({\cal H}_k^b)/d_b \rceil d_b +d_b) \ge f_k \tilde b_{k,b}$, then
$s_{k,b}>\frac{f_k\tilde b_{k,b}}{d_b}$ (and  $\sigma_{k,b}=\frac{f_k\tilde b_{k,b}}{d_b}$),  both when $a_{k,b}=\alpha_{k,b}$ and $a_{k,b}=\beta_{k,b}$. If $\sum\limits_{h=b}^{l}y^h_{k,b}=\tilde b_{k,b}$ then from \eqref{formula g_k(F)k>1} and by induction we have that all the inequalities \eqref{eq:feasPrinc}  are valid. 
Suppose that $\sum\limits_{h=b}^{l}y^h_{k,b}<\tilde b_{k,b}$. If $a_{k,b}=\beta_{k,b}$, by  \eqref{eq:feasM1} we have

$$\sum\limits_{q=1}^{b-1} \sum\limits_{w=1}^{k} \frac{a_{w,q} f_w}{d_q}  \sum\limits_{h=q}^{l}y^h_{w,q}+ \sum\limits_{w=1}^{k-1}\frac{a_{w,b} f_{w}}{d_b}  \sum\limits_{h=b}^{l}y^h_{w,b} + \frac{\beta_{k,b} f_k}{d_b}  \sum\limits_{h=b}^{l}y^h_{k,b}\le$$
$$ g_{k-1,b}( \sum\limits_{h=b}^{l}F_h-  \sum\limits_{h=b}^{l}\left\lceil \frac{f_ky^h_{k,b}}{d_b}\right\rceil d_b)+ \beta_{k,b} \sum\limits_{h=b}^{l}\left\lceil \frac{f_ky^h_{k,b}}{d_b}\right\rceil=$$
$$ g_{k-1,b}( \sum\limits_{h=b}^{l}F_h-  \sum\limits_{h=b}^{l}\left\lceil \frac{f_ky^h_{k,b}}{d_b}\right\rceil d_b)+ \sum\limits_{h=b}^{l}\left\lceil \frac{f_ky^h_{k,b}}{d_b}\right\rceil ( g_{k-1,b} (\lfloor f({\cal H}_k^{b})/\delta_b\rfloor \delta_b +2\delta_b)- g_{k-1,b} (\lfloor f({\cal H}_k^{b})/\delta_b\rfloor \delta_b +\delta_b) )=Q,$$
and by condition $(i)$ of Lemma \ref{lem:two_properties1} we have
$$ 
Q\le g_{k-1,b}( \sum\limits_{h=b}^{l}F_h-f_k \tilde b_{k,b})+ (\frac{ f_k}{d_b}\tilde b_{k,b}-\sum\limits_{h=b}^{l}\left\lceil \frac{f_ky^h_{k,b}}{d_b}\right\rceil)( g_{k-1,b} (\lfloor f({\cal H}_k^{b})/\delta_b\rfloor \delta_b +2\delta_b)- g_{k-1,b} (\lfloor f({\cal H}_k^{b})/\delta_b\rfloor \delta_b +\delta_b) )+   $$
$$\sum\limits_{h=b}^{l}\left\lceil \frac{f_ky^h_{k,b}}{d_b}\right\rceil ( g_{k-1,b} (\lfloor f({\cal H}_k^{b})/\delta_b\rfloor \delta_b +2\delta_b)- g_{k-1,b} (\lfloor f({\cal H}_k^{b})/\delta_b\rfloor \delta_b +\delta_b) )=$$$$
g_{k-1,b}( \sum\limits_{h=b}^{l}F_h-f_k \tilde b_{k,b})+\frac{ f_k}{d_b}\tilde b_{k,b}\beta_{k,b}=g_{k-1,b}( U)+\frac{ f_k}{d_b}\tilde b_{k,b}\beta_{k,b} =g_{k,b}(\sum\limits_{h=b}^{l}F_h),$$
where the last two equalities follow from 
\eqref{U} and \eqref{formula g_k(F)k>1}.
If $a_{k,b}=\alpha_{k,b}$, the above arguments can be repeated and easily adapted to show the thesis, too.

Finally, consider the case in which $\sum\limits_{h=b}^{l}F_h - (\lceil f({\cal H}_k^b)/d_b \rceil d_b +d_b) < f_k \tilde b_{k,b}$ and $ \sum\limits_{h=b}^{l}F_h>\lfloor f({\cal H}_k^b)/d_b \rfloor d_b$. Two subcases can be considered: either $\sum\limits_{h=b}^{l}F_h - \lfloor f({\cal H}_k^b)/d_b \rfloor d_b \le f_k \tilde b_{k,b}$ or $\sum\limits_{h=b}^{l}F_h - \lfloor f({\cal H}_k^b)/d_b \rfloor d_b > f_k \tilde b_{k,b}$.
In the first subcase, we have
$0\le s_{k,b}\le \frac{f_k\tilde b_{k,b}}{d_b}$. Then, by  \eqref{skh)}, \eqref{U} and \eqref{formula g_k(F)k>1}, and since $\sigma_{k,b}=s_{k,b}$ (see \eqref{sigmakb)}),
inequality \eqref{eq:feasM1} becomes
$$\sum\limits_{q=1}^{b-1} \sum\limits_{w=1}^{k} \frac{a_{w,q} f_w}{d_q}  \sum\limits_{h=q}^{l}y^h_{w,q}+ \sum\limits_{w=1}^{k-1}\frac{a_{w,b} f_{w}}{d_b}  \sum\limits_{h=b}^{l}y^h_{w,b} + \frac{a_{k,b} f_k}{d_b}  \sum\limits_{h=b}^{l}y^h_{k,b}\le$$$$
g_{k-1,b}( \sum\limits_{h=b}^{l}F_h-  \sum\limits_{h=b}^{l}\left\lceil \frac{f_ky^h_{k,b}}{d_b}\right\rceil d_b)+ a_{k,b}\sum\limits_{h=b}^{l}\left\lceil \frac{f_ky^h_{k,b}}{d_b}\right\rceil=$$
\begin{equation}\label{eq:feas4b}
g_{k-1,b}(U- d_b(  \sum\limits_{h=b}^{l}\left\lceil \frac{f_ky^h_{k,b}}{d_b}\right\rceil - s_{kb }))+ a_{k,b}( \sum\limits_{h=b}^{l}\left\lceil \frac{f_ky^h_{k,b}}{d_b}\right\rceil  - s_{kb })+a_{k,b}s_{k,b}.
\end{equation}
If $\sum\limits_{h=b}^{l}\left\lceil \frac{f_ky^h_{k,b}}{d_b}\right\rceil=s_{k,b}$, 
by \eqref{eq:feasM1} and by recalling that  $g_{k,b}(\sum\limits_{h=b}^{l}F_h)=g_{k-1,b}(U)+ a_{k,b}\sigma_{k,b}$, the inequalities \eqref{eq:feasPrinc} are satisfied for $a_{k,b}\in \{\alpha_{k,b}; \beta_{k,b}\}$.
Observe that since $0\le s_{k,b}\le \frac{f_k\tilde b_{k,b}}{d_b}$ then
\begin{equation}\label{eq:newakb}
a_{k,b}=g_{k-1,b}(U+ d_b)- g_{k-1,b}( U).
\end{equation}

If $ \sum\limits_{h=b}^{l}\left\lceil \frac{f_ky^h_{k,b}}{d_b}\right\rceil  > s_{kb }$, then  by condition $(ii)$ of Lemma \ref{lem:two_properties1} and by \eqref{eq:newakb} we have that

$$g_{k-1,b}(U- d_b(  \sum\limits_{h=b}^{l}\left\lceil \frac{f_ky^h_{k,b}}{d_b}\right\rceil - s_{kb }))+ a_{k,b}(  \sum\limits_{h=b}^{l}\left\lceil \frac{f_ky^h_{k,b}}{d_b}\right\rceil  - s_{kb })\le g_{k-1,b}(U).$$
Then by \eqref{eq:feas4b} and since $g_{k,b}(\sum\limits_{h=b}^{l}F_h)=g_{k-1,b}(U)+ a_{k,b}s_{k,b}$, we have that the inequality \eqref{eq:feasPrinc} holds. 
Finally, if $ \sum\limits_{h=b}^{l}\left\lceil \frac{f_ky^h_{k,b}}{d_b}\right\rceil  < s_{kb }$,   by condition $(i)$ of Lemma \ref{lem:two_properties1} and by \eqref{eq:newakb} it follows that
$$g_{k-1,b}(U- d_b(  \sum\limits_{h=b}^{l}\left\lceil \frac{f_ky^h_{k,b}}{d_b}\right\rceil - s_{kb }))+ a_{k,b}(  \sum\limits_{h=b}^{l}\left\lceil \frac{f_ky^h_{k,b}}{d_b}\right\rceil  - s_{kb })=$$
$$g_{k-1,b}(U+ d_b(  s_{kb }-\sum\limits_{h=b}^{l}\left\lceil \frac{f_ky^h_{k,b}}{d_b}\right\rceil ))- a_{k,b}(  s_{kb } - \sum\limits_{h=b}^{l}\left\lceil \frac{f_ky^h_{k,b}}{d_b}\right\rceil )
\le g_{k-1,b}(U).$$
Hence, by \eqref{eq:feas4b}, the inequalities \eqref{eq:feasPrinc} are valid.

In the second subcase, i.e., $\sum\limits_{h=b}^{l}F_h - \lfloor f({\cal H}_k^b)/d_b \rfloor d_b > f_k \tilde b_{k,b}$, we have
$0\le s_{k,b}\le \frac{f_k\tilde b_{k,b}}{d_b}+1$. If $0\le s_{k,b}\le \frac{f_k\tilde b_{k,b}}{d_b}$ the arguments of the fisrt subcase can be used to show the thesis. If $ s_{k,b}= \frac{f_k\tilde b_{k,b}}{d_b}+1>\frac{f_k\tilde b_{k,b}}{d_b}$, the arguments applied in the case $\sum\limits_{h=b}^{l}F_h - (\lceil f({\cal H}_k^b)/d_b \rceil d_b +d_b) \ge f_k \tilde b_{k,b}$ prove that  the inequalities \eqref{eq:feasPrinc} are valid.

$(c)$

We show now that condition $3)$ holds, i.e.,  $MO^{OO}(k,b,{\bf F})$ is contained in the faces induced by the inequalities in $I(k,b,{\bf F})$. Recall that, by induction,  condition $3)$ holds when $T'_b$ does no contain items of type $k$.
According to the Definition of $U$ and $s_{k,b}$, we can write $\sum\limits_{h=b}^{l}F_h=U+\sigma_{k,b}d_b$. 
Let   $y \in MO^{OO}(k,b,{\bf F})$ be an optimal solution and   let 
$$c= \frac{\sum\limits_{h=b}^{l}F_h - \lfloor f({\cal H}_k^b)/d_b\rfloor d_b}{d_b}.$$

Recalling that $\delta_b=d_b$, Theorem \ref{thm:itembk(hl)} and Proposition \ref{prop:M-SP-1} imply that 
$\sum\limits_{h=b}^{l}\frac{f_ky^h_{k,b}}{d_b}$ can attain one of the following values: $\min \{c^+; f_k \tilde b_{k,b}/d_b\}$ or 
$\min \{(c-1)^+; f_k \tilde b_{k,b}/d_b\}$ or $\min \{(c-2)^+; f_k \tilde b_{k,b}/d_b\}$.

In the case $c\le 0$ then  $\sum\limits_{h=b}^{l}y^h_{k,b}=0$ in every optimal solution. Hence, by induction,  for each optimal solution in $ y \in MO^{OO}(k,b,{\bf F})$ 
there exists a choice of $a_{w,q}\in \{\alpha_{wq},\beta_{wq}\}$ and $a_{w,b}\in \{\alpha_{w,b},\beta_{w,b}\}$, for $q=1,\ldots,b-1$ and $w=1,\ldots,k$, such that
$$\sum\limits_{q=1}^{b-1} \sum\limits_{w=1}^{k} \frac{a_{w,q} f_w}{d_q}  \sum\limits_{h=q}^{l}y^h_{w,q}+ \sum\limits_{w=1}^{k-1}\frac{a_{w,b} f_{w}}{d_b}  \sum\limits_{h=b}^{l}y^h_{w,b} = g_{k-1,b}(  \sum\limits_{h=b}^{l}F_h).$$
Since in this case $s_{k,b}\le 0$, then $\sigma_{k,b}= 0$ and $g_{k-1,b}(  \sum\limits_{h=b}^{l}F_h)=g_{k,b}( \sum\limits_{h=b}^{l}F_h)$. Hence, the thesis follows. 

In the case $c-2\ge f_k \tilde b_{k,b}/d_b $, then Theorem \ref{thm:itembk(hl)} and Proposition \ref{prop:M-SP-1} imply that $\sum\limits_{h=b}^{l}y^h_{k,b}= \tilde b_{k,b}$ in every optimal solution.
Hence, by induction,  for each optimal solution in $ y \in MO^{OO}(k,b,{\bf F})$ there exists an inequality such that
$$\sum\limits_{q=1}^{b-1} \sum\limits_{w=1}^{k} \frac{a_{w,q} f_w}{d_q}  \sum\limits_{h=q}^{l}y^h_{w,q}+ \sum\limits_{w=1}^{k-1}\frac{a_{w,b} f_{w}}{d_b}  \sum\limits_{h=b}^{l}y^h_{w,b} = g_{k-1,b}(  \sum\limits_{h=b}^{l}F_h-  f_k \tilde b_{k,b}).$$
In this case $s_{k,b}\ge f_k \tilde b_{k,b}/d_b$, both when $a_{k,b}=\alpha_{k,b}$ and $a_{k,b}=\beta_{k,b}$. Hence, by definition, $\sigma_{k,b}= f_k \tilde b_{k,b}/d_b$, as  $\sum\limits_{h=b}^{l}y^h_{k,b}= \tilde b_{k,b}$ in every optimal solution,
we obtain 
$$\sum\limits_{q=1}^{b-1} \sum\limits_{w=1}^{k} \frac{a_{w,q} f_w}{d_q}  \sum\limits_{h=q}^{l}y^h_{w,q}+ \sum\limits_{w=1}^{k-1}\frac{a_{w,b} f_{w}}{d_b}  \sum\limits_{h=b}^{l}y^h_{w,b} 
+ \frac{a_{k,b} f_k}{d_b}  \sum\limits_{h=b}^{l}y^h_{k,b}
= g_{k-1,b}(  \sum\limits_{h=b}^{l}F_h-  f_k \tilde b_{k,b})+a_{k,b}f_k \tilde b_{k,b}/d_b=g_{k,b}( \sum\limits_{h=b}^{l}F_h).$$ 
And, the thesis holds both if $a_{k,b}=\alpha_{k,b}$ and  $a_{k,b}=\beta_{k,b}$. 

In the case $c= f_k \tilde b_{k,b}/d_b +1$,  
two subcases can be considered: $(I)$ $ \lfloor f({\cal H}_k^b)/d_b\rfloor = \lceil f({\cal H}_k^b)/d_b\rceil $; $(II)$ $ \lfloor f({\cal H}_k^b)/d_b\rfloor \ne \lceil f({\cal H}_k^b)/d_b\rceil $.

 Theorem \ref{thm:itembk(hl)} and Proposition \ref{prop:M-SP-1} imply that $\sum\limits_{h=b}^{l}f_k y^h_{k,b}/d_b= c-1=f_k \tilde b_{k,b}/d_b$ in Case $(I)$, while 
$\sum\limits_{h=b}^{l}f_k y^h_{k,b}/d_b\in \{c-1; c-2\}$ in  Case $(II)$ (note that since w.l.o.g. $f_k \tilde b_{k,b}/d_b\ge 0$ then $c-2\ge0$, in this case), in every optimal solution. 

In Case $(I)$, by induction,  for each optimal solution in $ y \in MO^{OO}(k,b,{\bf F})$ there exists an inequality such that
$$\sum\limits_{q=1}^{b-1} \sum\limits_{w=1}^{k} \frac{a_{w,q} f_w}{d_q}  \sum\limits_{h=q}^{l}y^h_{w,q}+ \sum\limits_{w=1}^{k-1}\frac{a_{w,b} f_{w}}{d_b}  \sum\limits_{h=b}^{l}y^h_{w,b} = g_{k-1,b}(  \sum\limits_{h=b}^{l}F_h-  (c-1)d_b)= g_{k-1,b}(  \lfloor f({\cal H}_k^b)/d_b\rfloor+d_b).$$
Hence, setting $a_{k,b}=\alpha_{k,b} $ and consequently $s_{k,b}=\sigma_{k,b}=c-1$ we obtain 
$$\sum\limits_{q=1}^{b-1} \sum\limits_{w=1}^{k} \frac{a_{w,q} f_w}{d_q}  \sum\limits_{h=q}^{l}y^h_{w,q}+ \sum\limits_{w=1}^{k-1}\frac{a_{w,b} f_{w}}{d_b}  \sum\limits_{h=b}^{l}y^h_{w,b} + \frac{\alpha_{k,b} f_k}{d_b}  \sum\limits_{h=b}^{l}y^h_{k,b}=
 g_{k-1,b}(  \lfloor f({\cal H}_k^b)/d_b\rfloor+d_b)+ (c-1)\alpha_{k,b}= $$
$$g_{k-1,b}(  \lfloor f({\cal H}_k^b)/d_b\rfloor)+  (g_{k-1,b}(  \lfloor f({\cal H}_k^b)/d_b\rfloor+d_b)-g_{k-1,b}(  \lfloor f({\cal H}_k^b)/d_b\rfloor))+ (c-1)\alpha_{k,b}=$$
$$g_{k-1,b}(  \lfloor f({\cal H}_k^b)/d_b\rfloor)+ \alpha_{k,b}\sigma_{k,b}=g_{k,b}( \sum\limits_{h=b}^{l}F_h)$$
Note that, since $ \lfloor f({\cal H}_k^b)/d_b\rfloor = \lceil f({\cal H}_k^b)/d_b\rceil $, the above relations hold even if $\alpha_{k,b}$ and $\lfloor f({\cal H}_k^b)/d_b\rfloor$ are replaced by $\beta_{k,b}$ and $\lceil f({\cal H}_k^b)/d_b\rceil$, respectively.

In Case $(II)$, by induction,  for each optimal solution in $ y \in MO^{OO}(k,b,{\bf F})$ there exists an inequality such that:\\ $(II.a)$
$$\sum\limits_{q=1}^{b-1} \sum\limits_{w=1}^{k} \frac{a_{w,q} f_w}{d_q}  \sum\limits_{h=q}^{l}y^h_{w,q}+ \sum\limits_{w=1}^{k-1}\frac{a_{w,b} f_{w}}{d_b}  \sum\limits_{h=b}^{l}y^h_{w,b} = g_{k-1,b}(  \sum\limits_{h=b}^{l}F_h-  (c-1)d_b)= g_{k-1,b}(  \lfloor f({\cal H}_k^b)/d_b\rfloor+d_b)$$
if $\sum\limits_{h=b}^{l}f_k y^h_{k,b}/d_b= c-1$; and\\  $(II.b)$
$$\sum\limits_{q=1}^{b-1} \sum\limits_{w=1}^{k} \frac{a_{w,q} f_w}{d_q}  \sum\limits_{h=q}^{l}y^h_{w,q}+ \sum\limits_{w=1}^{k-1}\frac{a_{w,b} f_{w}}{d_b}  \sum\limits_{h=b}^{l}y^h_{w,b} = g_{k-1,b}(  \sum\limits_{h=b}^{l}F_h- (c-2)d_b)= g_{k-1,b}(  \lfloor f({\cal H}_k^b)/d_b\rfloor+2d_b)$$
if $\sum\limits_{h=b}^{l}f_k y^h_{k,b}/d_b= c-2$.

Note that, since $ \lfloor f({\cal H}_k^b)/d_b\rfloor \ne \lceil f({\cal H}_k^b)/d_b\rceil $, when $a_{k,b}=\beta_{k,b}$ we have 
$s_{k,b}=\sigma_{k,b}=c-1$.

In Case  $(II.a)$ we obtain 
$$\sum\limits_{q=1}^{b-1} \sum\limits_{w=1}^{k} \frac{a_{w,q} f_w}{d_q}  \sum\limits_{h=q}^{l}y^h_{w,q}+ \sum\limits_{w=1}^{k-1}\frac{a_{w,b} f_{w}}{d_b}  \sum\limits_{h=b}^{l}y^h_{w,b} + \frac{\beta_{k,b} f_k}{d_b}  \sum\limits_{h=b}^{l}y^h_{k,b}=g_{k-1,b}(  \lfloor f({\cal H}_k^b)/d_b\rfloor +d_b)+\beta_{k,b}(c-1)=$$
$$g_{k-1,b}(  \lfloor f({\cal H}_k^b)/d_b\rfloor +d_b)+\beta_{k,b}\sigma_{k,b}=g_{k,b}( \sum\limits_{h=b}^{l}F_h).$$
In Case  $(II.b)$ we have 
$$\sum\limits_{q=1}^{b-1} \sum\limits_{w=1}^{k} \frac{a_{w,q} f_w}{d_q}  \sum\limits_{h=q}^{l}y^h_{w,q}+ \sum\limits_{w=1}^{k-1}\frac{a_{w,b} f_{w}}{d_b}  \sum\limits_{h=b}^{l}y^h_{w,b} + \frac{\beta_{k,b} f_k}{d_b}  \sum\limits_{h=b}^{l}y^h_{k,b}=g_{k-1,b}(  \lfloor f({\cal H}_k^b)/d_b\rfloor +2d_b)+\beta_{k,b}(c-2)=$$
$$g_{k-1,b}(  \lfloor f({\cal H}_k^b)/d_b\rfloor +d_b) + (g_{k-1,b}(   \lfloor f({\cal H}_k^b)/d_b\rfloor +2d_b)-g_{k-1,b}(  \lfloor f({\cal H}_k^b)/d_b\rfloor +d_b))+\beta_{k,b}(c-2)=$$
$$g_{k-1,b}(  \lfloor f({\cal H}_k^b)/d_b\rfloor +d_b)+\beta_{k,b}(c-1)=g_{k-1,b}(  \lfloor f({\cal H}_k^b)/d_b\rfloor +d_b)+\beta_{k,b}\sigma_{k,b}=g_{k,b}( \sum\limits_{h=b}^{l}F_h).$$
Finally, we have the case  $1 \le c\le f_k \tilde b_{k,b}/d_b $. Two subcases are considered: $2 \le c\le f_k \tilde b_{k,b}/d_b $ and $c=1\le f_k \tilde b_{k,b}/d_b $. In the subcase $2 \le c\le f_k \tilde b_{k,b}/d_b $, 
Theorem \ref{thm:itembk(hl)} and Proposition \ref{prop:M-SP-1} imply that $\sum\limits_{h=b}^{l}f_k y^h_{k,b}/d_b\in \{c-2; c-1; c\}$ in every optimal solution. If $\sum\limits_{h=b}^{l}f_k y^h_{k,b}/d_b=c-2$ or $\sum\limits_{h=b}^{l}f_k y^h_{k,b}/d_b=c-1$ then the thesis follows by using the same arguments of Cases  $(II.a)$ and $(II.b)$.
 If $\sum\limits_{h=b}^{l}f_k y^h_{k,b}/d_b=c$,  by the induction hypothesis, for each optimal solution in $ y \in MO^{OO}(k,b,{\bf F})$ there exists a choice of $a_{w,q}\in \{\alpha_{wq},\beta_{wq}\}$ and $a_{w,b}\in \{\alpha_{w,b},\beta_{w,b}\}$, for $q=1,\ldots,b-1$ and $w=1,\ldots,k$, such that
\begin{eqnarray}\label{eq:c=1d}
\sum\limits_{q=1}^{b-1} \sum\limits_{w=1}^{k} \frac{a_{w,q} f_w}{d_q}  \sum\limits_{h=q}^{l}y^h_{w,q}+ \sum\limits_{w=1}^{k-1}\frac{a_{w,b} f_{w}}{d_b}  \sum\limits_{h=b}^{l}y^h_{w,b} =g_{k-1,b}(  \sum\limits_{h=b}^{l}F_h-  c d_b)= g_{k-1,b}(  \lfloor f({\cal H}_k^b)/d_b\rfloor).
\end{eqnarray}
Hence, we obtain 
\begin{eqnarray}\label{c=1d}
\sum\limits_{q=1}^{b-1} \sum\limits_{w=1}^{k} \frac{a_{w,q} f_w}{d_q}  \sum\limits_{h=q}^{l}y^h_{w,q}+ \sum\limits_{w=1}^{k-1}\frac{a_{w,b} f_{w}}{d_b}  \sum\limits_{h=b}^{l}y^h_{w,b} + \frac{\alpha_{k,b} f_k}{d_b}  \sum\limits_{h=b}^{l}y^h_{k,b}= \nonumber\\
 g_{k-1,b}(  \lfloor f({\cal H}_k^b)/d_b\rfloor)+  \alpha_{k,b}c=
 g_{k-1,b}(  \lfloor f({\cal H}_k^b)/d_b\rfloor) +\alpha_{k,b}\sigma_{k,b}=g_{k,b}( \sum\limits_{h=b}^{l}F_h),
\end{eqnarray}
and the thesis follows. In the subcase $c=1\le f_k \tilde b_{k,b}/d_b $, 
Theorem \ref{thm:itembk(hl)} and Proposition \ref{prop:M-SP-1} imply that $\sum\limits_{h=b}^{l}f_k y^h_{k,b}/d_b\in \{ c-1=0; c\}$ in every optimal solution. 
If $\sum\limits_{h=b}^{l}f_k y^h_{k,b}/d_b=c-1$ the thesis follows by the above Case $(I)$. If $\sum\limits_{h=b}^{l}f_k y^h_{k,b}/d_b=c$ then the thesis follows from \eqref{eq:c=1d} and \eqref{c=1d}. 

\end{proof}
{\bf Example} (continued)\\
We continue the example reported at the end of the Section \ref{sec:enum}.  
For the sake of simplicity we derive only the set of inequalities $I(5,2,(1,6,8))$ (also  in this case $I(k,b,(F_1,F_2,F_3))$ denotes the set  $I(k,b,{\bf F})$, where ${\bf F}$ has components $F_1$, $F_2$ and $F_3$) related to the solution set  
 $MO(5,2,(1,6,8))\equiv MO(4,2,(1,6,8))$. At this aim, we first consider  the sets $I(1,1,(1,F_2,F_3))$, 
$I(2,2,(1,6,F_3))$, $I(3,2,(1,6,8))$, and, finally, $I(5,2,(1,6,8))\equiv I(4,2,(1,6,8))$.

$I(1,1,(1,F_2,F_3))\equiv I(5,1,(1,F_2,F_3))$\\  By \eqref{eq:MP^{OO}(k,1,F)}, since $\tilde b_{1,1}=2$, we have that $ I(1,1,(1,F_2,F_3))$ contains the inequality  $\sum\limits_{h=1}^{4}y_{1,1}^h \le 2$ when $F_2+F_3>0$, and $\sum\limits_{h=1}^{4}y_{1,1}^h \le 1$ if $F_2+F_3=0$. Recall that, given $\hat y $ in $MO(1,1,(1,F_2,F_3))$, we have $\hat y_{1,1}^1 =1$ and $\hat y_{1,1}^2 +\hat y_{1,1}^3 =1$ if $F_2+F_3=2$, and $\hat y_{1,1}^1 =1$ and $\hat y_{1,1}^2 +\hat y_{1,1}^3 =0$ if $F_2+F_3=0$. Hence, $I(1,1,(1,F_2,F_3))$ contains valid inequalities and $MO^{OO}(1,1,(1,F_2,F_3))$ is contained in the faces induced by $I(1,1,(1,F_2,F_3))$.

\medskip
$I(2,2,(1,6,F_3))$\\ 
Computation of $\alpha_{2,2}$ and $\beta_{2,2}$: Recall that  $H_2=\emptyset$ and hence $f({\cal H}_2^2)=0$. By definition, we have
$\alpha_{2,2}=g_{1,2}(d_2)-g_{1,2}(0)=g_{1,2}(2)-g_{1,2}(0)=1$ and $\beta_{2,2}=g_{1,2}(2d_2)-g_{1,2}(d_2)=g_{1,2}(4)-g_{1,2}(2)=g_{1,1}(1+4)-g_{1,1}(1+2)=0 $. Hence, $I(2,2,(1,6,F_3))$ contains the two inequalities:
$\sum\limits_{h=1}^{3}y_{1,1}^h + \alpha_{2,2} \sum\limits_{h=2}^{3}y_{2,2}^h = \sum\limits_{h=1}^{3}y_{1,1}^h +  \sum\limits_{h=2}^{3}y_{2,2}^h\le g_{2,2}(6+F_3)$ and $\sum\limits_{h=1}^{3}y_{1,1}^h + \beta_{2,2} \sum\limits_{h=2}^{3}y_{2,2}^h =\sum\limits_{h=1}^{3}y_{1,1}^h \le g_{2,2}(6+F_3)$.

\medskip
$I(3,2,(1,6,8))$\\
Computation of $\alpha_{3,2}$ and $\beta_{3,2}$: Since $f({\cal H}_3^2)=8$, we have
   $\alpha_{3,2}=g_{2,2}(8+d_2)-g_{2,2}(8)=g_{2,2}(10)-g_{2,2}(8)$ and 
$\beta_{3,2}=g_{2,2}(8+2d_2)-g_{2,2}(8+d_2)=g_{2,2}(12)-g_{2,2}(10)$.\\
Computation of $\alpha_{3,2}$. Two cases hold: $1)$   $a_{2,2}=\alpha_{2,2}=1$; $2)$   $a_{2,2}=\beta_{2,2}=0$.\\
Case $1)$\\
 $g_{2,2}(10)$ $\rightarrow$ $s_{2,2}=10/2=5$,  $\sigma_{2,2}=4$, then $g_{2,2}(10)=g_{1,2}(2)+ \alpha_{2,2} 4=g_{1,1}(1+2)+4=6$.\\ 
 $g_{2,2}(8)$$\rightarrow$ $s_{2,2}=8/2=4$,  $\sigma_{2,2}=4$, then $g_{2,2}(8)=g_{1,2}(0)+ \alpha_{2,2} 4=g_{1,1}(1)+4=5$. Hence, $\alpha_{3,2}=1$.\\
Case $2)$\\
 $g_{2,2}(10)$ $\rightarrow$ $s_{2,2}=(10-2)/2=4$,  $\sigma_{2,2}=4$, then $g_{2,2}(10)=g_{1,2}(2)+ \beta_{2,2} 4=g_{1,1}(1+2)=2$.\\ 
 $g_{2,2}(8)$$\rightarrow$ $s_{2,2}=(8-2)/2=3$,  $\sigma_{3,2}=3$, then $g_{2,2}(8)=g_{1,2}(2)+ \beta_{2,2} 3=g_{1,1}(1+2)=2$. Hence, $\alpha_{3,2}=0$.\\
 Computation of $\beta_{3,2}$. Two cases hold: $1)$   $a_{2,2}=\alpha_{2,2}=1$; $2)$   $a_{2,2}=\beta_{2,2}=0$.\\
Case $1)$\\
 $g_{2,2}(12)$ $\rightarrow$ $s_{2,2}=12/2=6$,  $\sigma_{2,2}=4$, then $g_{2,2}(12)=g_{1,2}(4)+ \alpha_{2,2} 4=g_{1,1}(1+4)+4=6$.\\ 
 $g_{2,2}(10)$$\rightarrow$ $s_{2,2}=10/2=5$,  $\sigma_{2,2}=4$, then $g_{2,2}(10)=g_{1,2}(2)+ \alpha_{2,2} 4=g_{1,1}(1+2)+4=6$. Hence, $\beta_{3,2}=0$.\\
Case $2)$\\
 $g_{2,2}(12)$ $\rightarrow$ $s_{2,2}=(12-2)/2=5$,  $\sigma_{2,2}=4$, then $g_{2,2}(12)=g_{1,2}(4)+ \beta_{2,2} 4=g_{1,1}(1+4)=2$.\\ 
 $g_{2,2}(10)$$\rightarrow$ $s_{2,2}=(10-2)/2=4$,  $\sigma_{2,2}=4$, then $g_{2,2}(10)=g_{1,2}(2)+ \beta_{2,2} 4=g_{1,1}(1+2)=2$. Hence, $\beta_{3,2}=0$.\\
 The inequalities  in $I(3,2,(1,6,8))$ will be given later.
 
 \medskip
 $I(4,2,(1,6,8))\equiv I(5,2,(1,6,8))$\\
 Computation of $\alpha_{4,2}$ and $\beta_{4,2}$: Since $f({\cal H}_4^2)=24$ we have
   $\alpha_{4,2}=g_{3,2}(24+d_2)-g_{3,2}(24)=g_{3,2}(26)-g_{3,2}(24)$ and 
$\beta_{4,2}=g_{3,2}(24+2d_2)-g_{3,2}(24+d_2)=g_{3,2}(28)-g_{3,2}(26)$.\\
Computation of $\alpha_{4,2}$. Two cases hold: $1)$   $a_{3,2}=\alpha_{3,2}$; $2)$   $a_{3,2}=\beta_{3,2}=0$.\\
Case $1)$\\  
$g_{3,2}(26)$$\rightarrow$  $s_{3,2}=(26-8)/2=9$,  $\sigma_{3,2}=8$, then $g_{3,2}(26)=g_{2,2}(10)+ \alpha_{3,2} 8$. 
By Cases 1 and 2 of  $\alpha_{3,2}$, we have that $g_{2,2}(10)=6$ when  $a_{2,2}=\alpha_{2,2}=1$, and 
 $g_{2,2}(10)=2$ when  $a_{2,2}=\beta_{2,2}=0$. Hence, $g_{3,2}(26)=6+ \alpha_{3,2} 8$ when  $a_{2,2}=\alpha_{2,2}=1$, and $g_{3,2}(26)=2+ \alpha_{3,2} 8$ when  $a_{2,2}=\beta_{2,2}=0$.\\
$g_{3,2}(24)$$\rightarrow$  $s_{3,2}=(24-8)/2=8$,  $\sigma_{3,2}=8$, then $g_{3,2}(24)=g_{2,2}(8)+ \alpha_{3,2} 8$. 
By Cases 1 and 2 of  $\alpha_{3,2}$, we have that $g_{2,2}(8)=5$ when  $a_{2,2}=\alpha_{2,2}=1$, and 
 $g_{2,2}(8)=2$ when  $a_{2,2}=\beta_{2,2}=0$. Hence, $g_{3,2}(24)=5+ \alpha_{3,2} 8$ when  $a_{2,2}=\alpha_{2,2}=1$, and $g_{3,2}(24)=2+ \alpha_{3,2} 8$ when  $a_{2,2}=\beta_{2,2}=0$.\\
Hence, $\alpha_{4,2}=6+ \alpha_{3,2} 8-(5+ \alpha_{3,2} 8)=1$ when  $a_{2,2}=\alpha_{2,2}=1$, and 
$\alpha_{4,2}= 2+ \alpha_{3,2} 8- (2+ \alpha_{3,2} 8)=0$ when  $a_{2,2}=\beta_{2,2}=0$.\\
Case $2)$\\ 
$g_{3,2}(26)$$\rightarrow$  $s_{3,2}=(26-10)/2=8$,  $\sigma_{3,2}=8$, then $g_{3,2}(26)=g_{2,2}(10)+ \beta_{3,2} 8=g_{2,2}(10)$. 
By Cases 1 and 2 of  $\beta_{3,2}$, we have that $g_{2,2}(10)=6$ when  $a_{2,2}=\alpha_{2,2}=1$, and 
 $g_{2,2}(10)=2$ when  $a_{2,2}=\beta_{2,2}=0$. Hence, $g_{3,2}(26)=6$ when  $a_{2,2}=\alpha_{2,2}=1$, and $g_{3,2}(26)=2$ when  $a_{2,2}=\beta_{2,2}=0$.\\
$g_{3,2}(24)$$\rightarrow$  $s_{3,2}=(24-10)/2=7$,  $\sigma_{3,2}=7$, then $g_{3,2}(24)=g_{2,2}(10)+ \beta_{3,2} 7=g_{2,2}(10)$. 
Again, by Cases 1 and 2 of  $\beta_{3,2}$, we have that $g_{2,2}(10)=6$ when  $a_{2,2}=\alpha_{2,2}=1$, and 
 $g_{2,2}(10)=2$ when  $a_{2,2}=\beta_{2,2}=0$. Hence, $g_{3,2}(24)=6$ when  $a_{2,2}=\alpha_{2,2}=1$, and $g_{3,2}(24)=2$ when  $a_{2,2}=\beta_{2,2}=0$.\\
Hence, $\alpha_{4,2}=0$ both when  $a_{2,2}=\alpha_{2,2}=1$, and 
when $a_{2,2}=\beta_{2,2}=0$.

Computation of $\beta_{4,2}$. Two cases hold: $1)$   $a_{3,2}=\alpha_{3,2}$; $2)$   $a_{3,2}=\beta_{3,2}=0$.\\
Case $1)$\\  
$g_{3,2}(28)$$\rightarrow$  $s_{3,2}=(28-8)/2=10$,  $\sigma_{3,2}=8$, then $g_{3,2}(28)=g_{2,2}(12)+ \alpha_{3,2} 8$. 
By Cases 1 and 2 of  $\alpha_{3,2}$, we have that $g_{2,2}(12)=6$ when  $a_{2,2}=\alpha_{2,2}=1$, and 
 $g_{2,2}(12)=2$ when  $a_{2,2}=\beta_{2,2}=0$. Hence, $g_{3,2}(28)=6+ \alpha_{3,2} 8$ when  $a_{2,2}=\alpha_{2,2}=1$, and $g_{3,2}(28)=2+ \alpha_{3,2} 8$ when  $a_{2,2}=\beta_{2,2}=0$.\\
$g_{3,2}(26)$$\rightarrow$  $s_{3,2}=(26-8)/2=9$,  $\sigma_{3,2}=8$, then $g_{3,2}(26)=g_{2,2}(10)+ \alpha_{3,2} 8$. 
By Cases 1 and 2 of  $\alpha_{3,2}$, we have that $g_{2,2}(10)=6$ when  $a_{2,2}=\alpha_{2,2}=1$, and 
 $g_{2,2}(10)=2$ when  $a_{2,2}=\beta_{2,2}=0$. Hence, $g_{3,2}(26)=6+ \alpha_{3,2} 8$ when  $a_{2,2}=\alpha_{2,2}=1$, and $g_{3,2}(26)=2+ \alpha_{3,2} 8$ when  $a_{2,2}=\beta_{2,2}=0$.\\
Hence, $\beta_{4,2}=0$ both when  $a_{2,2}=\alpha_{2,2}=1$, and 
when  $a_{2,2}=\beta_{2,2}=0$.\\
Case $2)$\\  
$g_{3,2}(28)$$\rightarrow$  $s_{3,2}=(28-10)/2=9$,  $\sigma_{3,2}=8$, then $g_{3,2}(28)=g_{2,2}(12)+ \beta_{3,2} 8=g_{2,2}(12)$. 
By Cases 1 and 2 of  $\beta_{3,2}$, we have that $g_{2,2}(12)=6$ when  $a_{2,2}=\alpha_{2,2}=1$, and 
 $g_{2,2}(12)=2$ when  $a_{2,2}=\beta_{2,2}=0$. Hence, $g_{3,2}(28)=6$ when  $a_{2,2}=\alpha_{2,2}=1$, and $g_{3,2}(28)=2$ when  $a_{2,2}=\beta_{2,2}=0$.\\
 $g_{3,2}(26)$$\rightarrow$  $s_{3,2}=(26-10)/2=8$,  $\sigma_{3,2}=8$, then $g_{3,2}(26)=g_{2,2}(10)+ \beta_{3,2} 8=g_{2,2}(10)$. 
By Cases 1 and 2 of  $\beta_{3,2}$, we have that $g_{2,2}(10)=6$ when  $a_{2,2}=\alpha_{2,2}=1$, and 
 $g_{2,2}(10)=2$ when  $a_{2,2}=\beta_{2,2}=0$. Hence, $g_{3,2}(26)=6$ when  $a_{2,2}=\alpha_{2,2}=1$, and $g_{3,2}(26)=2$ when  $a_{2,2}=\beta_{2,2}=0$.\\
Hence, also in this case, $\beta_{4,2}=0$ both when  $a_{2,2}=\alpha_{2,2}=1$, and 
when  $a_{2,2}=\beta_{2,2}=0$.

Now, the inequalities in $I(3,2,(1,6,8))$ and $I(4,2,(1,6,8))$ are derived.
By the above discussion, we have that  $I(3,2,(1,6,8))$ contains the following inequalities:
 \begin{eqnarray}\label{I(3,2,(1,6,8))}
\sum\limits_{h=1}^{3} y^h_{1,1} + \alpha_{2,2}\sum\limits_{h=2}^{3} y^h_{2,2} + \alpha_{3,2}\sum\limits_{h=2}^{3} y^h_{3,2}= \sum\limits_{h=1}^{3} y^h_{1,1} + \sum\limits_{h=2}^{3} y^h_{2,2} + \sum\limits_{h=2}^{3} y^h_{3,2} \le g_{3,2}(14)=8\label{in1}\\
\sum\limits_{h=1}^{3} y^h_{1,1} + \beta_{2,2}\sum\limits_{h=2}^{3} y^h_{2,2} + \alpha_{3,2}\sum\limits_{h=2}^{3} y^h_{3,2}= \sum\limits_{h=1}^{3} y^h_{1,1}  + \sum\limits_{h=2}^{3} y^h_{3,2} \le g_{3,2}(14)=5\label{in2}\\
\sum\limits_{h=1}^{3} y^h_{1,1} + \alpha_{2,2}\sum\limits_{h=2}^{3} y^h_{2,2} + \beta_{3,2}\sum\limits_{h=2}^{3} y^h_{3,2}= \sum\limits_{h=1}^{3} y^h_{1,1} + \sum\limits_{h=2}^{3} y^h_{2,2} \le g_{3,2}(14)=6\label{in3}\\
\sum\limits_{h=1}^{3} y^h_{1,1} + \beta_{2,2}\sum\limits_{h=2}^{3} y^h_{2,2} + \beta_{3,2}\sum\limits_{h=2}^{3} y^h_{3,2}= \sum\limits_{h=1}^{3} y^h_{1,1}   \le g_{3,2}(14)=2\label{in4}
\end{eqnarray}
Observe that, the right-hand side of the inequalities are different, since they are computed by \eqref{formula g_k(F)} and \eqref{formula g_k(F)k>1} and depend on the variable coefficients of the left-hand sides. For instance, in \eqref{in1}, 
 we have $s_{3,2}=(14-8)/2=3$, then $\sigma_{3,2}=3$ and  $g_{3,2}(14)=g_{2,2}(8)+ 3 \alpha_{3,2}=g_{2,2}(8)+ 3$.  
 Then, $s_{2,2}=(8)/2=4$, $\sigma_{2,2}=4$, and, hence, the right-hand side is $g_{3,2}(14) =g_{1,2}(0)+ 4\alpha_{2,2}+ 3 \alpha_{3,2}=g_{1,1}(1)+ 4+3=8$. While, in \eqref{in2}, 
 we have $s_{3,2}=(14-8)/2=3$, then $\sigma_{3,2}=3$ and  $g_{3,2}(14)=g_{2,2}(8)+ 3 \alpha_{3,2}$. Then, $s_{2,2}=(8-2)/2=3$, $\sigma_{2,2}=3$, and $g_{3,2}(14) =g_{1,2}(2)+ 3\beta_{2,2}+ 3 \alpha_{3,2}=g_{1,1}(1)+ 3=5$ (recall that 
 $\beta_{2,2}=0$).

According to the values of $\alpha_{4,2}$ and $\beta_{4,2}$ computed above, $I(4,2,(1,6,8))$ contains the following inequalities:
 \begin{eqnarray}\label{I(3,2,(1,6,8))}
\sum\limits_{h=1}^{3} y^h_{1,1} + \alpha_{2,2}\sum\limits_{h=2}^{3} y^h_{2,2} + \alpha_{3,2}\sum\limits_{h=2}^{3} y^h_{3,2}+ \alpha_{4,2}\sum\limits_{h=2}^{3} y^h_{4,2}&=& \sum\limits_{h=1}^{3} y^h_{1,1} + \sum\limits_{h=2}^{3} y^h_{2,2} + \sum\limits_{h=2}^{3} y^h_{3,2}+ \sum\limits_{h=2}^{3} y^h_{4,2} \le \nonumber \\ g_{4,2}(14)=8\label{in1b}\\
\sum\limits_{h=1}^{3} y^h_{1,1} + \beta_{2,2}\sum\limits_{h=2}^{3} y^h_{2,2} + \alpha_{3,2}\sum\limits_{h=2}^{3} y^h_{3,2}+ \alpha_{4,2}\sum\limits_{h=2}^{3} y^h_{4,2}&=& \sum\limits_{h=1}^{3} y^h_{1,1}  + \sum\limits_{h=2}^{3} y^h_{3,2} \le g_{4,2}(14)=5\label{in2b}\\
\sum\limits_{h=1}^{3} y^h_{1,1} + \alpha_{2,2}\sum\limits_{h=2}^{3} y^h_{2,2} + \beta_{3,2}\sum\limits_{h=2}^{3} y^h_{3,2}+ \alpha_{4,2}\sum\limits_{h=2}^{3} y^h_{4,2}&=& \sum\limits_{h=1}^{3} y^h_{1,1} + \sum\limits_{h=2}^{3} y^h_{2,2} \le g_{4,2}(14)=6\label{in3b}\\
\sum\limits_{h=1}^{3} y^h_{1,1} + \beta_{2,2}\sum\limits_{h=2}^{3} y^h_{2,2} + \beta_{3,2}\sum\limits_{h=2}^{3} y^h_{3,2}+ \alpha_{4,2}\sum\limits_{h=2}^{3} y^h_{4,2}&=& \sum\limits_{h=1}^{3} y^h_{1,1}   \le g_{4,2}(14)=2\label{in4b}\\
\sum\limits_{h=1}^{3} y^h_{1,1} + \alpha_{2,2}\sum\limits_{h=2}^{3} y^h_{2,2} + \alpha_{3,2}\sum\limits_{h=2}^{3} y^h_{3,2}+ \beta_{4,2}\sum\limits_{h=2}^{3} y^h_{4,2}&=& \sum\limits_{h=1}^{3} y^h_{1,1} + \sum\limits_{h=2}^{3} y^h_{2,2} + \sum\limits_{h=2}^{3} y^h_{3,2} \le g_{4,2}(14)=8\label{in5b}\\
\sum\limits_{h=1}^{3} y^h_{1,1} + \beta_{2,2}\sum\limits_{h=2}^{3} y^h_{2,2} + \alpha_{3,2}\sum\limits_{h=2}^{3} y^h_{3,2}+ \beta_{4,2}\sum\limits_{h=2}^{3} y^h_{4,2}&=& \sum\limits_{h=1}^{3} y^h_{1,1}  + \sum\limits_{h=2}^{3} y^h_{3,2} \le g_{4,2}(14)=5\label{in6b}\\
\sum\limits_{h=1}^{3} y^h_{1,1} + \alpha_{2,2}\sum\limits_{h=2}^{3} y^h_{2,2} + \beta_{3,2}\sum\limits_{h=2}^{3} y^h_{3,2}+ \beta_{4,2}\sum\limits_{h=2}^{3} y^h_{4,2}&=& \sum\limits_{h=1}^{3} y^h_{1,1} + \sum\limits_{h=2}^{3} y^h_{2,2} \le g_{4,2}(14)=6\label{in7b}\\
\sum\limits_{h=1}^{3} y^h_{1,1} + \beta_{2,2}\sum\limits_{h=2}^{3} y^h_{2,2} + \beta_{3,2}\sum\limits_{h=2}^{3} y^h_{3,2}+ \beta_{4,2}\sum\limits_{h=2}^{3} y^h_{4,2}&=& \sum\limits_{h=1}^{3} y^h_{1,1}   \le g_{4,2}(14)=2.\label{in8b}
\end{eqnarray}
Note that, the inequalities \eqref{in6b}--\eqref{in8b} are the same of \eqref{in2}--\eqref{in4} and that, since $f({\cal H}_4^2)=24>14$, we have 
$g_{4,2}(14)=g_{3,2}(14)$ in all above inequalities. Recalling that, in  the example reported at the end of Section \ref{sec:enum}, $MO(4,2,(1,6,8))$ contains  points $\hat y$ with the following characteristics:
\begin{enumerate}
\item $\sum\limits_{h=1}^{3} \hat y^h_{1,1}=2$; $\hat y^2_{2,2}+\hat y^3_{2,2}=4$; $\hat y^3_{3,2}=2$; $\hat y^3_{4,2}=0$;
\item $\sum\limits_{h=1}^{3} \hat y^h_{1,1}=2$; $\hat y^2_{2,2}+\hat y^3_{2,2}=3$; $\hat y^3_{3,2}=3$; $\hat y^3_{4,2}=0$;
\item $\sum\limits_{h=1}^{3} \hat y^h_{1,1}=1$; $\hat y^2_{2,2}+\hat y^3_{2,2}=4$; $\hat y^3_{3,2}=3$; $\hat y^3_{4,2}=0$.
\end{enumerate}
It is easy to verify that the inequalities in $I(4,2,(1,6,8))$ are a valid for $MO(4,2,(1,6,8))$, and that each point in $MO(4,2,(1,6,8))$ is contained by (at least) a face induced by the inequalities in $I(4,2,(1,6,8))$.


\subsection{Description of  $P^{OO}$}
In this section, a description of  $P^{OO}$ is provided as a system of inequalities. Similar arguments to those used in \cite{Pochet98} are employed. Let $W \subseteq N$ be a subset  of items, let $B=\{B_1,\ldots, B_t\}$ be the maximal block 
partition of $W$ into blocks.
Let  $r=argmin_{j \in B_q}\{s_j\}$, and let $f'_q=s_r$ and $p_q=p_r$ be the weight and the profit of the block $q$, respectively. Let $\tilde b_q=\sum_{j \in B_q}\frac{s_jb_j}{f'_q}$
be the multiplicity of block $B_q$. Suppose that $f'_1 \le f'_2\le \ldots \le f'_t$ and let $f_q=f'_q/f'_1$, for $q=1,\ldots,t$.

Let ${\bf F}$ be a vector of $l$ components such that $F_h=\sum\limits_{i=1}^m r_i^h$, for $h=1,\ldots,l$, and let $MP^{OO,W}_t({\bf F}/f_1')$ be the set containing all the OPT and ordered solutions of M-$SP$ defined on the maximal  block partition $B$ of $W$, weights $f_1,\ldots,f_t$, multiplicities $\tilde b_1,\ldots,\tilde b_t$, knapsack part capacities $F_h=\sum\limits_{i=1}^m r_i^h$, for $h=1,\ldots,l$, and objective coefficients $p_1,\ldots,p_t$.
Hence, by \eqref{prob:M-SP-cond1} and \eqref{prob:M-SP-cond2}, $MP^{OO,W}_t({\bf F}/f_1')$ is defined as
{\small
\begin{equation}\label{poly:covM-SP}
MP^{OO,W}_t({\bf F}/f_1')=conv\left(y \in(\mathbb{Z}^+\cup\{0\})^{t\times l \times l}|
\begin{array}{rclr}
\sum\limits_{w=1}^{t}\sum\limits_{q=1}^{h }\left\lceil \frac{f_w y_{wq}^h}{d_q}\right\rceil  d_q &\leq& F_h/f_1' \;\;\;\text{ for } h=1,\ldots, l \\
 \sum\limits_{h=q}^{l} \left\lceil \frac{f_w y_{wq}^h}{d_q}\right\rceil &\leq&  \frac{f_w \tilde b_{wq}}{d_q} \;\;\;\text{ for } w =1,\ldots, t \text{ and } q =1,\ldots, l\\
 y &&\text{is an OPT and an ordered solution}
 \end{array}
\right)
\end{equation}
}
Observe that, since $MP^{OO,W}_t({\bf F}/f_1')$ contains OPT and ordered solutions, then $MP^{OO,W}_t({\bf F}/f_1')$ depends on the objective function coefficients.
Let $I_W(t,l,{\bf F}/f_1')$ be the set of inequalities, introduced in the previous Section,  satisfying conditions $1)$, $2)$ and $3)$ for $MP^{OO,W}_t({\bf F}/f_1')$, 
written as 
\begin{equation}\label{eq:M-SPx}
 \sum\limits_{w=1}^{t} \sum\limits_{q=1}^{l} 
  \frac{a_{w,q} f_w}{d_q}  \sum\limits_{h=q}^{l}y^h_{w,q}
  \le g_{t,l}(F_l/f'_1).
\end{equation}
and let ${\cal I}(W)$ be the set containing  inequalities of the form 
\begin{equation}\label{eq:MSkPx}
 \sum\limits_{w=1}^{t} \sum\limits_{q=1}^{l} 
   \frac{a_{w,q} f_w}{d_q}
 \sum\limits_{j \in B_w:\\ g(j)=q} 
 \left( \frac{s_j}{f_jf'_1}\right)
 \sum\limits_{h: h\ge q} \sum\limits_{i=1}^m x_{ij}^h=
 \sum\limits_{w=1}^{t} \sum\limits_{q=1}^{l} 
  a_{w,q} f_w 
 \sum\limits_{j \in B_w:\\ g(j)=q} 
 \left( \frac{1}{f'_j}\right)
 \sum\limits_{h=q}^l \sum\limits_{i=1}^m x_{ij}^h
  \le g_{t,l}(F_l/f'_1).
\end{equation}

\begin{theorem}\label{thm:final}
The polytope $P^{OO}$ is described by the following system\\
$ {\cal I}(W) \text{for all } W\subseteq N$\\
$x \in (\mathbb{R}^+\cup\{0\})^{n\times m \times l}.$ 

\end{theorem}
\begin{proof}
By Theorem \ref{thm:polyMOPMain},  the inequalities in $I_W(t,l,{\bf F}/f'_1)$ are valid for all points in  $MP^{OO,W}_t({\bf F}/f'_1)$. Then,  by the correspondance between feasible solutions of SMKP and M-SP (see Section \ref{sec:corr}) and by \eqref{eq:valid2}, we have that the inequalities in  ${\cal I}(W)$
are valid for the polytope

{\small
\begin{equation}
P^{OO}(W)=conv\left(
\begin{array}{rclr}
 \sum\limits_{j\in W: g(j) \le h} s_j/f_1' x_{ij}^h &\leq& r_i^h/f_1' \text{ for } h=1,\ldots, l \text{ and } i=1,\ldots, m\\
 \sum\limits_{i=1}^{m} \sum\limits_{h: h\ge g(j)}x_{ij}^h &\leq& b_j \text{ for } j\in W  \\
 x &&\text{is an OPT and an ordered solution}\\
x& \in&(\mathbb{Z}^+\cup\{0\})^{n\times m \times l}
\end{array}
\right)
\end{equation}
}
i.e., $P^{OO}$ restricted to the item set $W$.

In the following, an inequality that is satisfied as an equation by all optimal solution of MSKP, formulated as in \eqref{newform}, is built.
As in \cite{Pochet98},  if an objective function coefficient $v_j<0$ exists for some $j \in N$, then $x_{ij}^h=0$ for all $h$ and $i$ in any optimal solution. If  $v_j\ge 0$ for all $j \in N$ we set $W=\{j \in N: v_j>0\}$. Let  $B=\{B_1,\ldots, B_t\}$ be the maximal block partition of $W$, and let us consider the transformation into M-$SP$ when only items in $W$ are considered. By the arguments used in Section \ref{sec:corr}  on the equivalence of optimal solutions and by \eqref{eq:valid2}, inequalities in ${\cal I}(W)$ are satisfied at equality by all the optimal solution of problem $\max \{v^Tx: x \in P^{OO}(W)\}.$
As $v_j=0$ for all $j \in N \setminus W$, a solution $x$ for MSKP defined on the set $N$ can be optimal only if the components of $x$ corresponding  to itsms in $W$ satisfy the inequalities in ${\cal I}(W)$ at equality.
\end{proof}


\newpage
 \section{Appendix}
\subsection{Proof of Lemma \ref{lem:opt_aux}}

By contradiction, assume that the lemma does not hold. By the optimality of A-OPT, one has $S(\bar x)=S(x)$, i.e., A-OPT is able to assign all items assigned in $x$. Let $x_{O}$ be the solution found by A-OPT applied to the whole item set $N$, let $q\le h$ be the first index such that $S(\bar x^{q})\not= S(x_{O}^{q})$. \\ 
Let us consider the execution of A-OPT applied to the restricted set of items $S(\bar x)$.
At  iteration $q$, after the grouping procedure,
let $\bar L$ be the  list (ordered in non-increasing order of value) of, eventually grouped, items of sizes $d_q$. Let us consider now,  the execution of A-OPT applied to the whole item set.
At  iteration $q$, after the grouping procedure, let $L_{O}$ be the  list (ordered in non-increasing order of value) of, eventually grouped, items of sizes $d_q$.
 In what follows, by abusing notation, given an item $t$ generated by the grouping procedure of A-OPT,  we denote by $S(t)$ the set of items that have been used to generate it. If $t$ corresponds to a single item, let $S(t)=\{t\}$. To simplify the notation,   $v(t)$ is used in place of $v(S(t))$.
  According to these two lists, let $j$ and $j_{O}$ be the first two assigned (eventually grouped) items by A-OPT in  $\bar L$ and  $L_{O}$, respectively, such that $S(j)\not=S(j_{O})$. 
Observe that, by definition, before the assignment of $j$ and $j_{O}$, the set of items assigned by A-OPT (during the production of $\bar x$ and $x_O$, respectively) is the same. By the above observation and by definition of $\bar L$ and $L_{O}$, it follows that $v(j)\le v(j_{O})$. 
In the following, we show that $v(j)= v(j_{O})$. 
By contradiction, suppose that $v(j)< v(j_{O})$. 
Let $\{j,j^1,j^2,\ldots,j^z\}$ and $\{j_{O},j_{O}^1,j_{O}^2,\ldots,j_{O}^y\}$ be
the last parts of the ordered lists $\bar L$ and $L_{O}$, resulting after the assignment of $j$ and $j_{O}$, respectively.
By definition of $j$ and $j_{O}$, observe that $S(j)\cup S(j^1)\cup S(j^2) \cup \ldots S(j^z)$ is a subset of $S(j_{O})\cup S(j_{O}^1)\cup S(j_{O}^2) \cup \ldots S(j_{O}^y)$ (since $L_{O}$ is the list obtained by applying A-OPT on the whole item set $N$).
Moreover, since we are considering the iteration $q$ of A-OPT, $f(S(t))\le d_q$ for all $t \in \{j,j^1,j^2,\ldots,j^z\}\cup \{j_{O},j_{O}^1,j_{O}^2,\ldots,j_{O}^y\}$. Recall that, a size equal to $d_q$ is assigned by A-OPT to each  item in $\{j,j^1,j^2,\ldots,j^z\}\cup \{j_{O},j_{O}^1,j_{O}^2,\ldots,j_{O}^y\}$.
Let $\{j,t^1,t^2,\ldots,t^g\}$ be the set containing all the  items in $\{j,j^1,j^2,\ldots,j^z\}$ that are assigned in $\bar x$. 
Let $A=(S(j)\cup S(t^1)\cup S(t^2) \cup \ldots S(t^g))\cap S(j_{O})$. Two cases are considered:  $A=\emptyset$ and  $A\not=\emptyset$. If $A=\emptyset$ then it is feasible to replace  the item $j$ with $j_{O}$ in $\bar x$, obtaining a  solution $\bar x'$ such that $v(\bar x')>v(\bar x)$. A contradiction of the optimality of $\bar x$.\\
Suppose that $A\not=\emptyset$. Note that, since $f(S(j_{O}))\le d_q$,  then $f(A)\le d_q$, too. In the following, we show that the items in $S(j)\cup S(t^1)\cup S(t^2) \cup \ldots S(t^g)$  
can be rearranged into $g+1$ chunks (i.e., subsets) each of size at most $d_q$, in such a way that $A$ is contained in exactly one chunk. 
First observe that, by definition, the grouping procedure can not arrange  the items in $S(j)\cup S(t^1)\cup S(t^2) \cup \ldots S(t^g)$ in less than $g+1$ chunks of size $d_q$.
By Proposition \ref{prop:chunks}, the items in $B=\{(S(j)\cup S(t^1)\cup S(t^2) \cup \ldots S(t^g))\setminus A\}$ can be rearranged in $\left\lceil \frac {f(B)}{d_q} \right\rceil -1$
chunks of total size $d_q$ and one chunk, say $C$, of total size at most $d_q$. Hence, $f(A\cup C)\le  2d_q$. Now, if $f(A\cup C)>d_q$, then the items in the sets $C$ and $A$ can be arranged into two chunks, while, if  $f(A\cup C)\le d_q$, the items in $A\cup C$ are arranged in a single chunk. By Proposition \ref{prop:chunks} and the optimality of A-OPT, it follows that the number of chunks obtained so far is exactly $g+1$. (Otherwise, the production of more than $g+1$ chunks by the grouping procedure could lead to a not optimal solution, contradicting the optimality of A-OPT.)
Let $\{ \bar j, \bar t^1,\ldots, \bar t^2,\ldots \bar t^g\}$ be the new $g+1$ chunks, obtained so far. Observe that, $v(t) \le v(j)< v(j_{O})$ for each $t$ in $\{\bar j, \bar t^1,\ldots, \bar t^2,\ldots \bar t^g\}$.
 Without loss of generality, let $A \subseteq \bar j$.  Hence, $v(\bar j)< v(j_{O})$. Since     A-OPT has assigned a size $d_q$ to each of the (grouped) items  in $\{j,t^1,t^2,\ldots,t^g\}$ before assigning them  in $\bar x$, it is feasible replace  items   $j,t^1,t^2,\ldots,t^g$ with items 
$\bar j, \bar t^1,\ldots, \bar t^2,\ldots \bar t^g$. Let $\tilde x$ be the new feasible solution obtained after this replacement. Note that, $v(\bar x)=v(\tilde x)$. Moreover, observe that, by replacing in $\tilde x$ the item $\bar j$ with $j_{O}$,  a  new feasible solution is obtained with objective function value strictly greater than  $v(\bar x)=v(\tilde x)$. A contradiction of the optimality of $\bar x$.\\
By the above discussion, it follows that $v(j)= v(j_{O})$. Recall that, the items not yet assigned at the beginning of iteration $q$ by A-OPT, when A-OPT  is applied to $S(\bar x)$, are a subset of the items not yet assigned at the beginning of iteration $q$ by A-OPT, when A-OPT  applied to the whole item set.
As a consequence, when A-OPT is applied 
  to the whole item set, during the grouping procedure of iteration $q$, an ordering of the items not yet assigned exists, such that, after the grouping procedure, $j$ appears in place of $j_{O}$ in the list $L_{O}$. And, in this case, A-OPT will assign $j$ instead of $j_{O}$ in $x_{O}$.\qed

\end{document}